\documentclass[11pt]{amsart}

\usepackage{amsmath,amssymb,amsthm,mathtools,mathrsfs}
\usepackage{enumitem}
\usepackage{microtype}
\usepackage{xcolor}
\usepackage{geometry}
\usepackage[colorlinks=true,linkcolor=blue!55!black,citecolor=blue!55!black,urlcolor=blue!55!black,hypertexnames=false]{hyperref}
\hypersetup{
 pdftitle={Commutators with Two Matrix Weights: Boundedness, Compactness and  Riesz-Transform Schatten Classes},
 pdfauthor={Ji Li; Chong-Wei Liang; Chun-Yen Shen and Brett D. Wick},
 pdfsubject={General Calderon--Zygmund commutators and Schatten classes of Riesz commutators},
 pdfkeywords={matrix weight, commutator, Calderon--Zygmund operator, compactness, Schatten class, Riesz transform}
}

\newtheorem{theorem}{Theorem}[section]
\newtheorem{proposition}[theorem]{Proposition}
\newtheorem{lemma}[theorem]{Lemma}
\newtheorem{corollary}[theorem]{Corollary}
\newtheorem{remark}[theorem]{Remark}
\theoremstyle{definition}
\newtheorem{definition}[theorem]{Definition}
\numberwithin{equation}{section}

\newcommand{\R}{\mathbb R}
\newcommand{\C}{\mathbb C}
\newcommand{\D}{\mathcal D}
\newcommand{\BMO}{\mathrm{BMO}}
\newcommand{\VMO}{\mathrm{VMO}}
\newcommand{\dist}{\operatorname{dist}}
\newcommand{\ess}{\mathrm e}
\newcommand{\avg}{\mathop{\rlap{\raisebox{0.15ex}{--}}\!\int}}
\newcommand{\one}{\mathbf 1}
\newcommand{\op}{\mathrm{op}}
\newcommand{\HS}{\mathrm{HS}}
\newcommand{\Sch}{\mathcal S}
\newcommand{\Lag}{\mathcal L}
\newcommand{\Sig}{\mathrm{Sig}_d}
\title[Commutators with two matrix weights]
{Commutators with Two Matrix Weights}
\author[J. Li]{Ji Li}
\author[C.W. Liang]{Chong-Wei Liang}
\author[C.Y. Shen]{Chun-Yen Shen}
\author[B. D. Wick]{Brett D. Wick}
\date{}

\begin{document}

\begin{abstract}
Let $U,V$ be matrix $\mathcal A_p$ weights, $1<p<\infty$, and
let $B$ be a matrix-valued function.  For Calder\'on--Zygmund
operators $T$ with scalar standard kernels, we characterize
boundedness and compactness of
$[M_B,T\otimes I_m]:L^p(U)\to L^p(V)$ by two-matrix BMO and
vanishing oscillation conditions.  The converse statements require
non-degeneracy condition.

For matrix \(\mathcal A_2\) weights \(U,V\) and a fixed Riesz
transform \(R_j\) on \(\mathbb R^d\), \(d\geq2\), we characterize
the Schatten class membership of
\(C_{B,R_j}:L^2(U)\to L^2(V)\).
We characterize $\mathcal S^r$ membership for $r>d$
and $\mathcal S^{d,\infty}$ membership by $\ell^r$ and weak
$\ell^d$ conditions on the two Hilbert--Schmidt cube oscillations.
For $0<r\leq d$, $\mathcal S^r$ membership is equivalent to
$B$ being constant almost everywhere.  The weak endpoint condition
implies a weighted first-order Sobolev estimate for $C^1$ symbols,
with a converse for constant or uniformly elliptic weights.
\end{abstract}

\maketitle

\section*{Introduction and the main results}

We study boundedness, compactness and Schatten class membership of
commutators with two matrix weights and a matrix-valued symbol.  In
the scalar unweighted setting, Coifman, Rochberg and Weiss characterized
boundedness of commutators with Riesz transforms by BMO \cite{CRW}, and Uchiyama proved the
corresponding compactness theorem \cite{Uchiyama}. 
Janson and Wolff \cite{JansonWolff}, Rochberg and Semmes
\cite{RS} further developed the Schatten class theory of these commutators.
  Bloom introduced a
two-weight BMO condition for the Hilbert transform \cite{Bloom}.
Holmes, Lacey and Wick extended the two-weight theory to higher
dimensions \cite{HolmesLaceyWick}.  For two scalar weights, Lacey and
Li characterized compactness of Riesz-transform commutators
\cite{LaceyLi}.  Schatten class results in this setting were obtained
for the Hilbert transform in \cite{LLW} and for Riesz transforms in
\cite{LLWW}.
We also recall some very recent closely related unweighted or one-weight results:
Wei and Zhang studied Schatten class membership for unweighted scalar
and operator-valued commutators \cite{WeiZhangNC,WeiZhang}.
Hyt\"onen obtained results on spaces of homogeneous type for
commutators acting on a single scalar-weighted space $L^2(w)$,
$w\in A_2$ \cite{HytSchatten}.  

Volberg formulated the matrix $\mathcal A_p$ condition using
averaged norms and reducing operators \cite{Volberg1997}.  He proved
that this condition characterizes boundedness of the Hilbert transform
on $L^p(W)$, and obtained Haar and $S$-function estimates.
Matrix-weighted commutators were studied by Isralowitz, Kwon and Pott
\cite{IKP}, Isralowitz \cite{IsralowitzTwoMatrix}, and Isralowitz, Pott
and Treil \cite{IPT}.  For recent work on matrix weights, see, for
example,
\cite{BownikCruzUribe,CruzUribePenrod,JQWX,KakaroumpasNieraeth,NieraethLattice}.
Matrix-weighted estimates for Bergman projections, square functions and
maximal functions appear in
\cite{HuoWick,TreilSquare,TreilVolbergMaximal}.  There are also recent
results on function spaces with matrix weights; see, for example,
\cite{BuHytYangYuanI,BuHytYangYuanII,BuHytYangYuanIII,BuYangYuan}.

Let $1<p<\infty$.  Let $U,V$ be $m\times m$ matrix
$\mathcal A_p$ weights on $\mathbb R^d$ and let
$B:\mathbb R^d\to\mathbb C^{m\times m}$ be locally integrable.  If
$T$ has a scalar kernel and acts componentwise, we write
$$
 C_{B,T}=[M_B,T\otimes I_m].
$$
For a cube $Q$, let $\mathcal U_Q$ be a reducing matrix for $U$
at exponent $p$ and let $\mathcal V'_Q$ be a reducing matrix for
$V^{-p'/p}$ at exponent $p'$.  Put $B_Q=\avg_QB$ and $ \omega_Q(B)=\max\{\alpha_Q(B),\beta_Q(B)\}
$,
\begin{align*}
 \alpha_Q(B)
 &=\bigg(\avg_Q
 \|V^{1/p}(x)(B(x)-B_Q)\mathcal U_Q^{-1}\|_{\op}^{p}\,dx
 \bigg)^{1/p},
\end{align*}
\begin{align*}
 \beta_Q(B)
 &=\bigg(\avg_Q
 \|U^{-1/p}(x)(B^*(x)-B_Q^*)(\mathcal V'_Q)^{-1}\|_{\op}^{p'}\,dx
 \bigg)^{1/p'}.
\end{align*}
The quantity $\alpha_Q(B)$ measures the oscillation from $L^p(U)$ to
$L^p(V)$, and $\beta_Q(B)$ is the corresponding quantity for the
adjoint.  Their suprema over all cubes are equivalent.  This
equivalence alone gives no comparison when the cubes are restricted
by size or location, so we use both quantities.

Under the non-degeneracy assumptions below, $\sup_Q\omega_Q(B)$ is
comparable to the operator norm.  The limits over small cubes, large
cubes and cubes tending to infinity give the essential norm.  For a
Riesz transform at $p=2$, we replace the operator norm in these
oscillations by the Hilbert--Schmidt norm.  Their $\ell^r$ norm gives
the Schatten norm for $r>d$, and their weak $\ell^d$ norm gives the
weak Schatten norm at $r=d$.

For Riesz transforms, Isralowitz characterized boundedness with two
matrix weights \cite{IsralowitzTwoMatrix}.  Isralowitz, Pott and Treil
proved the upper bound for general Calder\'on--Zygmund operators by a
block-weight argument \cite{IPT}.  They also proved a converse using
the maximum over a collection of operators satisfying a lower-bound
condition; the full family of Riesz transforms is one such
collection.  Theorem A proves the converse for a fixed
Calder\'on--Zygmund operator with a scalar standard kernel satisfying
source-side non-degeneracy.  The kernel may depend on both variables.
Isralowitz, Rivera-R\'ios and S\'aez-Rivas have also proved convex-body
domination and strong-type bounds for vector-valued commutators with
several matrix symbols \cite{IRRS}.

\medskip
\noindent\textbf{Theorem A.}
Let $T$ be an $L^2$-bounded Calder\'on--Zygmund operator with a
scalar standard kernel.  Then
$$
 \|C_{B,T}\|_{L^p(U)\to L^p(V)}
 \lesssim \sup_Q\omega_Q(B).
$$
If the kernel is source-side non-degenerate in the sense of
\eqref{eq:source-side-nondegeneracy}, then
$$
 \|C_{B,T}\|_{L^p(U)\to L^p(V)}
 \simeq \sup_Q\omega_Q(B).
$$
The upper estimate is due to \cite{IPT}.  We prove the reverse
estimate under the stated non-degeneracy assumption.  See
Theorem~\ref{thm:two-matrix-boundedness}.

For compactness in the scalar two-weight setting, see
\cite{CLLV,HytOikSinko,LaceyLi}.  Laukkarinen and Sinko proved a
two-weight compactness characterization for rough homogeneous singular
integrals \cite{LaukkarinenSinko}.  They also gave a sufficient
condition for compactness when the symbol is scalar and there is one
matrix weight.  Theorem B treats two matrix $\mathcal A_p$ weights and
a matrix symbol.  It characterizes compactness and gives an
essential-norm estimate when the scalar standard kernel is
non-degenerate in both variables.  We impose the vanishing conditions
separately on the oscillations for the commutator and its adjoint.

Define
\begin{align*}
 \delta_{\rm s}(B)&=\lim_{a\downarrow0}
       \sup_{\ell(Q)\leq a}\omega_Q(B),\quad
 \delta_{\rm l}(B)=\lim_{a\uparrow\infty}
       \sup_{\ell(Q)\geq a}\omega_Q(B),\quad
 \delta_{\rm f}(B)=\lim_{a\uparrow\infty}
       \sup_{\dist(Q,0)\geq a}\omega_Q(B),
\end{align*}
and let
$
 \mathcal T_{V,U,p}(B)
 =\max\{\delta_{\rm s}(B),\delta_{\rm l}(B),\delta_{\rm f}(B)\}.
$

\medskip
\noindent\textbf{Theorem B.}
Suppose that the scalar kernel of $T$ is two-sided non-degenerate.
Then
$
 C_{B,T}:L^p(U;\mathbb C^m)\longrightarrow L^p(V;\mathbb C^m)
 \quad\text{is compact}
$
if and only if $B\in\BMO^p_{V,U}$ and
$
 \delta_{\rm s}(B)=\delta_{\rm l}(B)=\delta_{\rm f}(B)=0.
$
Moreover, for $B\in\BMO^p_{V,U}$,
$$
 \|C_{B,T}\|_{\ess}
 \simeq \mathcal T_{V,U,p}(B).
$$
The upper essential-norm estimate holds without non-degeneracy.
Thus the vanishing conditions imply compactness for every operator in
the stated Calder\'on--Zygmund class.  See Theorem~\ref{thm:main}.

We next consider Schatten classes for a fixed Riesz transform $R_j$
with $p=2$.  Janson and Wolff proved the scalar unweighted theorem in
dimensions $d\geq2$ \cite{JansonWolff}.  Rochberg and Semmes developed
estimates for nearly weakly orthonormal (NWO) sequences and applied
them to singular values \cite{RS}.  For weighted and endpoint
results, see \cite{FrankSukochevZanin,GongLiWick}.

We use the complex median method of Wei and Zhang \cite{WeiZhang}.
Their operator-valued theorem \cite{WeiZhangNC} includes matrix
symbols with $U=V=I_m$.  Hyt\"onen's weak endpoint and rigidity
results require further assumptions on the space, including
completeness and a suitable Poincar\'e inequality \cite{HytSchatten}.

In Theorem C, both weights and the symbol are matrix-valued.  On
conjugating the commutator to unweighted vector-valued $L^2$, we obtain
the kernel
$$
 V^{1/2}(x)(B(x)-B(y))K_j(x,y)U^{-1/2}(y).
$$
The factors $V^{1/2}(x)$ and $U^{-1/2}(y)$ remain on opposite sides
of $B(x)-B(y)$.  They need not commute with the symbol or with each
other.  Thus the conjugated operator is in general not an unweighted
operator-valued commutator, and the scalar two-weight theorem does not
apply to its matrix entries.

At exponent $2$, take
$$
 \mathcal U_Q=\langle U\rangle_Q^{1/2},
 \  \mathcal V'_Q=\langle V^{-1}\rangle_Q^{1/2}.
$$
Let $\alpha_Q^{\HS}(B)$ and $\beta_Q^{\HS}(B)$ be the preceding
oscillations with the operator norm replaced by the Hilbert--Schmidt
norm and put
$$
 \omega_Q^{\HS}(B)
 =\max\{\alpha_Q^{\HS}(B),\beta_Q^{\HS}(B)\}.
$$
Fix adjacent dyadic systems $\D^1,\ldots,\D^{3^d}$.  For
$0<r<\infty$ and $0<q\leq\infty$, define
$$
 \|B\|_{\mathfrak B^{r,q}_{V,U}}
 =\sum_{t=1}^{3^d}
 \big\|\{\omega_Q^{\HS}(B):Q\in\D^t\}\big\|_{\ell^{r,q}},
 \qquad
 \mathfrak W^d_{V,U}=\mathfrak B^{d,\infty}_{V,U}.
$$
We write $\mathfrak B^r_{V,U}=\mathfrak B^{r,r}_{V,U}$.  In the
proof we use the singular-value decomposition of the normalized Haar
coefficients.  The Hilbert--Schmidt norm is the $\ell^2$ norm of their
singular values.  The singular vectors may depend on the cube and the
Haar signature.

\medskip
\noindent\textbf{Theorem C.}
Let $d\geq2$, let $U,V\in\mathcal A_2$ and fix
$1\leq j\leq d$.
\begin{enumerate}[label=\textup{(\roman*)},leftmargin=2.5em]
\item If $d<r<\infty$, then
$
 C_{B,R_j}\in\Sch^r(L^2(U),L^2(V))
 \quad\Longleftrightarrow\quad
 B\in\mathfrak B^r_{V,U},
$
with equivalence of the two seminorms.
\item If $0<r\leq d$, then $C_{B,R_j}\in\Sch^r$ if and only if
there is a fixed matrix $B_0\in\mathbb C^{m\times m}$ such that
$B(x)=B_0$ for almost every $x$.
\item At the critical index,
$
 C_{B,R_j}\in\Sch^{d,\infty}(L^2(U),L^2(V))
 \quad\Longleftrightarrow\quad
 B\in\mathfrak W^d_{V,U},
$
again with equivalence of seminorms.
\end{enumerate}
See Theorem~\ref{sch:thm:target}.  The statement is restricted to
$d\geq2$; for the Hilbert transform, see \cite{LLW}.

When $m=1$, the conditions in Theorem C are equivalent to those in
\cite{LLWW}.  When $U=V=I_m$, one can apply the scalar unweighted
theorem to each entry; this case is also included in
\cite{WeiZhangNC}.  The norm equivalences in Theorem C show that the
spaces $\mathfrak B^r_{V,U}$, $r>d$, and $\mathfrak W^d_{V,U}$ do
not depend on the chosen adjacent dyadic systems, up to equivalence
of seminorms.

We also obtain a Sobolev estimate at the critical index.  For a
matrix function with weak first derivatives, set
$$
 G_{V,U}B(x)
 =\bigg(\sum_{k=1}^d
 \|V^{1/2}(x)(\partial_kB)(x)U^{-1/2}(x)\|_{\HS}^2
 \bigg)^{1/2}.
$$

\medskip
\noindent\textbf{Theorem D.}
If $B\in C^1(\mathbb R^d;\mathbb C^{m\times m})$, then
$
 \|G_{V,U}B\|_{L^d}
 \lesssim \|B\|_{\mathfrak W^d_{V,U}}
 \simeq \|C_{B,R_j}\|_{\Sch^{d,\infty}}.
$

If $U$ and $V$ are constant or uniformly elliptic matrix weights,
the first inequality also holds in the reverse direction.  For these
weights, weak Schatten membership is equivalent to the first-order
Sobolev condition, with derivatives understood in the weak sense.
For general matrix $\mathcal A_2$ weights, we do not prove the reverse
inequality.  The corresponding Sobolev characterization is open even
for two scalar weights.  See
Proposition~\ref{sch:prop:Sobolev-necessity} and
Theorem~\ref{sch:thm:elliptic-Sobolev}.

\subsection{Proof strategy
}
We now describe the proofs.  For the lower bounds in Theorems A and
C, we fix the reducing matrices on each cube and apply the complex
median lemma to scalar functions obtained from $B$.  The choices of
scalar functions are different in the two proofs.

For Theorem~A, fix a cube $Q$ and two standard basis vectors.  Apply
the complex median lemma to
$z_{st,Q}(x)=e_t^*\mathcal V_QB(x)\mathcal U_Q^{-1}e_s$.
Non-degeneracy and kernel regularity give a companion cube such that
the kernel is bounded away from zero on the product of the two cubes
and its values lie in a fixed sector.
Pairings with the commutator then bound the mean oscillation of
$z_{st,Q}$.  Summing over the finitely many pairs $(s,t)$ bounds the
mean oscillation of $\mathcal V_QB\mathcal U_Q^{-1}$ in operator
norm.  The matrix John--Nirenberg equivalences in \cite{IPT} give
bounds for the suprema of both $\alpha_Q(B)$ and $\beta_Q(B)$.

For Theorem~B we need estimates on each cube before taking the limits
in its size or location.  The global John--Nirenberg equivalence does
not give these estimates.  We test the commutator using the reciprocal
of the kernel on separated cubes.  This bounds the normalized columns
of the oscillation on each cube, with an error that becomes small
after subdivision.  Testing the adjoint gives the estimate for
$\beta_Q(B)$.  The test functions tend weakly to zero as the cubes
shrink, grow or tend to infinity.  This proves the lower
essential-norm bound.  For the upper bound, we truncate the operator
to obtain a compact term and estimate the remaining terms by
localized convex-body sparse bounds.

For the lower estimate in Theorem~C, take the singular-value
decomposition of each normalized Haar coefficient
$A_{Q,\varepsilon}=\mathcal V_Q\widehat B(Q,\varepsilon)
|Q|^{-1/2}\mathcal U_Q^{-1}$.  Let
$u_{Q,\varepsilon,\rho}$ and $v_{Q,\varepsilon,\rho}$ be its left
and right singular vectors.  Apply the complex median lemma to
$u_{Q,\varepsilon,\rho}^*\mathcal V_QB
\mathcal U_Q^{-1}v_{Q,\varepsilon,\rho}$.
Haar cancellation bounds the corresponding singular value by the
oscillation of this scalar function.  The NWO estimate for these
vector-valued families has a constant independent of the chosen
singular vectors.  We then sum over $\varepsilon$ and $\rho$ to
estimate the Hilbert--Schmidt oscillation.

For the upper estimate in Theorem~C, we conjugate the commutator to
unweighted vector-valued $L^2$ and expand the Riesz kernel by a
Whitney decomposition and Fourier series.  The resulting rank-one
terms contain the two matrix weights.  We decompose each matrix
coefficient by its singular values and group the Haar descendants by
generation.  Scalar $A_\infty$ estimates in each vector direction
give the decay needed to sum over the generations.  We use the scalar
NWO theorem of Rochberg and Semmes \cite{RS} and prove the required
extensions to finite-dimensional vector-valued families and matrix
weights.

Part~I proves Theorem A.  Part~II proves Theorem B.
Part~III proves Theorems C and D and discusses the
Sobolev question for general matrix weights.

\subsection{Further comments}
We make an additional remark on the  proof of Theorem A and Theorem B in this paper.
For Theorem A and Theorem B, our argument can be extend to the setting of space of homogeneous type $(X,d,\mu)$, see for example \cite{MR4203673,CLLV}. The boundedness for the commutator, Theorem A, can also be extended to the iterated commutators. As for the characterization of the Schatten class, it is very likely that our results can has the generalization on either stratified Lie groups \cite{MR4756021} or metric space with the corresponding Poincar\'e inequality condition 
\cite{HytSchatten}.

\section{Definitions and the boundedness theorem}

By a matrix weight we mean a measurable function $W$ with values in the
positive definite $m\times m$ matrices and with locally integrable
entries.  For $1<p<\infty$, put $p'=p/(p-1)$ and write
$$
 \|f\|_{L^p(W)}=
 \bigg(\int_{\R^d}|W^{1/p}(x)f(x)|^p\,dx\bigg)^{1/p}.
$$
The matrix $\mathcal A_p$ characteristic is
\begin{equation}\label{eq:matrix-Ap}
 [W]_{\mathcal A_p}
 :=\sup_Q\avg_Q
 \bigg(\avg_Q
 \|W^{1/p}(x)W^{-1/p}(y)\|_{\op}^{p'}\,dy
 \bigg)^{p/p'}dx.
\end{equation}
We write $W\in\mathcal A_p$ when this quantity is finite.  Then the
dual weight $W^\#=W^{-p'/p}$ belongs to $\mathcal A_{p'}$ and
\begin{equation}\label{eq:dual-Ap-characteristic}
 [W^\#]_{\mathcal A_{p'}}^{1/p'}
 \simeq_{m,p} [W]_{\mathcal A_p}^{1/p}.
\end{equation}
See \cite[Lemma~2.4]{CUIM}.

For every cube $Q$, choose positive definite reducing matrices
$\mathcal W_Q$ and $\mathcal W'_Q$ such that
\begin{align*}
 |\mathcal W_Qe|
 &\simeq \bigg(\avg_Q|W^{1/p}(x)e|^p\,dx\bigg)^{1/p}\qquad {\rm and }\qquad
 |\mathcal W'_Qe|
 &\simeq \bigg(\avg_Q|W^{-1/p}(x)e|^{p'}\,dx\bigg)^{1/p'}.
\end{align*}
Thus $\mathcal W'_Q$ is a reducing matrix for $W^\#$ at exponent
$p'$.  We use the notation $\mathcal U_Q,\mathcal U'_Q$ and
$\mathcal V_Q,\mathcal V'_Q$ for $U$ and $V$.  Different choices of
reducing matrices give equivalent quantities below.

Put $B_Q=\avg_QB$ and define
\begin{align}
 \alpha_Q(B)
 &=\bigg(\avg_Q
 \big\|V^{1/p}(x)(B(x)-B_Q)\mathcal U_Q^{-1}\big\|_{\op}^{p}\,dx
 \bigg)^{1/p},                                      \label{eq:alpha}\\
 \beta_Q(B)
 &=\bigg(\avg_Q
 \big\|U^{-1/p}(x)(B^*(x)-B_Q^*)(\mathcal V'_Q)^{-1}\big\|_{\op}^{p'}\,dx
 \bigg)^{1/p'},                                     \label{eq:beta}\\
\omega_Q(B)&=\max\{\alpha_Q(B),\beta_Q(B)\}.       \label{eq:omega}
\end{align}
The two global seminorms are equivalent, but this equivalence does not
compare their small-, large-, or far-cube tails separately.  We
therefore retain both quantities.

\begin{definition}\label{def:bmo}
We write $B\in\BMO^p_{V,U}$ if
$
 \|B\|_{\BMO^p_{V,U}}:=\sup_Q\omega_Q(B)<\infty.
$
\end{definition}

We use the following operator class.  A Calder\'on--Zygmund operator
$T$ is an operator bounded on $L^2(\R^d)$ and associated, away from the
diagonal, with a scalar kernel $K$ such that
\begin{align}
 &|K(x,y)|\leq C_K|x-y|^{-d},                                      \label{eq:kernel-size}\\
 &|K(x,y)-K(x',y)|+|K(y,x)-K(y,x')|
 \leq C_K\frac{|x-x'|^\delta}{|x-y|^{d+\delta}}                  \label{eq:kernel-holder}
\end{align}
whenever $|x-x'|\leq |x-y|/2$, where $0<\delta\leq1$.  The operator
acts componentwise on vector-valued functions.  We also require the
usual off-diagonal identity
$$
 \langle Tf,g\rangle
 =\iint K(x,y)f(y)\overline{g(x)}\,dy\,dx
$$
for bounded compactly supported scalar functions with disjoint
supports.  We put
$
 C_{B,T}=[M_B,T\otimes I_m].
$

\begin{definition}
\label{def:two-sided-nondegeneracy}
The kernel $K$ is two-sided non-degenerate if there are $c_0,C_0>0$
such that, for every $x\in\R^d$ and $r>0$, there are points
$y_+,y_-\in B(x,C_0r)\setminus B(x,r)$ for which
\begin{equation}\label{eq:two-sided-nondegeneracy}
 |K(x,y_+)|\geq \frac1{c_0r^d},
 \qquad
 |K(y_-,x)|\geq \frac1{c_0r^d}.
\end{equation}
Thus both $T$ and $T^*$ satisfy the non-degeneracy condition used in
\cite{HytNondeg,LaceyLi}.
\end{definition}

The upper estimate below is the block-weight theorem of
Isralowitz--Pott--Treil \cite{IPT} for the present matrix $\mathcal A_p$
setting.  Isralowitz \cite{IsralowitzTwoMatrix} proved the reverse
estimate for Riesz transforms.  Part~I proves the reverse estimate for
one variable-kernel operator under the source-side portion
of Definition~\ref{def:two-sided-nondegeneracy}, in the same
primal--dual notation that is used for the vanishing tails.

\begin{theorem}
\label{thm:two-matrix-boundedness}
Let $1<p<\infty$, let $U,V$ be $m\times m$ matrix $\mathcal A_p$
weights on $\R^d$ and  let $B:\R^d\to\C^{m\times m}$ be locally
integrable.  Let $T$ be a Calder\'on--Zygmund operator whose kernel is
source-side non-degenerate: there are $c_0,C_0>0$ such that, for every
$x\in\mathbb R^d$ and $r>0$, there is
$y_-\in B(x,C_0r)\setminus B(x,r)$ for which
\begin{equation}\label{eq:source-side-nondegeneracy}
 |K(y_-,x)|\geq \frac1{c_0r^d}.
\end{equation}
Then
\begin{equation}\label{eq:two-matrix-bdd-equivalence}
 C_{B,T}:L^p(U;\C^m)\longrightarrow L^p(V;\C^m)
 \quad\hbox{is bounded}
\end{equation}
if and only if
\begin{equation}\label{eq:two-matrix-bdd-BMO}
 \sup_Q\omega_Q(B)
 =\max\left\{\sup_Q\alpha_Q(B),\sup_Q\beta_Q(B)\right\}<\infty.
\end{equation}
Moreover,
\begin{equation}\label{eq:two-matrix-bdd-norm}
 \|C_{B,T}\|_{L^p(U)\to L^p(V)}
 \simeq \|B\|_{\BMO^p_{V,U}}.
\end{equation}
The constants depend only on $d,m,p$, the $\mathcal A_p$
characteristics of $U$ and $V$, the Calder\'on--Zygmund constants of
$T$, the constants
in \eqref{eq:source-side-nondegeneracy} and  the fixed normalizations of
the reducing matrices.  The upper estimate does not require
non-degeneracy.
\end{theorem}

The next section introduces the local two-matrix form.
Section~\ref{sec:boundedness-proof} proves the two estimates in
Theorem~\ref{thm:two-matrix-boundedness}.

\section{The local two-matrix form}

The pairwise quantity in the block-weight argument of
Isralowitz--Pott--Treil is
\begin{equation}
 \widehat\omega_Q(B)=
 \bigg[\avg_Q\bigg(
 \avg_Q\big\|V^{1/p}(x)(B(x)-B(y))U^{-1/p}(y)\big\|_{\op}^{p'}\,dy
 \bigg)^{p/p'}dx\bigg]^{1/p}.                       \label{eq:pairwise}
\end{equation}

\begin{lemma}\label{lem:localcomparison}
Let $U,V\in\mathcal A_p$ and  let $B$ be locally integrable.  For
every cube $Q$,
\begin{equation}
 \alpha_Q(B)+\beta_Q(B)\lesssim \widehat\omega_Q(B)
 \lesssim \alpha_Q(B)+\beta_Q(B).                    \label{eq:localcomparison}
\end{equation}
No supremum over other cubes occurs in this estimate.
\end{lemma}

\begin{proof}
Split $B(x)-B(y)$ at $B_Q$.  For the first part, factor
$$
 V^{1/p}(x)(B(x)-B_Q)U^{-1/p}(y)
 =\big[V^{1/p}(x)(B(x)-B_Q)\mathcal U_Q^{-1}\big]
  \big[\mathcal U_QU^{-1/p}(y)\big].
$$
The reducing-matrix characterization of matrix $\mathcal A_p$ gives
$$
 \bigg(\avg_Q\|\mathcal U_QU^{-1/p}(y)\|_{\op}^{p'}\,dy\bigg)^{1/p'}
 \lesssim_{m,p} [U]_{\mathcal A_p}^{1/p}.
$$
Indeed, apply the reducing property of $\mathcal U'_Q$ to the columns
of $\mathcal U_Q$ and use
$\|\mathcal U'_Q\mathcal U_Q\|_{\op}
 \lesssim [U]_{\mathcal A_p}^{1/p}$; see
\cite[Proposition~2.2 and formula~(2.6)]{CUIM}.
Thus this part is bounded by a constant times $\alpha_Q(B)$.  For the
other part, take adjoints and use
$$
 U^{-1/p}(y)(B^*(y)-B_Q^*)V^{1/p}(x)
 =\big[U^{-1/p}(y)(B^*(y)-B_Q^*)(\mathcal V'_Q)^{-1}\big]
  \big[\mathcal V'_QV^{1/p}(x)\big].
$$
The reducing property and the matrix $\mathcal A_p$ condition give
$$
 \bigg(\avg_Q\|\mathcal V'_QV^{1/p}(x)\|_{\op}^{p}\,dx\bigg)^{1/p}
 \lesssim_{m,p} [V]_{\mathcal A_p}^{1/p}.
$$
This part is bounded by a constant times $\beta_Q(B)$.

For the reverse estimate, use
\begin{align*}
 &V^{1/p}(x)(B(x)-B_Q)\mathcal U_Q^{-1}=\avg_Q V^{1/p}(x)(B(x)-B(y))U^{-1/p}(y)
       \big[U^{1/p}(y)\mathcal U_Q^{-1}\big] \,dy .
\end{align*}
The reducing property, applied to a fixed orthonormal basis, gives
\begin{equation}
 \bigg(\avg_Q
 \|U^{1/p}(y)\mathcal U_Q^{-1}\|_{\op}^{p}\,dy
 \bigg)^{1/p}\lesssim_m1.
 \label{eq:primal-reducing-basis-bound}
\end{equation}
H\"older's inequality in $y$ and
\eqref{eq:primal-reducing-basis-bound} give
$\alpha_Q(B)\lesssim \widehat\omega_Q(B)$.
Also,
\begin{align*}
 &U^{-1/p}(x)(B^*(x)-B_Q^*)(\mathcal V'_Q)^{-1}=\avg_Q U^{-1/p}(x)(B^*(x)-B^*(y))V^{1/p}(y)
       \big[V^{-1/p}(y)(\mathcal V'_Q)^{-1}\big] \,dy .
\end{align*}
The corresponding reducing estimate is
\begin{equation}
 \bigg(\avg_Q
 \|V^{-1/p}(y)(\mathcal V'_Q)^{-1}\|_{\op}^{p'}\,dy
 \bigg)^{1/p'}\lesssim_m1.
 \label{eq:dual-reducing-basis-bound}
\end{equation}
Take the $L^{p'}(Q,dx/|Q|)$ norm, apply Minkowski's inequality and 
then apply H\"older's inequality in $y$ with exponents $p$ and $p'$.
The first factor is $\widehat\omega_Q(B)$ after interchanging the names
of $x$ and $y$ and  the second is bounded by the reducing property of
$\mathcal V'_Q$.  Hence
$\beta_Q(B)\lesssim \widehat\omega_Q(B)$.
\end{proof}

For a dyadic grid $\D$, let $\BMO^p_{V,U,\D}$ denote the same space
with the supremum restricted to $Q\in\D$.

\section{Proof of the boundedness theorem}
\label{sec:boundedness-proof}

Changing the reducing matrices in the definition of
$\BMO^p_{V,U}$ changes its norm only by a constant depending on $m$ and
$p$.

For a locally integrable matrix function $A$, put
$$
 \mathscr T_A=\{f\in L_c^\infty(\mathbb R^d;\mathbb C^m):
 Af\in L^2(\mathbb R^d;\mathbb C^m)\}.
$$
The class $\mathscr T_B$ is dense in $L^p(U)$ and 
$\mathscr T_{B^*}$ is dense in
$L^{p'}(V^{-p'/p})$.  We give the argument for the first assertion.
Under the isometry $f\mapsto U^{1/p}f$, first approximate an arbitrary
unweighted $L^p$ function by bounded compactly supported functions.
Then restrict these approximants to the increasing measurable sets on
which both $\|U^{-1/p}\|_{\op}$ and $\|B\|_{\op}$ are bounded.  After
multiplication by $U^{-1/p}$, the resulting functions are bounded,
compactly supported and  their products with $B$ belong to $L^2$.
This proves density.  The second assertion follows in the same way,
using the isometry $g\mapsto V^{-1/p}g$ from
$L^{p'}(V^{-p'/p})$ to unweighted $L^{p'}$ and restricting to sets on
which $\|V^{1/p}\|_{\op}$ and $\|B^*\|_{\op}$ are bounded.

On $\mathscr T_B\times\mathscr T_{B^*}$ define the commutator form by
\begin{equation}\label{eq:commutator-form-definition}
 \Lambda_{B,T}(f,g)
 :=\langle (T\otimes I_m)f,B^*g\rangle_{L^2}
   -\langle (T\otimes I_m)(Bf),g\rangle_{L^2}.
\end{equation}
All $L^2$ and duality pairings in this paper are linear in the first
variable.
Both terms are well defined.  We say that $C_{B,T}$ is bounded from
$L^p(U)$ to $L^p(V)$ if this form has a bounded extension to
$L^p(U)\times L^{p'}(V^{-p'/p})$.  The representing operator exists
by reflexivity and duality and is unique by density.

If $f,g$ are bounded, compactly supported vector functions and the
distance between their supports is positive, then the extended form,
when it exists, agrees with
\begin{equation}\label{eq:commutator-separated-form-bdd}
 \Lambda_{B,T}(f,g)
 =\iint
 \big\langle(B(x)-B(y))K(x,y)f(y),g(x)\big\rangle
 \,dy\,dx.
\end{equation}
The integral is absolutely convergent.  Formula
\eqref{eq:commutator-separated-form-bdd} follows first for
$f\in\mathscr T_B$ and $g\in\mathscr T_{B^*}$ from the off-support
kernel representation and Fubini's theorem.  For general bounded
separated tests, truncate on the sets where $\|B\|_{\op}$ is bounded
and use dominated convergence in the double integral and in the two
weighted norms.  This also shows that the bounded extension agrees with
the usual distributional commutator wherever the latter is defined.

We divide the proof into the upper and lower estimates.

\subsection{The upper estimate}

Recall the pairwise quantity $\widehat\omega_Q(B)$ from
\eqref{eq:pairwise} and  put
\begin{equation}\label{eq:pairwise-global-bdd}
 \widehat{\mathcal O}_{V,U,p}(B)=\sup_Q\widehat\omega_Q(B).
\end{equation}
Lemma~\ref{lem:localcomparison} gives
\begin{equation}\label{eq:pairwise-from-alpha-beta}
 \widehat{\mathcal O}_{V,U,p}(B)
 \leq C_0\sup_Q\omega_Q(B),
\end{equation}
where $C_0$ depends only on $m,p$ and the $\mathcal A_p$
characteristics of $U$ and $V$.

\begin{lemma}\label{lem:block-weight-bdd}
Let $U,V\in\mathcal A_p$ and  let $B$ be locally integrable.  Suppose
that $\sup_Q\omega_Q(B)<\infty$ and
$0<\widehat{\mathcal O}_{V,U,p}(B)<\infty$ and  put $A=B/\widehat{\mathcal O}_{V,U,p}(B)$.  Define
\begin{equation}\label{eq:block-Phi-bdd}
 \Phi_A(x)=
 \begin{pmatrix}
  V^{1/p}(x)&V^{1/p}(x)A(x)\\
  0&U^{1/p}(x)
 \end{pmatrix},
 \qquad
 \mathbb W_A(x)=\big(\Phi_A(x)^*\Phi_A(x)\big)^{p/2}.
\end{equation}
Then $\mathbb W_A$ is a $2m\times2m$ matrix $\mathcal A_p$ weight and
\begin{equation}\label{eq:block-Ap-boundedness}
 [\mathbb W_A]_{\mathcal A_p}
 \leq c_p\big([U]_{\mathcal A_p}+[V]_{\mathcal A_p}+1\big).
\end{equation}
One may take $c_p=3^{\max\{1,p-1\}}$.
\end{lemma}

\begin{proof}
We first check local integrability.  On a fixed cube $Q$, the definition
of $\alpha_Q(A)$ gives
$$
 V^{1/p}(A-A_Q)\mathcal U_Q^{-1}\in L^p(Q).
$$
Multiplication on the right by the constant matrix $\mathcal U_Q$
shows that $V^{1/p}(A-A_Q)\in L^p(Q)$.  Also
$V^{1/p}A_Q\in L^p(Q)$.  Hence $V^{1/p}A\in L^p(Q)$.  In the same way,
the definition of $\beta_Q(A)$ and passage to adjoints give
$$
 (\mathcal V'_Q)^{-1}(A-A_Q)U^{-1/p}\in L^{p'}(Q).
$$
It follows that $(A-A_Q)U^{-1/p}\in L^{p'}(Q)$, while
$A_QU^{-1/p}\in L^{p'}(Q)$.  Thus $AU^{-1/p}\in L^{p'}(Q)$.  These
two observations show that $\Phi_A\in L^p_{\rm loc}$ and
$\Phi_A^{-1}\in L^{p'}_{\rm loc}$ in operator norm.  Put
$|\Phi_A|=(\Phi_A^*\Phi_A)^{1/2}$.  Then
$\mathbb W_A=|\Phi_A|^p$ and $\mathbb W_A^{1/p}=|\Phi_A|$, so
$$
 \|\mathbb W_A\|_{\op}=\|\Phi_A\|_{\op}^p,
 \qquad
 \|\mathbb W_A^{-p'/p}\|_{\op}
 =\|\Phi_A^{-1}\|_{\op}^{p'}.
$$
Hence $\mathbb W_A$ and $\mathbb W_A^{-p'/p}$ are locally integrable.

The inverse of $\Phi_A$ is
$$
 \Phi_A(y)^{-1}
 =\begin{pmatrix}
 V^{-1/p}(y)&-A(y)U^{-1/p}(y)\\
 0&U^{-1/p}(y)
 \end{pmatrix}.
$$
Consequently,
\begin{equation}\label{eq:block-product-bdd}
 \Phi_A(x)\Phi_A(y)^{-1}
 =\begin{pmatrix}
 V^{1/p}(x)V^{-1/p}(y)&
 V^{1/p}(x)(A(x)-A(y))U^{-1/p}(y)\\
 0&U^{1/p}(x)U^{-1/p}(y)
 \end{pmatrix}.
\end{equation}
Write the polar decomposition as
$$
 \Phi_A(x)=E_A(x)\mathbb W_A^{1/p}(x),
$$
where $E_A(x)$ is unitary.  Hence
$$
 \mathbb W_A^{1/p}(x)\mathbb W_A^{-1/p}(y)
 =E_A(x)^*\Phi_A(x)\Phi_A(y)^{-1}E_A(y),
$$
and the two sides have the same operator norm.

For a block upper triangular matrix,
$$
 \left\|\begin{pmatrix}P&D\\0&S\end{pmatrix}\right\|_{\op}
 \leq \|P\|_{\op}+\|D\|_{\op}+\|S\|_{\op}.
$$
Use this estimate in \eqref{eq:block-product-bdd}.  Raise first to
$p'$, take the $y$-average, raise to $p/p'$ and  then take the
$x$-average.  The two diagonal terms are bounded by
$[V]_{\mathcal A_p}$ and $[U]_{\mathcal A_p}$.  The off-diagonal term
is $\widehat\omega_Q(A)^p$.  Since
$\sup_Q\widehat\omega_Q(A)=1$, we obtain
\eqref{eq:block-Ap-boundedness}.
The elementary inequalities
$$
 (a_1+a_2+a_3)^{p'}\leq3^{p'-1}\sum_{k=1}^3a_k^{p'},
 \qquad
 (b_1+b_2+b_3)^{p/p'}
 \leq3^{\max\{p/p'-1,0\}}\sum_{k=1}^3b_k^{p/p'}
$$
give $c_p=3^{\max\{1,p-1\}}$.
\end{proof}

\begin{proposition}\label{prop:two-matrix-upper}
Let $1<p<\infty$, let $U,V$ be matrix $\mathcal A_p$ weights, let
$B$ be locally integrable and  let $T$ be an $L^2$-bounded
Calder\'on--Zygmund operator with a scalar standard kernel.  No
non-degeneracy is assumed.  Then
\begin{equation}\label{eq:two-matrix-upper}
 \|[M_B,T\otimes I_m]\|_{L^p(U)\to L^p(V)}
 \leq C\sup_Q\omega_Q(B).
\end{equation}
\end{proposition}

\begin{proof}
Assume first that $0<\widehat{\mathcal O}_{V,U,p}(B)<\infty$ and use the notation in
Lemma~\ref{lem:block-weight-bdd}.  Multiplication by $\Phi_A$ is an
isometry from
$L^p(\mathbb W_A;\C^{2m})$ onto unweighted
$L^p(\R^d;\C^{2m})$.  Indeed,
$$
 |\Phi_A(x)F(x)|=|\mathbb W_A^{1/p}(x)F(x)|
$$
 because the polar factor is unitary.  Define the bounded operator
\begin{equation}\label{eq:block-conjugation-bdd}
 \mathcal T_A
 :=M_{\Phi_A}(T\otimes I_{2m})M_{\Phi_A^{-1}}
 \quad\hbox{on }L^p(\R^d;\C^{2m})
\end{equation}
by this isometric conjugation.  This definition does not multiply a
distribution by the measurable matrix $\Phi_A$.

Let $I_2h=(0,h)$ and let $P_1(h_1,h_2)=h_1$.  Take
$f\in\mathscr T_A$ and $g\in\mathscr T_{A^*}$.  Then
$f,Af,g,A^*g\in L^2$, so direct block multiplication, interpreted in
$L^2$ pairings, gives
\begin{equation}\label{eq:block-upper-entry-bdd}
 \big\langle P_1\mathcal T_AI_2(U^{1/p}f),V^{-1/p}g\big\rangle
 =\Lambda_{A,T}(f,g).
\end{equation}
Indeed,
$\Phi_A^{-1}I_2(U^{1/p}f)=(-Af,f)$ and  the first component after
applying $T\otimes I_{2m}$ and multiplying by $\Phi_A$ pairs with
$V^{-1/p}g$ as
$$
 \langle (T\otimes I_m)f,A^*g\rangle_{L^2}
 -\langle (T\otimes I_m)(Af),g\rangle_{L^2}.
$$
The density of $\mathscr T_A$ and $\mathscr T_{A^*}$ shows that the
upper-right block represents the unique bounded extension of the
commutator form.  This is also the block argument of
\cite[Lemma~1.3]{IPT}.  Thus
\begin{align}
 \|[M_A,T\otimes I_m]\|_{L^p(U)\to L^p(V)}
 &\leq
 \|T\otimes I_{2m}\|_{L^p(\mathbb W_A)\to L^p(\mathbb W_A)}.
                                                        \label{eq:block-upper-bdd}
\end{align}
The matrix-weighted Calder\'on--Zygmund theorem gives
\begin{equation}\label{eq:matrix-CZO-bound-bdd}
 \|T\otimes I_n\|_{L^p(W)\to L^p(W)}
 \leq C_{T,d,n,p}[W]_{\mathcal A_p}^{
  1+\frac1{p-1}-\frac1p};
\end{equation}
see \cite[Corollary~1.16]{CUIM}.  Lemma~\ref{lem:block-weight-bdd} and
\eqref{eq:block-upper-bdd} therefore give
\begin{equation}\label{eq:pairwise-upper-bdd}
 \|[M_B,T\otimes I_m]\|_{L^p(U)\to L^p(V)}
 \leq C_{T,d,m,p,[U]_{\mathcal A_p},[V]_{\mathcal A_p}}
 \widehat{\mathcal O}_{V,U,p}(B).
\end{equation}
Now apply \eqref{eq:pairwise-from-alpha-beta}.

If $\widehat{\mathcal O}_{V,U,p}(B)=0$, then $B(x)=B(y)$ for almost every
$(x,y)\in Q\times Q$ for each cube $Q$ with rational center and rational
side length.  This countable family is connected by overlapping cubes
and covers $\R^d$.  Thus $B$ is a constant matrix modulo one null set,
and the commutator is zero.  If
$\sup_Q\omega_Q(B)<\infty$, then \eqref{eq:pairwise-from-alpha-beta}
rules out $\widehat{\mathcal O}_{V,U,p}(B)=\infty$.  This proves
\eqref{eq:two-matrix-upper}.
\end{proof}

\subsection{The lower estimate by complex medians}
\label{subsec:complex-median-bdd-lower}

We record two consequences of the matrix $\mathcal A_q$ condition.
They will be used only for cubes whose side lengths and mutual
distances are comparable.

\begin{lemma}
\label{lem:local-matrix-weight-bdd}
Let $1<q<\infty$.  Let $W$ be an $m\times m$ matrix
$\mathcal A_q$ weight and  let $\mathcal W_Q$ be an exponent-$q$
reducing matrix for $W$ on $Q$.

\noindent (i) For every measurable matrix function $F$ on $Q$ for which the
right side below is finite,
\begin{equation}\label{eq:weighted-average-center-bdd}
 \bigg(\avg_Q
 \|W^{1/q}(x)\avg_QF(y)\,dy\|_{\op}^{q}\,dx
 \bigg)^{1/q}
 \leq C_W
 \bigg(\avg_Q\|W^{1/q}(x)F(x)\|_{\op}^{q}\,dx\bigg)^{1/q}.
\end{equation}

\noindent  (ii) Suppose that $Q_1,Q_2$ have the same side length and that a cube
$S$ contains $Q_1\cup Q_2$ with $|S|\leq A|Q_1|$.  Then
\begin{equation}\label{eq:companion-reducing-bdd}
 \bigg(\avg_{Q_2}
 |W^{1/q}(x)\mathcal W_{Q_1}^{-1}e|^q\,dx\bigg)^{1/q}
 \leq C_{W,A}|e|,
 \qquad e\in\C^m.
\end{equation}
The constants depend only on $m,q,[W]_{\mathcal A_q}$ and , in
\eqref{eq:companion-reducing-bdd}, on $A$.
\end{lemma}

\begin{proof}
For a cube $S$, define
$$
 \mathsf A_Sf=\one_S\avg_S f.
$$
For $x\in S$, H\"older's inequality gives
\begin{align*}
 \left|W^{1/q}(x)\avg_Sf(y)\,dy\right|
 &\leq
 \bigg(\avg_S
  \|W^{1/q}(x)W^{-1/q}(y)\|_{\op}^{q'}\,dy\bigg)^{1/q'}
 \bigg(\avg_S|W^{1/q}(y)f(y)|^q\,dy\bigg)^{1/q}.
\end{align*}
Raise this inequality to the power $q$ and integrate over $S$.  By
the definition of $[W]_{\mathcal A_q}$,
\begin{equation}\label{eq:weighted-averaging-operator-bdd}
 \|\mathsf A_Sf\|_{L^q(W)}
 \leq [W]_{\mathcal A_q}^{1/q}
 \|\one_Sf\|_{L^q(W)}.
\end{equation}
Apply this estimate to the columns of $F$.  The inequalities
$\|T\|_{\op}\leq\sum_{s=1}^m|Te_s|$ and
$|Te_s|\leq\|T\|_{\op}$ give
\eqref{eq:weighted-average-center-bdd}, with a factor at most $m$.

For (ii), put $a=\mathcal W_{Q_1}^{-1}e$.  Since
$$
 \mathsf A_S(\one_{Q_1}a)
 =\frac{|Q_1|}{|S|}\one_Sa,
$$
formula \eqref{eq:weighted-averaging-operator-bdd} gives
$$
 \frac{|Q_1|}{|S|}
 \|\one_{Q_2}a\|_{L^q(W)}
 \leq [W]_{\mathcal A_q}^{1/q}
 \|\one_{Q_1}a\|_{L^q(W)}.
$$
Now use $|Q_1|=|Q_2|$, $|S|\leq A|Q_1|$ and  the reducing property
on $Q_1$.  This proves \eqref{eq:companion-reducing-bdd}.
\end{proof}

The next estimate controls an average over a companion cube.

\begin{lemma}
\label{lem:change-center-bdd}
Let $U,V\in\mathcal A_p$.  For every cube $Q$ and every constant
matrix $A$,
\begin{equation}\label{eq:change-center-bdd}
 \alpha_Q(B)
 \leq C
 \bigg(\avg_Q
 \|V^{1/p}(x)(B(x)-A)\mathcal U_Q^{-1}\|_{\op}^{p}\,dx
 \bigg)^{1/p}.
\end{equation}
The constant depends only on $m,p,[V]_{\mathcal A_p}$.
\end{lemma}

\begin{proof}
Put $F=(B-A)\mathcal U_Q^{-1}$.  Then
$$
 (B_Q-A)\mathcal U_Q^{-1}=\avg_QF.
$$
Apply Lemma~\ref{lem:local-matrix-weight-bdd}(i), with $q=p$ and
$W=V$, to
this identity.  The triangle inequality applied to
$$
 (B-B_Q)\mathcal U_Q^{-1}=F-\avg_QF
$$
gives \eqref{eq:change-center-bdd}.
\end{proof}

We prove the lower estimate without approximate weak factorization.
Only the source-side condition
\eqref{eq:source-side-nondegeneracy} is used.  The kernel and the
matrix symbol may be complex-valued.

We first record the form of the complex median lemma that we need.

\begin{lemma}
\label{lem:complex-median-WZ}
Let $(X,\mu)$ be a measure space, let $E\subset X$ satisfy
$0<\mu(E)<\infty$ and  let $z:E\to\mathbb C$ be measurable.  There
are a point $c\in\mathbb C$ and two perpendicular lines through $c$.
The four closed quadrants $\Gamma_1,\ldots,\Gamma_4$ determined by
these lines satisfy
\begin{equation}
 \mu\{x\in E:z(x)\in\Gamma_j\}
 \geq \frac1{16}\mu(E),
 \qquad 1\leq j\leq4.
 \label{eq:complex-median-major-sets}
\end{equation}
Opposite quadrants are indexed by $j$ and $j+2$, with indices taken
modulo $4$.
\end{lemma}

This is the complex median theorem of Wei and Zhang
\cite[Theorem~1.5]{WeiZhang}.  The quadrants in the statement are
closed.  Thus the statement also applies when the push-forward
measure $z_*(\mathbf1_E\mu)$ has atoms.  A point on a boundary may
belong to two or four quadrants.  We allow this overlap on the companion
cube and choose a measurable disjoint partition on the original cube.

We next turn a non-zero kernel value into one fixed kernel sector on a
pair of cubes.

\begin{lemma}
\label{lem:median-companion-sector}
Assume \eqref{eq:source-side-nondegeneracy}.  There are fixed constants
$A\geq4\sqrt d$, $c_A>0$ and  $0<\varepsilon<\pi/8$ with the
following property.  For every cube $Q$, there is a cube $Q^\sharp$
with $\ell(Q^\sharp)=\ell(Q)$.  There is also a number
$\phi_Q\in\mathbb R$ such that
\begin{align}
 \dist(Q,Q^\sharp)&\simeq A\ell(Q),
 \label{eq:median-companion-distance}\\
 |K(y,x)|&\geq \frac{c_A}{|Q|},
 \qquad x\in Q,\quad y\in Q^\sharp,
 \label{eq:median-kernel-lower}\\
 |\arg(e^{-i\phi_Q}K(y,x))|&\leq\varepsilon,
 \qquad x\in Q,\quad y\in Q^\sharp.
 \label{eq:median-kernel-sector}
\end{align}
There is also a cube $S_Q\supset Q\cup Q^\sharp$ such that
\begin{equation}
 |S_Q|\leq C_A|Q|.
 \label{eq:median-common-ancestor}
\end{equation}
The constants are independent of $Q$.
\end{lemma}

\begin{proof}
Let $x_Q$ be the center of $Q$ and put $r=A\ell(Q)$.  By
\eqref{eq:source-side-nondegeneracy}, there is a point $y_Q$ such that
\begin{equation}
 A\ell(Q)<|y_Q-x_Q|\leq C_0A\ell(Q),
 \qquad
 |K(y_Q,x_Q)|\geq c_0^{-1}A^{-d}|Q|^{-1}.
 \label{eq:median-center-value}
\end{equation}
Let $Q^\sharp$ be the cube centered at $y_Q$ with the same side
length and orientation as $Q$.  If $A$ is large enough, the two
regularity estimates for $K$ give
\begin{align*}
 |K(y,x)-K(y_Q,x_Q)|
 &\leq |K(y,x)-K(y,x_Q)|
      +|K(y,x_Q)-K(y_Q,x_Q)|\\
 &\leq C_K\frac{|x-x_Q|^\delta}{|y-x|^{d+\delta}}
      +C_K\frac{|y-y_Q|^\delta}{|y-x_Q|^{d+\delta}}\\
 &\leq C_{d,\delta}C_K A^{-d-\delta}|Q|^{-1}
\end{align*}
for $x\in Q$ and $y\in Q^\sharp$.  In the first difference the
second kernel variable changes.  In the second difference the first
kernel variable changes.  The restriction in the H\"older estimate
holds because $A\geq4\sqrt d$.

Write $K(y_Q,x_Q)=e^{i\phi_Q}\kappa_Q$, where $\kappa_Q>0$.  It
follows from \eqref{eq:median-center-value} that
\begin{equation}
 \left|e^{-i\phi_Q}K(y,x)-\kappa_Q\right|
 \leq \eta_A\kappa_Q,
 \qquad
 \eta_A\leq C_{d,\delta}C_Kc_0A^{-\delta}.
 \label{eq:median-relative-kernel-error}
\end{equation}
Fix $A$ so large that $\eta_A<1/8$ and
$$
 \arctan\frac{\eta_A}{1-\eta_A}<\frac\pi8.
$$
Then \eqref{eq:median-kernel-lower} holds with
$c_A=(1-\eta_A)c_0^{-1}A^{-d}$ and 
\eqref{eq:median-kernel-sector} holds with
$\varepsilon=\arctan(\eta_A/(1-\eta_A))$.  The distance estimate
follows from the choice of $y_Q$.  A cube with side length at most
$C(C_0A+1)\ell(Q)$ contains both cubes.  This gives
\eqref{eq:median-common-ancestor}.
\end{proof}

We now prove the lower estimate.

\begin{proposition}
\label{lem:general-czo-global-lower}
Let $U,V\in\mathcal A_p$, let $B$ be locally integrable and  suppose
that
$
 C_{B,T}:L^p(U;\mathbb C^m)\longrightarrow L^p(V;\mathbb C^m)
$
is bounded.  If $K$ satisfies
\eqref{eq:source-side-nondegeneracy}, then
\begin{equation}
 \sup_Q\avg_Q
 \|\mathcal V_Q(B(x)-B_Q)\mathcal U_Q^{-1}\|_{\op}\,dx
 \lesssim
 \|C_{B,T}\|_{L^p(U)\to L^p(V)}.
 \label{eq:weak-reducing-lower}
\end{equation}
Consequently,
\begin{equation}
 \sup_Q\alpha_Q(B)+\sup_Q\beta_Q(B)
 \lesssim
 \|C_{B,T}\|_{L^p(U)\to L^p(V)}.
 \label{eq:alpha-beta-general-lower}
\end{equation}
The constants have the dependence stated in
Theorem~\ref{thm:two-matrix-boundedness}.
\end{proposition}

\begin{proof}
Write
$$
 \mathcal N=\|C_{B,T}\|_{L^p(U)\to L^p(V)}.
$$
Fix a cube $Q$.  Let $Q^\sharp$, $\phi_Q$ and  $\varepsilon$ be as
in Lemma~\ref{lem:median-companion-sector}.  For
$1\leq s,t\leq m$, put
\begin{equation}
 a_s=\mathcal U_Q^{-1}e_s,
 \qquad
 c_t=\mathcal V_Qe_t,
 \qquad
 z_{st,Q}(x)=c_t^*B(x)a_s.
 \label{eq:median-scalarization}
\end{equation}
Fix $s,t$ and write $z=z_{st,Q}$.  All quadrants, sets and  test
functions introduced below depend on this pair.  We suppress this
dependence in the notation.  Apply
Lemma~\ref{lem:complex-median-WZ} to $z$ on $Q^\sharp$.  Denote the
median center by $c=c_{st,Q}$.  After one rotation, write the four quadrants as
$$
 \Gamma_j
 =c+\{re^{i\theta}:r\geq0,
       \ \theta_j\leq\theta\leq\theta_j+\pi/2\},
 \qquad
 \theta_{j+1}=\theta_j+\pi/2.
$$
Set
\begin{equation}
 F_j=\{y\in Q^\sharp:z(y)\in\Gamma_j\}.
 \label{eq:median-Fj}
\end{equation}
Then
\begin{equation}
 |F_j|\geq |Q|/16.
 \label{eq:median-Fj-size}
\end{equation}
The sets $F_j$ need not be disjoint.  This permits atoms on the
quadrant boundaries.

Choose a measurable disjoint partition
\begin{equation}
 Q=E_1\mathbin{\dot\cup}E_2\mathbin{\dot\cup}E_3
     \mathbin{\dot\cup}E_4
 \label{eq:median-Ej-partition}
\end{equation}
such that
\begin{equation}
 E_j\subset\{x\in Q:z(x)\in\Gamma_{j+2}\}.
 \label{eq:median-Ej-opposite}
\end{equation}
Such a partition is obtained by assigning each boundary point to the
first admissible index.  If $x\in E_j$ and $y\in F_j$, then
$z(y)-z(x)$ lies in the closed sector
\begin{equation}
 \{re^{i\theta}:r\geq0,
       \ \theta_j\leq\theta\leq\theta_j+\pi/2\}.
 \label{eq:median-difference-sector}
\end{equation}
Indeed, after translating by $c$ and rotating by $e^{-i\theta_j}$,
the two coordinates of $z(y)$ are nonnegative and the two coordinates
of $z(x)$ are nonpositive.  The same observation gives
\begin{equation}
 |z(x)-c|\leq |z(y)-z(x)|.
 \label{eq:median-distance-dominated}
\end{equation}

Define
\begin{equation}
 \xi_{Q,j}=e^{-i(\phi_Q+\theta_j+\pi/4)}.
 \label{eq:median-phase}
\end{equation}
Equations \eqref{eq:median-kernel-lower},
\eqref{eq:median-kernel-sector} and 
\eqref{eq:median-difference-sector} imply
\begin{align}
 \operatorname{Re}\!\left\{
   \xi_{Q,j}K(y,x)[z(y)-z(x)]\right\}&\geq
 \cos(\pi/4+\varepsilon)
 |K(y,x)|\,|z(y)-z(x)|\notag\\
 &\geq
 \frac{c_A\cos(\pi/4+\varepsilon)}{|Q|}|z(x)-c|
 \label{eq:median-positive-product}
\end{align}
for $x\in E_j$ and $y\in F_j$.  Integrating first in $y$ and using
\eqref{eq:median-Fj-size} gives
\begin{equation}
 \int_{E_j}|z(x)-c|\,dx
 \lesssim_A
 \left|\int_{F_j}\int_{E_j}
 K(y,x)[z(y)-z(x)]\,dx\,dy\right|.
 \label{eq:median-sector-integral}
\end{equation}
The phase in \eqref{eq:median-phase} has modulus one and disappears
after the absolute value is taken.

Since
$$
 \avg_Q|z-z_Q|
 \leq \avg_Q|z-c|+|z_Q-c|
 \leq2\avg_Q|z-c|,
$$
equations \eqref{eq:median-Ej-partition} and
\eqref{eq:median-sector-integral} yield
\begin{equation}
 \avg_Q|z-z_Q|
 \lesssim_A
 \sum_{j=1}^4\frac1{|Q|}
 \left|\int_{F_j}\int_{E_j}
 K(y,x)[z(y)-z(x)]\,dx\,dy\right|.
 \label{eq:median-oscillation-by-pairings}
\end{equation}

For each $j$, define the normalized vector tests
\begin{equation}
 f_{j,s}=|Q|^{-1/p}\mathbf1_{E_j}\mathcal U_Q^{-1}e_s,
 \qquad
 g_{j,t}=|Q|^{-1/p'}\mathbf1_{F_j}\mathcal V_Qe_t.
 \label{eq:median-normalized-tests}
\end{equation}
Their supports are separated.  The off-support commutator formula and
the convention that the pairing is linear in its first variable give
the identity
\begin{align}
 \langle C_{B,T}f_{j,s},g_{j,t}\rangle
 &=\frac1{|Q|}\int_{F_j}\int_{E_j}
   \big\langle(B(y)-B(x))K(y,x)a_s,c_t\big\rangle
   \,dx\,dy\notag\\
 &=\frac1{|Q|}\int_{F_j}\int_{E_j}
   K(y,x)[z_{st,Q}(y)-z_{st,Q}(x)]\,dx\,dy.
 \label{eq:median-exact-commutator-pairing}
\end{align}
This also fixes the order of the two kernel variables.  The integration
variable lies in $Q$ and  the value of the commutator is taken on
$Q^\sharp$.

The reducing property on $Q$ gives
\begin{align}
 \|f_{j,s}\|_{L^p(U)}
 &\leq |Q|^{-1/p}
 \bigg(\int_Q
 |U^{1/p}(x)\mathcal U_Q^{-1}e_s|^p\,dx\bigg)^{1/p}
 \lesssim 1.
 \label{eq:median-source-norm}
\end{align}
For the second test, put $V^\#=V^{-p'/p}$.  Then
$(V^\#)^{1/p'}=V^{-1/p}$ and  $\mathcal V'_Q$ is an
exponent-$p'$ reducing matrix for $V^\#$ on $Q$.  Apply
Lemma~\ref{lem:local-matrix-weight-bdd}(ii), with $q=p'$ and
$W=V^\#$, to the pair $Q,Q^\sharp$ and the vector
$$
 e=\mathcal V'_Q\mathcal V_Qe_t.
$$
The reducing-matrix form of the matrix $\mathcal A_p$ condition gives
\begin{equation}
 \|\mathcal V'_Q\mathcal V_Q\|_{\op}
 \lesssim_{m,p} [V]_{\mathcal A_p}^{1/p}.
 \label{eq:median-primal-dual-reducing-product}
\end{equation}
See \cite[Proposition~2.2 and formula~(2.6)]{CUIM}.  Therefore
\begin{align}
 \|g_{j,t}\|_{L^{p'}(V^\#)}
 &\leq |Q|^{-1/p'}
 \bigg(\int_{Q^\sharp}
 |V^{-1/p}(x)\mathcal V_Qe_t|^{p'}\,dx\bigg)^{1/p'}
 \lesssim_A 1.
 \label{eq:median-target-norm}
\end{align}
Here \eqref{eq:median-common-ancestor} supplies the cube $S$ required
in Lemma~\ref{lem:local-matrix-weight-bdd}(ii).

Equations \eqref{eq:median-oscillation-by-pairings}--
\eqref{eq:median-target-norm} now give
\begin{equation}
 \avg_Q|z_{st,Q}-(z_{st,Q})_Q|
 \lesssim_A \mathcal N,
 \qquad 1\leq s,t\leq m.
 \label{eq:median-scalar-entry-lower}
\end{equation}
Finally,
\begin{align*}
 \avg_Q
 \|\mathcal V_Q(B-B_Q)\mathcal U_Q^{-1}\|_{\op}\leq
 \sum_{s,t=1}^m\avg_Q
 |e_t^*\mathcal V_Q(B-B_Q)\mathcal U_Q^{-1}e_s|=
 \sum_{s,t=1}^m
 \avg_Q|z_{st,Q}-(z_{st,Q})_Q|
 \lesssim_A \mathcal N.
\end{align*}
This proves \eqref{eq:weak-reducing-lower}.

The global matrix John--Nirenberg equivalences of
Isralowitz--Pott--Treil \cite[Corollary~4.7]{IPT} give
\begin{equation}
 \sup_Q\alpha_Q(B)
 \simeq
 \sup_Q\avg_Q
 \|\mathcal V_Q(B-B_Q)\mathcal U_Q^{-1}\|_{\op}
 \simeq
 \sup_Q\beta_Q(B).
 \label{eq:median-IPT-global-equivalence}
\end{equation}
Combining \eqref{eq:weak-reducing-lower} with
\eqref{eq:median-IPT-global-equivalence} proves
\eqref{eq:alpha-beta-general-lower}.
\end{proof}
The preceding proposition is enough for boundedness.  Compactness requires
tests that retain the size and location of each cube.  We now build
such tests.  The first inequality in
\eqref{eq:two-sided-nondegeneracy} is used in the next lemma; the
second will be used after taking the adjoint.

\begin{lemma}
\label{lem:reciprocal-tests}
Let $U,V\in\mathcal A_p$ and  let $B$ be locally integrable.  Assume
that the first inequality in
\eqref{eq:two-sided-nondegeneracy} holds.
There are fixed constants $A\geq4\sqrt d$ and $C_A$ with the following
property.  For every cube $Q$, there is a cube $Q^\sharp$ of the same
side length such that
\begin{equation}\label{eq:reciprocal-companion-geometry}
 A\ell(Q)\leq\dist(Q,Q^\sharp)\leq C_A\ell(Q).
\end{equation}
Let $M\geq2$ be an integer and partition $Q$ into $M^d$ congruent cubes
$R\in\mathscr R_M(Q)$.  If $x_R$ is the center of $R$, then
\begin{equation}\label{eq:reciprocal-multiplier}
 \psi_{Q,R}(y)=
 \frac{\one_{Q^\sharp}(y)}{|Q|K(x_R,y)}
\end{equation}
is well defined and $\|\psi_{Q,R}\|_\infty\leq C_A$.  For
$1\leq s\leq m$, set
\begin{equation}\label{eq:reciprocal-test-function}
 g_{Q,R,s}(y)=|Q|^{-1/p}\psi_{Q,R}(y)
 \mathcal U_Q^{-1}e_s.
\end{equation}
Then
\begin{equation}\label{eq:reciprocal-test-bound}
 \sup_{Q,R,s}\|g_{Q,R,s}\|_{L^p(U)}\leq C_{A,U}.
\end{equation}
If $B\in\BMO^p_{V,U}$ and $S_Q$ is a cube containing
$Q\cup Q^\sharp$ with $|S_Q|\leq C_A|Q|$, then
\begin{align}
 \alpha_Q(B)
 &\leq C\sum_{s=1}^m
 \bigg(\sum_{R\in\mathscr R_M(Q)}
 \|\one_R C_{B,T}g_{Q,R,s}\|_{L^p(V)}^p\bigg)^{1/p}
 +C_A M^{-\delta}\omega_{S_Q}(B).             \label{eq:localized-reciprocal-alpha}
\end{align}
\end{lemma}

\begin{proof}
Apply the first non-degeneracy condition at the center $x_Q$ and at
radius $2A\ell(Q)$.  The size estimate bounds the resulting distance
from above by a fixed multiple of $A\ell(Q)$.  Take $Q^\sharp$ to be
the cube centered at the resulting point.  Since the radius is
$2A\ell(Q)$ and $A\geq4\sqrt d$, the two cubes are separated by at
least $A\ell(Q)$.  First fix $A$ so large that
the H\"older variation of the kernel on $Q\times Q^\sharp$ is at most
one half of its value at the two centers.  Thus
\begin{equation}\label{eq:kernel-lower-product-cubes}
 |K(x,y)|\geq c_A|Q|^{-1},
 \qquad x\in Q,\quad y\in Q^\sharp.
\end{equation}
This proves the geometry and the uniform bound for
\eqref{eq:reciprocal-multiplier}.  The same regularity estimate gives
\begin{equation}\label{eq:kernel-ratio-error}
 \sup_{\substack{x\in R\\y\in Q^\sharp}}
 \left|\frac{K(x,y)}{K(x_R,y)}-1\right|
 \leq C_A M^{-\delta}=:\eta_M.
\end{equation}
Lemma~\ref{lem:local-matrix-weight-bdd}(ii), applied to $U$, gives
\begin{align*}
 \|g_{Q,R,s}\|_{L^p(U)}
 &\leq C_A|Q|^{-1/p}
 \bigg(\int_{Q^\sharp}
 |U^{1/p}(y)\mathcal U_Q^{-1}e_s|^p\,dy\bigg)^{1/p}
 \leq C_{A,U}.
\end{align*}

For $x\in R$, the separated-kernel identity gives
\begin{equation}\label{eq:reciprocal-identity}
 C_{B,T}g_{Q,R,s}(x)
 =|Q|^{-1/p}\avg_{Q^\sharp}
 \frac{K(x,y)}{K(x_R,y)}
 (B(x)-B(y))\mathcal U_Q^{-1}e_s\,dy.
\end{equation}
The term obtained by replacing the quotient by one is
$$
 |Q|^{-1/p}(B(x)-B_{Q^\sharp})\mathcal U_Q^{-1}e_s.
$$
Denote the remaining average by $|Q|^{-1/p}E_{Q,R,s}(x)$.  Insert
$U^{-1/p}(y)U^{1/p}(y)$, apply H\"older's inequality in $y$ and  use
\eqref{eq:kernel-ratio-error}.  For $x\in R$ this gives
\begin{align*}
 |V^{1/p}(x)E_{Q,R,s}(x)|
 &\leq\eta_M
 \bigg(\avg_{Q^\sharp}
 \|V^{1/p}(x)(B(x)-B(y))U^{-1/p}(y)\|_{\op}^{p'}dy
 \bigg)^{1/p'}
 \bigg(\avg_{Q^\sharp}
 |U^{1/p}(y)\mathcal U_Q^{-1}e_s|^pdy\bigg)^{1/p}.
\end{align*}
The last factor is bounded by the companion-cube reducing estimate.
Take the $L^p$ average in $x$ over all of $Q$, not separately over
the individual cubes $R$.  Enlarging both normalized averages to
$S_Q$ costs only a constant.  Therefore
\begin{equation}\label{eq:reciprocal-error-bound}
 \bigg(\avg_Q\sum_{R\in\mathscr R_M(Q)}
 \one_R(x)|V^{1/p}(x)E_{Q,R,s}(x)|^pdx\bigg)^{1/p}
 \leq C_A M^{-\delta}\widehat\omega_{S_Q}(B)
 \lesssim C_A M^{-\delta}\omega_{S_Q}(B).
\end{equation}
Apply Lemma~\ref{lem:change-center-bdd} with
$A=B_{Q^\sharp}$, estimate the operator norm by its values on the
standard basis and  use \eqref{eq:reciprocal-identity} and
\eqref{eq:reciprocal-error-bound}.  This is
\eqref{eq:localized-reciprocal-alpha}.
\end{proof}

\begin{proposition}\label{prop:two-matrix-lower}
Under the assumptions of Theorem~\ref{thm:two-matrix-boundedness}, if
$C_{B,T}:L^p(U)\to L^p(V)$ is bounded, then
\begin{equation}\label{eq:two-matrix-lower}
 \sup_Q\omega_Q(B)
 \leq C\|C_{B,T}\|_{L^p(U)\to L^p(V)}.
\end{equation}
\end{proposition}

\begin{proof}
This is \eqref{eq:alpha-beta-general-lower}.  This global lower estimate
uses only \eqref{eq:source-side-nondegeneracy}.  Both
orientations are retained in the compactness theorem because
the localized compactness argument must recover the primal and dual
tails separately.
\end{proof}

Propositions~\ref{prop:two-matrix-upper} and
\ref{prop:two-matrix-lower} prove
Theorem~\ref{thm:two-matrix-boundedness}.  The upper estimate was
already available for all Calder\'on--Zygmund operators in
\cite{IPT}.  The lower estimate shows that one source-side
non-degenerate variable kernel suffices.

\section{The compactness theorem}

\begin{definition}\label{def:vmo}
We write $B\in\VMO^p_{V,U}$ if $B\in\BMO^p_{V,U}$ and
\begin{align}
 \lim_{r\downarrow0}\ \sup_{\ell(Q)\le r}\omega_Q(B)&=0,
                                                        \label{eq:vmosmall}\\
 \lim_{R\uparrow\infty}\ \sup_{\ell(Q)\ge R}\omega_Q(B)&=0,
                                                        \label{eq:vmolarge}\\
 \lim_{R\uparrow\infty}\ \sup_{\dist(Q,0)\ge R}\omega_Q(B)&=0.
                                                        \label{eq:vmofar}
\end{align}
\end{definition}

Set
\begin{align*}
 \delta_{\rm s}(B)&=\lim_{r\downarrow0}
   \sup_{\ell(Q)\le r}\omega_Q(B),\quad
 \delta_{\rm l}(B)=\lim_{R\uparrow\infty}
   \sup_{\ell(Q)\ge R}\omega_Q(B),\quad
 \delta_{\rm f}(B)=\lim_{R\uparrow\infty}
   \sup_{\dist(Q,0)\ge R}\omega_Q(B),\\
 \mathcal T_{V,U,p}(B)&=\max\{\delta_{\rm s}(B),\delta_{\rm l}(B),
 \delta_{\rm f}(B)\}.
\end{align*}

\begin{theorem}
\label{thm:main}
Let $1<p<\infty$, let $U,V$ be matrix $\mathcal A_p$ weights, let
$B:\R^d\to\C^{m\times m}$ be locally integrable and  let $T$ be an
$L^2$-bounded Calder\'on--Zygmund operator with a $\delta$-standard
scalar kernel.  Assume that the kernel is two-sided non-degenerate in
the sense of Definition~\ref{def:two-sided-nondegeneracy}.  Then
$$
 C_{B,T}:L^p(U;\C^m)\longrightarrow L^p(V;\C^m)
 \quad\text{is compact}
$$
if and only if $B\in\VMO^p_{V,U}$.  Moreover,
\begin{equation}
 \|C_{B,T}\|_{\ess} \simeq \mathcal T_{V,U,p}(B),          \label{eq:essential}
\end{equation}
whenever $B\in\BMO^p_{V,U}$.  Here $\|\cdot\|_{\ess}$ denotes distance
to the compact operators from $L^p(U;\C^m)$ to $L^p(V;\C^m)$.
The upper essential-norm estimate and  hence the implication
$B\in\VMO^p_{V,U}\Rightarrow C_{B,T}$ compact, does not use
non-degeneracy.
\end{theorem}

Sections~\ref{sec:upper} and~\ref{sec:lower} prove the two essential
norm bounds and the compactness equivalence.

\subsection{The upper essential norm}
\label{sec:upper}

Throughout this section, $1<p<\infty$, $U,V\in\mathcal A_p$ are
$m\times m$ matrix weights and  $T$ is a scalar
Calder\'on--Zygmund operator acting componentwise.  Put
$$
 X=L^p(U;\mathbb C^m),\qquad
 Y=L^p(V;\mathbb C^m),\qquad
 V^\#=V^{-p'/p}.
$$
We use only the size and regularity estimates for the kernel of $T$.
No non-degeneracy assumption is used.
The H\"older modulus in \eqref{eq:kernel-holder} satisfies the Dini
condition used in the convex-body sparse theorem below.

A family $\mathcal S$ in a dyadic grid $\mathcal D$ is
$\eta$-sparse if there are pairwise disjoint measurable sets
$E_Q\subset Q$, $Q\in\mathcal S$, such that $|E_Q|\geq\eta|Q|$.
For a matrix weight $W$, put
\begin{equation}
 [W]_{\mathcal A_p(\mathcal D)}
 :=\sup_{Q\in\mathcal D}\avg_Q
 \bigg(\avg_Q
 \|W^{1/p}(x)W^{-1/p}(y)\|_{\op}^{p'}\,dy
 \bigg)^{p/p'}dx.                                  \label{eq:dyadic-matrix-Ap-sparse}
\end{equation}
Thus $[W]_{\mathcal A_p(\mathcal D)}\leq[W]_{\mathcal A_p}$ for
every grid.

\subsubsection{An central decomposition}

For $N\geq1$, let
$
 Q_N=[-2^N,2^N]^d.
$
Let $\mathcal P_N$ be the partition of $Q_N$, up to its boundary, by
half-open cubes $P$ of side length $2^{-N}$.  All cubes $3P$ and
$3Q_N$ below are concentric dilates.  Define
\begin{equation}
 K_Nf=\sum_{P\in\mathcal P_N}
 \one_P C_{B,T}\big(\one_{3Q_N\setminus3P}f\big).
 \label{eq:CLLV-compact-part}
\end{equation}
This is a finite sum.  Define the three remainder operators
\begin{align}
 R_N^{\rm out}f
 &=\one_{Q_N^c}C_{B,T}f,                              \label{eq:CLLV-target-remainder}\\
 R_N^{\rm loc}f
 &=\sum_{P\in\mathcal P_N}
      \one_PC_{B,T}(\one_{3P}f),                     \label{eq:CLLV-local-remainder}\\
 R_N^{\rm rem}f
 &=\one_{Q_N}C_{B,T}(\one_{(3Q_N)^c}f).              \label{eq:CLLV-remote-remainder}
\end{align}
For $x\in P$, the sets
$
 3P,\ 3Q_N\setminus3P,\ (3Q_N)^c
$
partition $\mathbb R^d$, up to boundaries; here $3P\subset3Q_N$.
Therefore
\begin{equation}
 C_{B,T}-K_N
 =R_N^{\rm out}+R_N^{\rm loc}+R_N^{\rm rem}.         \label{eq:CLLV-exact-decomposition}
\end{equation}
This is an operator identity on the form domain.  The estimates
below extend it to $X$.

\begin{lemma}
\label{lem:compact-separated-commutator}
Let $E,F\subset\mathbb R^d$ be bounded measurable sets with
$\dist(E,F)>0$.  If $B\in\BMO^p_{V,U}$, then
$$
 \one_E C_{B,T}\one_F:X\longrightarrow Y
$$
is compact.
\end{lemma}

\begin{proof}
The operator has kernel
$$
 a(x,y)(B(x)-B(y)),\qquad
 a(x,y)=\one_E(x)K(x,y)\one_F(y).
$$
The function $a$ is bounded and supported on a bounded product set.
Choose a cube $Q$ containing $E\cup F$ and  set
$$
 H(x,y)=V^{1/p}(x)(B(x)-B(y))U^{-1/p}(y).
$$
Lemma~\ref{lem:localcomparison} gives
\begin{equation}
 \int_Q\bigg(\int_Q\|H(x,y)\|_{\op}^{p'}\,dy
       \bigg)^{p/p'}dx
 =|Q|^p\widehat\omega_Q(B)^p<\infty.                 \label{eq:separated-mixed-integrability}
\end{equation}

If $c$ is a bounded scalar kernel supported on $Q\times Q$, then
H\"older's inequality in $y$ gives
\begin{align}
 &\left\|\int_Qc(\,\cdot\,,y)(B(\,\cdot\,)-B(y))f(y)\,dy
       \right\|_Y \nonumber\\
 &\quad\leq
 \left\{\int_Q\left[
       \int_Q|c(x,y)|^{p'}\|H(x,y)\|_{\op}^{p'}\,dy
                    \right]^{p/p'}dx\right\}^{1/p}\|f\|_X.
                                                               \label{eq:separated-mixed-operator-bound}
\end{align}
Indeed, write $f=U^{-1/p}h$, apply H\"older's inequality and  take
the unweighted $L^p$ norm in $x$.

Partition $Q$ into its dyadic subcubes of side $2^{-j}\ell(Q)$.  Let $a_j$ be
the conditional expectation of $a$ with respect to the product
partition.  Then
$$
 a_j(x,y)=\sum_{\nu=1}^{M_j}c_{j,\nu}
      \one_{E_{j,\nu}}(x)\one_{F_{j,\nu}}(y),
 \qquad M_j<\infty,
$$
$|a_j|\leq\|a\|_\infty$ and  $a_j\to a$ almost everywhere.  Apply
dominated convergence first in $y$ and then in $x$ in
\eqref{eq:separated-mixed-operator-bound}.  Equation
\eqref{eq:separated-mixed-integrability} supplies the dominating
function.  The corresponding operators converge in norm.

It remains to check that each rectangle kernel has finite rank.  Its
action is a scalar multiple of
$$
 f\longmapsto
 \one_{E_{j,\nu}}(x)
 \left\{B(x)\int_{F_{j,\nu}}f(y)\,dy
              -\int_{F_{j,\nu}}B(y)f(y)\,dy\right\}.
$$
Its range is contained in the span of
$$
 \{\one_{E_{j,\nu}}e_s,\ \one_{E_{j,\nu}}Be_s:
                         1\leq s\leq m\}.
$$
These vectors belong to $Y$.  This follows from $\alpha_Q(B)$ and
the local integrability of $V$: apply the defining estimate to
$\mathcal U_Qe_s$ and then add the constant vector $B_Qe_s$.
The two displayed integrals define
bounded functionals on $X$.  The first follows from
$\one_Qe_s\in L^{p'}(U^\#)$.  For the second, use $\beta_Q(B)$ to
apply the defining estimate to $\mathcal V'_Qe_s$ to obtain
$$
 \one_Q(B^*-B_Q^*)e_s\in L^{p'}(U^\#),
$$
and treat $B_Q^*e_s$ as a constant vector.  Thus the rectangle
operators have finite rank.  Their norm limit is compact.
\end{proof}

\begin{corollary}
\label{cor:CLLV-compact-approximant}
For every $N$, the operator $K_N$ in
\eqref{eq:CLLV-compact-part} is compact from $X$ to $Y$.
\end{corollary}

\begin{proof}
For $P\in\mathcal P_N$, the sets $P$ and
$3Q_N\setminus3P$ are bounded and satisfy
$$
 \dist(P,3Q_N\setminus3P)\geq\ell(P).
$$
Lemma~\ref{lem:compact-separated-commutator} applies to every summand
in \eqref{eq:CLLV-compact-part}.  The sum is finite.
\end{proof}

\subsubsection{Global and prescribed-top convex-body domination}

For an integrable $\mathbb C^n$-valued function $F$, define
$$
 \langle\!\langle F\rangle\!\rangle_Q
 =\bigg\{\avg_Q\varphi(y)F(y)\,dy:
       \varphi\text{ is scalar and }\|\varphi\|_\infty\leq1\bigg\}.
$$
This is a compact, symmetric, convex subset of $\mathbb C^n$.
Compactness follows from weak-star compactness of the closed unit
ball of $L^\infty(Q)$.

For a scalar function $h$, let
$$
 Mh(x)=\sup_{Q\ni x}\avg_Q|h(y)|\,dy.
$$
For $r>0$, put
$$
 T_rh(x)=\int_{|x-y|>r}K(x,y)h(y)\,dy,
 \qquad T_*h(x)=\sup_{r>0}|T_rh(x)|.
$$
Also put
$$
 \mathfrak M_T h(x)
 =\sup_{Q\ni x}\ \mathop{\rm ess\,sup}_{\xi\in Q}
   |T(\one_{\mathbb R^d\setminus3Q}h)(\xi)|.
$$
The operators $T$, $T_*$ and  $M$ are of weak type $(1,1)$.  We now
verify the same assertion for $\mathfrak M_T$.  If
$x,\xi\in Q$, choose a fixed $A_d>10\sqrt d$ and put $S=A_dQ$.
The size estimate gives
$$
 \left|\int_{S\setminus3Q}K(\xi,y)h(y)\,dy\right|
 \leq C_{d,T}Mh(x).
$$
On $S^c$, decompose into the annuli
$2^{k+1}S\setminus2^kS$.  If $\omega$ is the kernel modulus of
continuity, the regularity estimate gives
\begin{align*}
 \int_{S^c}|K(\xi,y)-K(x,y)|\,|h(y)|\,dy
 &\leq C_{d,T}\sum_{k=0}^\infty\omega(c_d2^{-k})
 \avg_{2^{k+1}S}|h(y)|\,dy
 \leq C_{d,T}Mh(x).
\end{align*}
The last series is finite by the Dini condition.  For the kernel in
this paper one may take $\omega(t)=t^\delta$.
Finally, choose $r \simeq_d A_d\ell(Q)$ so that the symmetric
difference of $S^c$ and $\{y:|x-y|>r\}$ lies in an annulus centered
at $x$ whose two radii are comparable.  The size estimate on this
annulus gives
$$
 |T(\one_{S^c}h)(x)-T_rh(x)|\leq C_{d,T}Mh(x).
$$
Consequently,
$$
 |T(\one_{\mathbb R^d\setminus3Q}h)(\xi)|
 \leq C_T\{T_*h(x)+Mh(x)\}.
$$
Take the two suprema.  The weak $(1,1)$ estimates for $T_*$ and $M$
prove the assertion.  This is also the grand maximal hypothesis in
\cite[Lemma~13]{IRRS}.

\begin{lemma}
\label{lem:global-local-convex-body}
Let $T$ have a Dini-continuous Calder\'on--Zygmund kernel.  The
following assertions hold in every finite dimension $n$.

\begin{enumerate}
\item For every compactly supported $F\in L^1(\mathbb R^d;\mathbb C^n)$,
there are at most $L_d$ dyadic grids $\mathcal D^t$
and $\eta_d$-sparse families $\mathcal S^t\subset\mathcal D^t$ such
that
\begin{equation}
 (T\otimes I_n)F(x)
 \in C_T\sum_{t=1}^{L_d}\sum_{Q\in\mathcal S^t}
      \one_Q(x)\langle\!\langle F\rangle\!\rangle_Q
 \quad\text{for almost every }x.                    \label{eq:global-convex-body}
\end{equation}

\item Let $P$ be a cube and let
$F\in L^1(\mathbb R^d;\mathbb C^n)$ be supported in $3P$.  There are
at most $L_d$ dyadic grids and sparse families
$\mathcal S_P^t$ such that, for almost every $x\in P$,
\begin{equation}
 (T\otimes I_n)F(x)
 \in C_T\sum_{t=1}^{L_d}\sum_{Q\in\mathcal S_P^t}
       \one_Q(x)\langle\!\langle F\rangle\!\rangle_Q,           \label{eq:local-convex-body}
\end{equation}
and every cube in these families satisfies
\begin{equation}
 \ell(Q)\leq C_d\ell(P).                              \label{eq:prescribed-top-length}
\end{equation}
 The constants are independent of $P$ and $F$.  They may depend on
 $d$ and $n$.
 If the cubes $P$ belong to one fixed grid, the finite collection of
 adjacent grids may be chosen once for all of them.
\end{enumerate}
\end{lemma}

\begin{proof}
The global assertion is the convex-body sparse theorem
\cite[Theorem~3.4]{NPTV}; see also
\cite[Theorem~13]{IRRS}.  We explain why the same stopping argument
gives the second assertion with the stated length bound.  We work in
$\mathbb C^n$, viewed as a real vector space.  Complex scalar
multipliers give the support-function identity below.

We give the local stopping argument.  Choose a dyadic grid containing
$P$ and fix a dyadic descendant
$Q\subset P$ and put
$$
 \mathcal K_Q=\langle\!\langle F\rangle\!\rangle_{3Q}.
$$
John's ellipsoid theorem in the real linear span of $\mathcal K_Q$
gives vectors $u_1,\ldots,u_\rho\in\mathbb C^n$, with
$\rho\leq2n$, such that the following implication holds.  If
$$
 h_j(y)=\langle F(y),u_j\rangle,\qquad
 \lambda_j=\avg_{3Q}|h_j(y)|\,dy,
$$
and $z$ belongs to that real linear span, then
\begin{equation}
 |\operatorname {Re}\langle z,u_j\rangle|
 \leq A\lambda_j\quad(1\leq j\leq\rho)
 \quad\Longrightarrow\quad
 z\in C_nA\mathcal K_Q.                              \label{eq:John-coordinate-body}
\end{equation}
If $\mathcal K_Q$ has empty interior, John's theorem is applied in
its real linear span.  The support function identity
$$
 \sup_{z\in\mathcal K_Q}
 |\operatorname {Re}\langle z,u\rangle|
 =\avg_{3Q}|\langle F(y),u\rangle|\,dy
$$
proves \eqref{eq:John-coordinate-body}.

Define
\begin{align*}
 \Omega_Q
 =\bigcup_{\substack{1\leq j\leq\rho\\\lambda_j>0}}\Big(
 \{x\in Q:|T(\one_{3Q}h_j)(x)|>A\lambda_j\}
 &{}\cup{}
 \{x\in Q:\mathfrak M_T(\one_{3Q}h_j)(x)>A\lambda_j\}\\
 &{}\cup{}
 \{x\in Q:M(\one_{3Q}h_j)(x)>A\lambda_j\}\Big).
\end{align*}
The three weak $(1,1)$ estimates give
$$
 |\Omega_Q|\leq C_{d,n,T}A^{-1}|3Q|.
$$
Fix $A=A(d,n,T)$ so large that
\begin{equation}
 |\Omega_Q|\leq2^{-d-2}|Q|.                          \label{eq:local-bad-small}
\end{equation}
Let $\{Q_s\}_s$ be the maximal dyadic subcubes of $Q$ satisfying
$$
 |Q_s\cap\Omega_Q|>2^{-d-1}|Q_s|.
$$
They cover $\Omega_Q$ up to a null set, are pairwise disjoint and 
\begin{equation}
 \sum_s|Q_s|\leq2^{d+1}|\Omega_Q|
 \leq\tfrac12|Q|.                                   \label{eq:local-children-half}
\end{equation}
The cube $Q$ itself is not selected by
\eqref{eq:local-bad-small}.  Hence the dyadic parent
$\widehat Q_s$ lies in $Q$.  Maximality gives a point
$z_s\in\widehat Q_s\setminus\Omega_Q$ at which all relevant
estimates hold.

If $A_0\subset B_0$ are cubes, then, for every coordinate $i$,
$$
 |(c_{A_0})_i-(c_{B_0})_i|+\tfrac32\ell(A_0)
 \leq\tfrac12\{\ell(B_0)-\ell(A_0)\}
      +\tfrac32\ell(A_0)
 \leq\tfrac32\ell(B_0).
$$
Thus $3A_0\subset3B_0$.  In particular,
$3Q_s\subset3\widehat Q_s\subset3Q$.

For almost every $x\in Q_s$,
\begin{align*}
 T(\one_{3Q\setminus3Q_s}h_j)(x)
 &=T(\one_{3Q\setminus3\widehat Q_s}h_j)(x)
   +T(\one_{3\widehat Q_s\setminus3Q_s}h_j)(x).
\end{align*}
The first term is bounded by
$\mathfrak M_T(\one_{3Q}h_j)(z_s)$.  Since
$3\widehat Q_s\subset3Q$ and
$$
 \dist(Q_s,3\widehat Q_s\setminus3Q_s)
 \geq c_d\ell(Q_s),
$$
the kernel size estimate bounds the second term by
$C_{d,T}M(\one_{3Q}h_j)(z_s)$.  Thus
\begin{equation}
 |T(\one_{3Q\setminus3Q_s}h_j)(x)|
 \leq C_{d,T}A\lambda_j.                             \label{eq:local-child-coordinate}
\end{equation}

Define
\begin{align*}
 G_Q(x)
 &=\one_Q(x)(T\otimes I_n)(\one_{3Q}F)(x)-\sum_s\one_{Q_s}(x)
        (T\otimes I_n)(\one_{3Q_s}F)(x).
\end{align*}
The vector $G_Q(x)$ lies in
$V_Q=\operatorname{span}_{\mathbb R}\mathcal K_Q$.  Indeed,
$\mathcal K_Q$ is balanced under complex scalars, so $V_Q$ is a
complex subspace.  If $u$ annihilates $V_Q$, then
$\langle F,u\rangle=0$ almost everywhere on $3Q$ and  complex
linearity gives
$$
 \langle(T\otimes I_n)(\one_EF),u\rangle
 =T(\one_E\langle F,u\rangle)=0
$$
for $E=3Q$ and $E=3Q_s$.  This proves the claim.  Also,
$\lambda_j=0$ implies $h_j=0$ almost everywhere, so the corresponding
coordinate of $G_Q$ is zero.

Outside $\bigcup_sQ_s$, the point $x$ is outside $\Omega_Q$ for
almost every $x$, so the first condition in the definition of
$\Omega_Q$ applies.  On $Q_s$, use
\eqref{eq:local-child-coordinate}.  Applying
\eqref{eq:John-coordinate-body} in both cases gives the one-step
inclusion
\begin{equation}
 G_Q(x)\in C_{d,n,T}\one_Q(x)\mathcal K_Q
 \quad\text{for almost every }x.                    \label{eq:local-one-step-body}
\end{equation}
This proves directly the prescribed-top version of
\cite[Lemma~3.2, with $r=3$ and 
Sections~3.2.1--3.2.2]{NPTV}; see also
\cite[Lemma~13]{IRRS}.  It requires only
$F\in L^1(3Q;\mathbb C^n)$.

Start with $Q=P$ and repeat \eqref{eq:local-one-step-body} on every
selected child.  Let $\mathcal G_k$ be the $k$th generation, with
$\mathcal G_0=\{P\}$.  After $k$ steps, the one-step inclusions give
\begin{align}
 \one_P(T\otimes I_n)(\one_{3P}F)(x)\in
 C_{d,n,T}\sum_{j=0}^{k-1}\sum_{Q\in\mathcal G_j}
   \one_Q(x)\mathcal K_Q
 +\sum_{Q\in\mathcal G_k}\one_Q(x)
   (T\otimes I_n)(\one_{3Q}F)(x).                  \label{eq:local-finite-generation}
\end{align}
By \eqref{eq:local-children-half},
$$
 \bigg|\bigcup_{Q\in\mathcal G_k}Q\bigg|
 \leq2^{-k}|P|.
$$
These unions decrease.  Their intersection has measure zero.
Discard also the countable union of the exceptional null sets in the
one-step inclusions.  For every remaining $x$, the last sum in
\eqref{eq:local-finite-generation} is zero for all sufficiently
large $k$.  Thus no limiting estimate for the remainder is needed.
We obtain a
$\tfrac12$-sparse family $\mathcal G(P)$ of dyadic descendants of
$P$ and the inclusion
\begin{equation}
 (T\otimes I_n)F(x)
 \in C_T\sum_{Q\in\mathcal G(P)}
       \one_Q(x)\langle\!\langle F\rangle\!\rangle_{3Q},
 \qquad x\in P.                                      \label{eq:local-3Q-inclusion}
\end{equation}
Indeed, the sets
$$
 E_Q=Q\setminus
 \bigcup_{\substack{Q_s\text{ selected}\\\text{from }Q}}Q_s
$$
are pairwise disjoint and satisfy $|E_Q|\geq|Q|/2$.  This also proves
the asserted sparseness without any boundedness assumption on $F$.

Use the three-lattice lemma.  There are $3^d$ dyadic grids
$\mathcal D^{(a)}$, depending only on the original grid and  a map
$a:\mathcal D\to\{1,\ldots,3^d\}$ such that
$$
 3Q\in\mathcal D^{(a(Q))}\qquad(Q\in\mathcal D).
$$
For $R=3Q$, enlarge $\one_Q$ to $\one_R$.  The convex body is already
$\langle\!\langle F\rangle\!\rangle_R$ and  the map $Q\mapsto R=3Q$
is injective on the whole original grid.  For each fixed $a$, attach
the set $E_Q$ above to $R=3Q$.  Then
$$
 E_Q\subset Q\subset R,\qquad
 |E_Q|\geq\tfrac12|Q|=\frac1{2\cdot3^d}|R|.
$$
Thus each image family is $(2\cdot3^d)^{-1}$-sparse.  Since
$Q\subset P$, every image cube satisfies
$\ell(R)\leq3\ell(P)$.  This proves
\eqref{eq:local-convex-body} and
\eqref{eq:prescribed-top-length}.  The grids and the map are fixed
before $P$ is chosen.
\end{proof}

\begin{lemma}
\label{lem:2m-commutator-lift}
Let $\mathcal T$ be a scalar linear operator.  Suppose that a
convex-body inclusion of the form
$$
 (\mathcal T\otimes I_{2m})F(x)
 \in C\sum_Q\one_Q(x)\langle\!\langle F\rangle\!\rangle_Q
$$
holds for
$$
 F(y)=\binom{-B(y)f(y)}{f(y)}.
$$
For fixed $x$ and $Q$, put
$$
 \mathcal C_Q(x)=
 \bigg\{\avg_Q\varphi(y)(B(x)-B(y))f(y)\,dy:
          \varphi\text{ scalar},\ |\varphi|\leq1\bigg\}.
$$
Then
\begin{equation}
 [M_B,\mathcal T]f(x)
 \in C\sum_Q\one_Q(x)\mathcal C_Q(x).
                                                               \label{eq:commutator-convex-body}
\end{equation}
\end{lemma}

\begin{proof}
For each $x$, define
$
 L_x(a,b)=a+B(x)b,\ a,b\in\mathbb C^m.
$
The image under $L_x$ of the left side is
$$
 -\mathcal T(Bf)(x)+B(x)\mathcal T f(x)
 =[M_B,\mathcal T]f(x).
$$
The image of
$$
 \avg_Q\varphi(y)\binom{-B(y)f(y)}{f(y)}\,dy
$$
is
$$
 \avg_Q\varphi(y)(B(x)-B(y))f(y)\,dy.
$$
Apply $L_x$ to the inclusion.
\end{proof}

For a dyadic grid $\mathcal D$, put
\begin{equation}
 \mathfrak T_N(\mathcal D)
 =\{Q\in\mathcal D:\ell(Q)<2^{-N}\ \text{or}\
   \ell(Q)>2^N\ \text{or}\ \dist(Q,0)>2^N\}.         \label{eq:tail-cube-family}
\end{equation}

\begin{lemma}
\label{lem:three-remainder-sparse-forms}
There are fixed integers $L,c_0$ and $\eta>0$ with the following
property.  For every $N$, every bounded compactly supported $f,g$ in
the form domain and  each
$$
 \tau\in\{{\rm out},{\rm loc},{\rm rem}\},
$$
there are dyadic grids $\mathcal D^{\tau,t}$ and $\eta$-sparse
families
$$
 \mathcal S_N^{\tau,t}\subset
 \mathfrak T_{N-c_0}(\mathcal D^{\tau,t}),
 \qquad 1\leq t\leq L,
$$
such that
\begin{align}
 |\langle R_N^\tau f,g\rangle|
 &\leq C_T\sum_{t=1}^L\sum_{Q\in\mathcal S_N^{\tau,t}}
 \frac1{|Q|}\int_Q\int_Q
 |\langle(B(x)-B(y))f(y),g(x)\rangle|\,dy\,dx.       \label{eq:three-remainder-positive-form}
\end{align}
\end{lemma}

\begin{proof}
We treat the three terms separately.
Since $f$ is bounded and compactly supported and $B$ is locally
integrable, the vector functions $F$, $F_N$ and  $F_P$ below belong
to $L^1$ and have compact support.  Thus both parts of
Lemma~\ref{lem:global-local-convex-body} apply on the stated form
domain.

For $R_N^{\rm out}$, apply
Lemma~\ref{lem:global-local-convex-body} in dimension $2m$ to
$F=(-Bf,f)$ and then apply
Lemma~\ref{lem:2m-commutator-lift}.  Multiply the resulting inclusion
by $\one_{Q_N^c}(x)$ and test it against $g(x)$.  Cubes which do not
intersect $Q_N^c$ in positive measure make no contribution.  If
$|Q\cap Q_N^c|>0$, then
\begin{equation}
 \ell(Q)\geq c_d2^N
 \quad\text{or}\quad
 \dist(Q,0)\geq c_d2^N.                              \label{eq:target-exterior-geometry}
\end{equation}
Indeed, if $\ell(Q)<2^{N-2}$ and $x\in Q\setminus Q_N$, then one
coordinate of every point of $Q$ has absolute value at least
$2^N-\ell(Q)>2^{N-1}$.  This proves
\eqref{eq:target-exterior-geometry} after decreasing $c_d$ by a
dimensional factor.  Enlarge the integral in $x$ from
$Q\cap Q_N^c$ to $Q$.  The support function of the convex body gives
\eqref{eq:three-remainder-positive-form}.

For $R_N^{\rm rem}$, apply the global inclusion to
$$
 F_N(y)=\one_{(3Q_N)^c}(y)\binom{-B(y)f(y)}{f(y)}
$$
and multiply the result by $\one_{Q_N}$.  Drop every sparse cube on which
the convex average of $F_N$ is zero.  Each remaining cube meets both
$Q_N$ and $(3Q_N)^c$.  Therefore
\begin{equation}
 \ell(Q)\geq c_d2^N.                                 \label{eq:remote-source-geometry}
\end{equation}
Apply the $2m$ lift, test against $g$ and  enlarge the $y$ integral
from $Q\cap(3Q_N)^c$ to $Q$ and the $x$ integral from
$Q\cap Q_N$ to $Q$.

For $R_N^{\rm loc}$, fix $P\in\mathcal P_N$ and apply the
prescribed-top inclusion to
$$
 F_P(y)=\one_{3P}(y)\binom{-B(y)f(y)}{f(y)}.
$$
Multiply the result by $\one_P$ and use the $2m$ lift.  Every resulting
cube satisfies
$$
 \ell(Q)\leq C_d\ell(P)=C_d2^{-N}.                  \label{eq:local-target-geometry}
$$
The cubes $P$ belong to one fixed grid.  Use the same $3^d$ adjacent
grids and
the same map $Q\mapsto3Q$ for every $P$.  Let
$E_{P,Q}\subset Q\subset P$ be the witness set from the local
recursion.  The cubes $P$ have disjoint interiors, so these witness
sets are pairwise disjoint even when they come from different top
cubes.  The map $Q\mapsto3Q$ is injective on the whole original
grid, not only inside one top cube.  Hence, in each of the $3^d$
three-lattice classes, the union over all $P$ remains
$(2\cdot3^d)^{-1}$-sparse.  Its cubes have side length at most
$3\cdot2^{-N}$.  Since $3Q\subset3P$, the vector $F_P$ equals the
full lifted vector on the averaging cube $3Q$.  Sum over $P$ and
enlarge only the $x$ integrals from $P\cap3Q$ to $3Q$.

Equations \eqref{eq:target-exterior-geometry},
\eqref{eq:remote-source-geometry} and 
\eqref{eq:local-target-geometry} place every cube in
$\mathfrak T_{N-c_0}$ for a fixed $c_0$.  This completes the proof.
\end{proof}

\begin{remark}
Lemma~\ref{lem:three-remainder-sparse-forms} estimates the three
operators in the identity \eqref{eq:CLLV-exact-decomposition}.
The sparse operators are auxiliary bounded operators.  They are not
claimed to be compact.
\end{remark}

\begin{lemma}
\label{lem:scalar-phase-linearization}
For fixed $f,g,B$ and a finite sparse family $\mathcal S$, there are
measurable scalar kernels $k_Q$, $|k_Q|\leq1$, such that
\begin{align}
 &\sum_{Q\in\mathcal S}\frac1{|Q|}\int_Q\int_Q
 |\langle(B(x)-B(y))f(y),g(x)\rangle|\,dy\,dx =\langle[M_B,\mathsf A_{\mathcal S,k}]f,g\rangle,
                                                               \label{eq:scalar-phase-linearization}
\end{align}
where
\begin{equation}
 \mathsf A_{\mathcal S,k}f(x)
 =\sum_{Q\in\mathcal S}\one_Q(x)\avg_Qk_Q(x,y)f(y)\,dy.         \label{eq:generalized-sparse-average}
\end{equation}
\end{lemma}

\begin{proof}
For each $Q$, put
$$
 \zeta_Q(x,y)=
 \langle(B(x)-B(y))f(y),g(x)\rangle
$$
and define
$$
 k_Q(x,y)=
 \begin{cases}
  \overline{\zeta_Q(x,y)}/|\zeta_Q(x,y)|,
       &\zeta_Q(x,y)\ne0,\\
  0,&\zeta_Q(x,y)=0.
 \end{cases}
$$
This is measurable, $|k_Q|\leq1$ and 
$$
 k_Q(x,y)\langle(B(x)-B(y))f(y),g(x)\rangle
 =|\langle(B(x)-B(y))f(y),g(x)\rangle|.
$$
Expand the commutator with \eqref{eq:generalized-sparse-average} and
pair it with $g$.  If the inner product is linear in its second
variable, use the conjugate phase.
\end{proof}

\subsubsection{Matrix-weighted sparse estimates}

\begin{lemma}
\label{lem:generalized-sparse-weighted}
Let $W\in\mathcal A_p(\mathcal D)$ be an $n\times n$ matrix weight.
For every finite $\eta$-sparse family
$\mathcal S\subset\mathcal D$ and all scalar kernels
$k_Q$ that are measurable and satisfy $|k_Q(x,y)|\leq1$,
\begin{equation}
 \|\mathsf A_{\mathcal S,k}\|_{L^p(W)\to L^p(W)}
 \leq C_{d,n,p,\eta,[W]_{\mathcal A_p(\mathcal D)}}.           \label{eq:generalized-sparse-weighted}
\end{equation}
The same constant applies to every finite subfamily of an infinite
sparse family.
\end{lemma}

\begin{proof}
Conjugate by $W^{1/p}$ and test against an unweighted function $g$.
Since the kernels are scalar and bounded by one, the absolute value of
the resulting form is at most
\begin{align}
 \sum_{Q\in\mathcal S}\frac1{|Q|}\int_Q\int_Q
 &\|W^{1/p}(x)W^{-1/p}(y)\|_{\op}
 |f(y)|\,|g(x)|\,dy\,dx.                            \label{eq:positive-matrix-sparse-form-new}
\end{align}

We prove the required estimate for this positive form.  Let
$\mathcal W_Q$ and $\mathcal W'_Q$ be exponent-$p$ and
exponent-$p'$ reducing matrices for $W$ and $W^{-p'/p}$ on $Q$.
For $e\ne0$, the scalar directional weight
$$
 w_e(x)=|W^{1/p}(x)e|^p
$$
has dyadic $A_p$ characteristic at most
$[W]_{\mathcal A_p(\mathcal D)}$.  The same assertion holds for the
dual matrix weight.  The scalar dyadic reverse H\"older theorem,
applied to the standard basis after normalization by
$\mathcal W_Q$ and $\mathcal W'_Q$, therefore gives one number
$\rho>1$ such that
\begin{align}
 \bigg(\avg_Q
 \|W^{1/p}(x)\mathcal W_Q^{-1}\|_{\op}^{p\rho}\,dx
 \bigg)^{1/(p\rho)}
 &\leq C,                                             \label{eq:sparse-normalized-primal-RH}\\
 \bigg(\avg_Q
 \|W^{-1/p}(x)(\mathcal W'_Q)^{-1}\|_{\op}^{p'\rho}\,dx
 \bigg)^{1/(p'\rho)}
 &\leq C.                                             \label{eq:sparse-normalized-dual-RH}
\end{align}
This is the dyadic version of the directional reverse H\"older
calculation in \cite[proof of Corollary~1.16]{CUIM}.  Matrix
$A_p$ duality also gives
$$
 \|\mathcal W_Q\mathcal W'_Q\|_{\op}\leq C.
$$
Since the matrices are self-adjoint,
$$
 \|\mathcal W_QW^{-1/p}(x)\|_{\op}
 \leq C\|W^{-1/p}(x)(\mathcal W'_Q)^{-1}\|_{\op}.
$$
Thus the left side has a uniform normalized $L^{p'\rho}(Q)$ norm.

For $x,y\in Q$, factor the operator norm in
\eqref{eq:positive-matrix-sparse-form-new} directly:
$$
 \|W^{1/p}(x)W^{-1/p}(y)\|_{\op}
 \leq
 \|W^{1/p}(x)\mathcal W_Q^{-1}\|_{\op}
 \|\mathcal W_QW^{-1/p}(y)\|_{\op}.
$$
Put
$$
 r=(p'\rho)',\qquad s=(p\rho)'.
$$
Then $1<r<p$ and $1<s<p'$.  H\"older's inequality and
\eqref{eq:sparse-normalized-primal-RH}--%
\eqref{eq:sparse-normalized-dual-RH} bound
\eqref{eq:positive-matrix-sparse-form-new} by
\begin{equation}
 C\sum_{Q\in\mathcal S}|Q|
 \bigg(\avg_Q|f|^r\bigg)^{1/r}
 \bigg(\avg_Q|g|^s\bigg)^{1/s}.                    \label{eq:positive-form-rs}
\end{equation}

Choose pairwise disjoint sets $E_Q\subset Q$ with
$|E_Q|\geq\eta|Q|$.  With
$$
 M_r^{\mathcal D}f=
 \big(M^{\mathcal D}(|f|^r)\big)^{1/r},
 \qquad
 M_s^{\mathcal D}g=
 \big(M^{\mathcal D}(|g|^s)\big)^{1/s},
$$
the sum in \eqref{eq:positive-form-rs} is at most
\begin{align*}
 \eta^{-1}\sum_{Q\in\mathcal S}
  \int_{E_Q}M_r^{\mathcal D}f\,M_s^{\mathcal D}g
 &\leq\eta^{-1}\int_{\mathbb R^d}
  M_r^{\mathcal D}f\,M_s^{\mathcal D}g
 \leq C\|f\|_{L^p}\|g\|_{L^{p'}}.
\end{align*}
The last estimate follows from $r<p$, $s<p'$ and  the dyadic maximal
theorem.  This is the positive-form argument in the proof of
\cite[Theorem~5.1]{CUIM}, now written for an arbitrary scalar kernel
with $|k_Q(x,y)|\leq1$.

Duality proves \eqref{eq:generalized-sparse-weighted}.  The estimate is
unchanged when cubes are deleted, so it is uniform over finite
subfamilies.
\end{proof}

For a matrix function $A$, put
\begin{equation}
 b_{\mathcal D}(A)=
 \sup_{Q\in\mathcal D}\bigg\{\avg_Q\bigg(
 \avg_Q\|V^{1/p}(x)(A(x)-A(y))U^{-1/p}(y)\|_{\op}^{p'}\,dy
                    \bigg)^{p/p'}dx\bigg\}^{1/p}.             \label{eq:dyadic-pairwise-sparse}
\end{equation}

\begin{lemma}
\label{lem:two-matrix-sparse-commutator}
Let $\mathcal S\subset\mathcal D$ be a finite $\eta$-sparse family.
If $b_{\mathcal D}(A)<\infty$, then
\begin{equation}
 \|[M_A,\mathsf A_{\mathcal S,k}]\|_{L^p(U)\to L^p(V)}
 \leq C\,b_{\mathcal D}(A),                          \label{eq:two-matrix-sparse-bound}
\end{equation}
uniformly in $\mathcal D,\mathcal S$ and  $k$, once $\eta$ is fixed.
\end{lemma}

\begin{proof}
Let $b=b_{\mathcal D}(A)$.  First suppose that $b>0$ and put $C=A/b$.
Define
\begin{equation}
 \Phi_C(x)=
 \begin{pmatrix}
 V^{1/p}(x)&V^{1/p}(x)C(x)\\
 0&U^{1/p}(x)
 \end{pmatrix},
 \qquad
 \mathbb W_C(x)=\big(\Phi_C(x)^*\Phi_C(x)\big)^{p/2}.            \label{eq:sparse-block-weight}
\end{equation}
The proof of Lemma~\ref{lem:block-weight-bdd} uses one cube at a time.
Restricting that proof to $Q\in\mathcal D$ gives
\begin{equation}
 [\mathbb W_C]_{\mathcal A_p(\mathcal D)}
 \leq C_{m,p}\big(1+[U]_{\mathcal A_p(\mathcal D)}
                     +[V]_{\mathcal A_p(\mathcal D)}\big).
                                                               \label{eq:sparse-block-Ap}
\end{equation}
Write $\mathsf A=\mathsf A_{\mathcal S,k}$ and
$\widetilde{\mathsf A}=\mathsf A\otimes I_{2m}$.  Direct
multiplication gives
\begin{equation}
 M_{\Phi_C}\widetilde{\mathsf A}M_{\Phi_C^{-1}}
 =\begin{pmatrix}
 M_{V^{1/p}}\mathsf A M_{V^{-1/p}}&
 M_{V^{1/p}}[M_C,\mathsf A]M_{U^{-1/p}}\\
 0&M_{U^{1/p}}\mathsf A M_{U^{-1/p}}
 \end{pmatrix}.                                      \label{eq:sparse-block-conjugation}
\end{equation}
Use the polar decomposition
$\Phi_C=\mathcal U_C\mathbb W_C^{1/p}$.  Multiplication by
$\mathcal U_C$ is an unweighted $L^p$ isometry.  Hence the norm of
the upper-right block is at most
$$
 \|\widetilde{\mathsf A}\|_{L^p(\mathbb W_C)
                           \to L^p(\mathbb W_C)}.
$$
Apply Lemma~\ref{lem:generalized-sparse-weighted} and
\eqref{eq:sparse-block-Ap}.  This proves the bound for $C$.
Multiplication by $b$ proves \eqref{eq:two-matrix-sparse-bound}.  If
$b=0$, replace $b$ by $b+\varepsilon$ and let
$\varepsilon\downarrow0$.
\end{proof}

\begin{lemma}
\label{lem:restricted-large-sparse}
Let $W\in\mathcal A_p(\mathcal D)$ be an $n\times n$ matrix weight.
Let $\mathcal S_{\geq N}\subset\mathcal D$ be a finite family consisting
of cubes with $\ell(Q)\geq2^N$.  Suppose that
$$
 \operatorname{diam}E\leq C_02^M,\qquad N\geq M+c_0.
$$
There are constants $C<\infty$ and $\theta>0$ such that
\begin{align}
 &\|\one_E\mathsf A_{\mathcal S_{\geq N},k}\|_{L^p(W)\to L^p(W)}
 +\|\mathsf A_{\mathcal S_{\geq N},k}\one_E\|_{L^p(W)\to L^p(W)}
 \leq C2^{-\theta(N-M)}.                             \label{eq:restricted-large-sparse}
\end{align}
The constants depend only on $d,n,p,C_0$ and the matrix $\mathcal A_p$
characteristic of $W$.
\end{lemma}

\begin{proof}
Let $\mathcal W_Q$ be an exponent-$p$ reducing matrix for $W$ on
$Q$.  For one cube, put
$$
 \mathsf A_{Q,k}f(x)=\one_Q(x)\avg_Qk_Q(x,y)f(y)\,dy.
$$
If $h=W^{1/p}f$, then
\begin{align*}
 |W^{1/p}(x)\mathsf A_{Q,k}f(x)|
 &\leq\one_Q(x)
 \|W^{1/p}(x)\mathcal W_Q^{-1}\|_{\op}
 \avg_Q\|\mathcal W_QW^{-1/p}(y)\|_{\op}|h(y)|\,dy.
\end{align*}
The matrix $\mathcal A_p$ condition gives
\begin{equation}
 \bigg(\avg_Q
 \|\mathcal W_QW^{-1/p}(y)\|_{\op}^{p'}\,dy
 \bigg)^{1/p'}\leq C.                              \label{eq:restricted-dual-factor}
\end{equation}
Therefore, for $F\subset Q$,
\begin{align}
 \|\one_F\mathsf A_{Q,k}f\|_{L^p(W)}
 &\leq C\bigg(\frac1{|Q|}\int_F
   \|W^{1/p}(x)\mathcal W_Q^{-1}\|_{\op}^p\,dx
             \bigg)^{1/p}\|\one_Qf\|_{L^p(W)}.      \label{eq:one-cube-before-Ainfty}
\end{align}

Uniformly in nonzero vectors $e$, the scalar weights
$|W^{1/p}e|^p$ have bounded directional $A_\infty$ characteristics.
Thus there are $C_1,\delta>0$ such that
$$
 \int_F|W^{1/p}(x)e|^p\,dx
 \leq C_1\bigg(\frac{|F|}{|Q|}\bigg)^\delta
       \int_Q|W^{1/p}(x)e|^p\,dx.
$$
Apply this to $e=\mathcal W_Q^{-1}e_s$, $1\leq s\leq n$.  Finite
dimensionality and the reducing property give
\begin{equation}
 \frac1{|Q|}\int_F
 \|W^{1/p}(x)\mathcal W_Q^{-1}\|_{\op}^p\,dx
 \leq C\bigg(\frac{|F|}{|Q|}\bigg)^\delta.          \label{eq:matrix-absolute-continuity}
\end{equation}
It follows that
\begin{equation}
 \|\one_E\mathsf A_{Q,k}\|_{L^p(W)\to L^p(W)}
 \leq C\bigg(\frac{|E\cap Q|}{|Q|}\bigg)^{\delta/p}.
                                                               \label{eq:one-cube-restricted}
\end{equation}

For $j\geq N$, let
$$
 \mathcal S_j=\{Q\in\mathcal S_{\geq N}:
                   2^j\leq\ell(Q)<2^{j+1}\}.
$$
The cubes in $\mathcal S_j$ belong to one dyadic generation and are
pairwise disjoint.  At most $C_{d,C_0}$ of them meet $E$ and , for each
such cube,
$$
 \frac{|E\cap Q|}{|Q|}
 \leq C_{d,C_0}2^{-d(j-M)}.
$$
It follows from \eqref{eq:one-cube-restricted} and disjointness that
$$
 \|\one_E\mathsf A_{\mathcal S_j,k}f\|_{L^p(W)}^p
 \leq C2^{-d\delta(j-M)}
       \sum_{Q\in\mathcal S_j}\|\one_Qf\|_{L^p(W)}^p
 \leq C2^{-d\delta(j-M)}\|f\|_{L^p(W)}^p.
$$
Sum this estimate over $j\geq N$:
$$
 \|\one_E\mathsf A_{\mathcal S_{\geq N},k}\|_{L^p(W)\to L^p(W)}
 \leq C\sum_{j=N}^\infty2^{-d\delta(j-M)/p}
 \leq C2^{-d\delta(N-M)/p}.
$$

For the restriction in the second variable, use duality.  The adjoint
is an operator
of the same form, with the variables in $k_Q$ interchanged.  Hence
\begin{align*}
 \|\mathsf A_{\mathcal S_{\geq N},k}\one_E\|_{L^p(W)\to L^p(W)}
 =\|\one_E\mathsf A_{\mathcal S_{\geq N},k}^*\|
       _{L^{p'}(W^\#)\to L^{p'}(W^\#)}.
\end{align*}
Apply the estimate in the first variable to $W^\#=W^{-p'/p}$.  Matrix $A_p$
duality controls its $A_{p'}$ characteristic.  Take the smaller of
the two exponents.
\end{proof}

\subsubsection{The dyadic tail of the symbol}

Let
\begin{equation}
 \epsilon_N(B)=\sup\{\omega_Q(B):
 \ell(Q)<2^{-N}\ \text{or}\ \ell(Q)>2^N\ \text{or}\
 \dist(Q,0)>2^N\}.                                  \label{eq:epsilon-tail}
\end{equation}
Equivalently, $\epsilon_N(B)$ is the maximum of the three suprema in
\eqref{eq:epsilon-tail}, taken separately over the small, large and 
far cubes.  These three indexing families decrease with $N$.  Thus
$\epsilon_N(B)$ is nonincreasing.
Changing the numerical constants in this definition only shifts $N$
by a fixed integer.  By the definition of the three $\omega$-tails,
or equivalently the six primal and dual tails,
\begin{equation}
 \lim_{N\to\infty}\epsilon_N(B)=\mathcal T_{V,U,p}(B).    \label{eq:epsdelta}
\end{equation}

Fix a dyadic grid $\mathcal D$.  Let
$$
 \Sig=\{0,1\}^d\setminus\{(1,\ldots,1)\}.
$$
For $I\in\mathcal D$ and $\varepsilon\in\Sig$, let
$h_I^\varepsilon$ be the $L^2$-normalized cancellative tensor Haar
function on $I$ and  put
$$
 B_I^\varepsilon=\int_I B(x)h_I^\varepsilon(x)\,dx.
$$
Set $C_M=[-2^M,2^M]^d$ and
\begin{align*}
 \Lag_M(\mathcal D)
 &=\{I\in\mathcal D:2^{-M}\leq\ell(I)\leq2^M,\
                         I\cap C_M\ne\varnothing\},\\
 P_M^{\mathcal D}B
 &=\sum_{I\in\Lag_M(\mathcal D)}
       \sum_{\varepsilon\in\Sig}B_I^\varepsilon h_I^\varepsilon,\\
 R_M^{\mathcal D}B&=B-P_M^{\mathcal D}B.             \label{eq:lagom-Haar-projection}
\end{align*}
The sum defining $P_M^{\mathcal D}B$ is finite.  All identities are
modulo constant matrices.

\begin{lemma}
\label{lem:dyadictail}
There is a fixed integer $c$ such that, uniformly over all dyadic
grids,
\begin{equation}
 \|R_M^{\mathcal D}B\|_{\BMO^p_{V,U,\mathcal D}}
 \leq C\epsilon_{M-c}(B).                            \label{eq:dyadictail}
\end{equation}
Consequently,
\begin{equation}
 b_{\mathcal D}(R_M^{\mathcal D}B)
 \leq C\epsilon_{M-c}(B).                            \label{eq:pairwise-dyadic-tail}
\end{equation}
\end{lemma}

\begin{proof}
We suppress the fixed shift $c$.  Fix $Q\in\mathcal D$.  On $Q$,
\begin{equation}
 R_M^{\mathcal D}B-(R_M^{\mathcal D}B)_Q
 =\sum_{\substack{I\subseteq Q\\I\notin\Lag_M(\mathcal D)}}
       \sum_{\varepsilon\in\Sig}
       B_I^\varepsilon h_I^\varepsilon.             \label{eq:Haar-tail-on-Q}
\end{equation}
All Haar identities in this proof are understood as limits of finite
martingale truncations.  For the primal columns, apply the truncations
to
$$
 \one_Q(B-B_Q)\mathcal U_Q^{-1}.
$$
They converge in $L^p(V)$: the matrix-weighted martingale transforms
are uniformly bounded and  finite Haar sums are dense.  For the dual
columns, apply the same argument in $L^{p'}(U^\#)$ to
$$
 \one_Q(B^*-B_Q^*)(\mathcal V'_Q)^{-1}.
$$
We first use the identities for finite generations and then pass to
the limit in these weighted spaces.
Suppose first that $Q\notin\Lag_M(\mathcal D)$.  The right side is a
scalar Haar projection of $B-B_Q$.  Scalar Haar multipliers with
coefficients in $\{0,1\}$ are uniformly bounded on $L^p(V)$; see
\cite[Section~2.1]{IKP}.  Apply this to the columns of
$(B-B_Q)\mathcal U_Q^{-1}$.  It gives
$$
 \alpha_Q(R_M^{\mathcal D}B)\leq C\alpha_Q(B).
$$
Apply the dual multiplier bound on $L^{p'}(U^\#)$ to the adjoint
expression.  It gives
$$
 \beta_Q(R_M^{\mathcal D}B)\leq C\beta_Q(B).
$$
Since $Q$ is non-lagom, both quantities are bounded by
$C\epsilon_M(B)$.

Now suppose that $Q\in\Lag_M(\mathcal D)$.  The non-lagom descendants
of $Q$ split into disjoint full dyadic subtrees.  Let
$\mathcal R(Q)$ be their maximal roots.  Every
$R\in\mathcal R(Q)$ is small or far and 
\begin{equation}
 R_M^{\mathcal D}B-(R_M^{\mathcal D}B)_Q
 =\sum_{R\in\mathcal R(Q)}\one_R(B-B_R).             \label{eq:Haar-tail-roots}
\end{equation}
Use disjointness and insert the reducing matrix $\mathcal U_R$:
\begin{align*}
 &\avg_Q\|V^{1/p}(R_M^{\mathcal D}B
       -(R_M^{\mathcal D}B)_Q)\mathcal U_Q^{-1}\|_{\op}^p\leq
 \frac{C\epsilon_M(B)^p}{|Q|}
 \sum_{R\in\mathcal R(Q)}|R|
       \|\mathcal U_R\mathcal U_Q^{-1}\|_{\op}^p.
\end{align*}
The reducing property and finite dimensionality give
\begin{align*}
 \sum_{R\in\mathcal R(Q)}|R|
       \|\mathcal U_R\mathcal U_Q^{-1}\|_{\op}^p
 &\leq C\sum_R\int_R
       \|U^{1/p}(x)\mathcal U_Q^{-1}\|_{\op}^p\,dx\leq C|Q|.
\end{align*}
Thus $\alpha_Q(R_M^{\mathcal D}B)\leq C\epsilon_M(B)$.
Take adjoints in \eqref{eq:Haar-tail-roots} and use the reducing
matrices $\mathcal V'_R$ for the dual oscillation.  The estimate
$$
 \sum_R|R|\|\mathcal V'_R(\mathcal V'_Q)^{-1}\|_{\op}^{p'}
 \leq C|Q|
$$
gives the same bound for $\beta_Q$.  This proves
\eqref{eq:dyadictail}.  The per-cube comparison between
$\widehat\omega_Q$ and $\alpha_Q+\beta_Q$ proves
\eqref{eq:pairwise-dyadic-tail}.
\end{proof}

\subsubsection{The finite-lagom sparse estimate}

\begin{proposition}
\label{prop:finite-lagom-sparse}
Let
$
 \mathcal S_N\subset\mathfrak T_N(\mathcal D)
$
be a finite $\eta$-sparse family.  If $1\leq M\leq N-c_1$, then
\begin{align}
 \|[M_B,\mathsf A_{\mathcal S_N,k}]\|_{X\to Y}
 \leq C\left\{\epsilon_{M-c_2}(B)
       +2^{-\theta(N-M)}\|B\|_{\BMO^p_{V,U}}\right\}.           \label{eq:finite-lagom-sparse-bound}
\end{align}
The constants are uniform in $\mathcal D,\mathcal S_N$ and  $k$, once
$\eta$ is fixed.
\end{proposition}

\begin{proof}
Put
$
 B_M=P_M^{\mathcal D}B,\
 R_M=R_M^{\mathcal D}B.
$
Lemma~\ref{lem:dyadictail} and
Lemma~\ref{lem:two-matrix-sparse-commutator} give
\begin{equation}
 \|[M_{R_M},\mathsf A_{\mathcal S_N,k}]\|_{X\to Y}
 \leq C\epsilon_{M-c_2}(B).                          \label{eq:finite-lagom-remainder}
\end{equation}

Let $E_M=\operatorname{supp}B_M$.  Since $B_M$ is a finite sum of
Haar functions from $\Lag_M(\mathcal D)$,
\begin{equation}
 E_M\subset[-C2^M,C2^M]^d,\qquad
 \operatorname{diam}E_M\leq C2^M.                   \label{eq:finite-Haar-support}
\end{equation}
Also, $B_M$ is constant on every dyadic cube of side length less than
$2^{-M}$.

Split $\mathcal S_N$ into
\begin{align*}
 \mathcal S_N^{\rm lg}
 &=\{Q\in\mathcal S_N:\ell(Q)>2^N\},\\
 \mathcal S_N^{\rm sm}
 &=\{Q\in\mathcal S_N\setminus\mathcal S_N^{\rm lg}:
                    \ell(Q)<2^{-N}\},\\
 \mathcal S_N^{\rm far}
 &=\mathcal S_N\setminus
       (\mathcal S_N^{\rm lg}\cup\mathcal S_N^{\rm sm}).
\end{align*}
For $N\geq M+c_1$, the function $B_M$ is constant on every cube in
$\mathcal S_N^{\rm sm}$.  Hence
\begin{equation}
 [M_{B_M},\mathsf A_{\mathcal S_N^{\rm sm},k}]=0.    \label{eq:small-sparse-zero}
\end{equation}
Every cube in $\mathcal S_N^{\rm far}$ is disjoint from $E_M$ after
the fixed shift in $c_1$.  Since $B_M$ vanishes on such a cube,
\begin{equation}
 [M_{B_M},\mathsf A_{\mathcal S_N^{\rm far},k}]=0.   \label{eq:far-sparse-zero}
\end{equation}

It remains to treat the large cubes.  Put
$
 a=b_{\mathcal D}(B_M).
$
The triangle inequality, the per-cube comparison and 
Lemma~\ref{lem:dyadictail} give
\begin{equation}
 a\leq C\|B\|_{\BMO^p_{V,U}}.                       \label{eq:finite-symbol-pairwise}
\end{equation}
For $a>0$, put $\widetilde B_M=B_M/a$ and form the block weight
$\mathbb W_{\widetilde B_M}$ in \eqref{eq:sparse-block-weight}.  Its dyadic
$A_p$ characteristic is bounded independently of $M$ and
$\mathcal D$.

Write
$$
 \mathsf A=\mathsf A_{\mathcal S_N^{\rm lg},k},
 \qquad D=[M_{\widetilde B_M},\mathsf A].
$$
The kernel of $D$ contains $\widetilde B_M(x)-\widetilde B_M(y)$.  It
vanishes when both
$x$ and $y$ lie outside $E_M$.  Therefore
\begin{equation}
 D=\one_{E_M}D+\one_{E_M^c}D\one_{E_M}.              \label{eq:large-symbol-support-split}
\end{equation}
Let
$$
 \mathcal L=M_{\Phi_{\widetilde B_M}}(\mathsf A\otimes I_{2m})
                       M_{\Phi_{\widetilde B_M}^{-1}}.
$$
The conjugated operator $M_{V^{1/p}}DM_{U^{-1/p}}$ is the
upper-right block of $\mathcal L$.  Hence
\begin{align}
 \|D\|_{L^p(U)\to L^p(V)}
 &\leq\|\one_{E_M}\mathcal L\|_{L^p\to L^p}
       +\|\mathcal L\one_{E_M}\|_{L^p\to L^p}.       \label{eq:large-symbol-block-restriction}
\end{align}
The polar decomposition of $\Phi_{\widetilde B_M}$ and the fact that scalar
indicators commute with matrix multiplication show that the right
side equals
\begin{align*}
 &\|\one_{E_M}(\mathsf A\otimes I_{2m})\|
       _{L^p(\mathbb W_{\widetilde B_M})\to L^p(\mathbb W_{\widetilde B_M})}
       +\|(\mathsf A\otimes I_{2m})\one_{E_M}\|
       _{L^p(\mathbb W_{\widetilde B_M})\to L^p(\mathbb W_{\widetilde B_M})}.
\end{align*}
Apply Lemma~\ref{lem:restricted-large-sparse} and
\eqref{eq:finite-Haar-support}.  After multiplying by $a$ and using
\eqref{eq:finite-symbol-pairwise}, we obtain
\begin{equation}
 \|[M_{B_M},\mathsf A_{\mathcal S_N^{\rm lg},k}]\|_{X\to Y}
 \leq C2^{-\theta(N-M)}\|B\|_{\BMO^p_{V,U}}.         \label{eq:large-sparse-decay}
\end{equation}
If $a=0$, use $a+\varepsilon$ and let $\varepsilon\downarrow0$.
Combine \eqref{eq:finite-lagom-remainder},
\eqref{eq:small-sparse-zero}, \eqref{eq:far-sparse-zero} and 
\eqref{eq:large-sparse-decay}.
\end{proof}

\subsubsection{The upper essential norm}

\begin{proposition}
\label{prop:upper}
If $B\in\BMO^p_{V,U}$, then
\begin{equation}
 \|C_{B,T}\|_{\ess,X\to Y}
 \leq C_T\mathcal T_{V,U,p}(B).                          \label{eq:general-czo-upper-essential-norm}
\end{equation}
In particular, if $B\in\VMO^p_{V,U}$, then
$C_{B,T}:X\to Y$ is compact.  No non-degeneracy assumption on $T$ is
used.
\end{proposition}

\begin{proof}
Fix a sufficiently large $N$ and put
$$
 M=\left\lfloor\frac{N-c_0}{2}\right\rfloor.
$$
We take $N$ large enough that $M\geq1$ and
$M\leq (N-c_0)-c_1$, as required in
Proposition~\ref{prop:finite-lagom-sparse}.
Apply Lemma~\ref{lem:three-remainder-sparse-forms} to one of the
three remainder operators and a pair $f,g$ in the form domain.
Fix a finite subfamily of one of the resulting sparse families.
Lemma~\ref{lem:scalar-phase-linearization} and
Proposition~\ref{prop:finite-lagom-sparse}, with the fixed shift
$N-c_0$, give
\begin{align*}
 &\sum_{Q\ \mathrm{in\ the\ finite\ subfamily}}
 \frac1{|Q|}\int_Q\int_Q
 |\langle(B(x)-B(y))f(y),g(x)\rangle|\,dy\,dx\\
 &\quad\leq C\left\{\epsilon_{M-c}(B)
       +2^{-\theta(N-c_0-M)}\|B\|_{\BMO^p_{V,U}}\right\}
       \|f\|_X\|g\|_{L^{p'}(V^\#)}.
\end{align*}
The constant is independent of the finite subfamily, the grid and 
the scalar phases.  Take the supremum over finite subfamilies.  Since
the positive form is an increasing sum, we obtain the same estimate
for the full sparse family.  Sum over the fixed number of grids and
over
$$
 \tau\in\{{\rm out},{\rm loc},{\rm rem}\}.
$$
Lemma~\ref{lem:three-remainder-sparse-forms} now gives
\begin{align}
 \|R_N^{\rm out}+R_N^{\rm loc}+R_N^{\rm rem}\|_{X\to Y}\leq C_T\left\{\epsilon_{M-c}(B)
       +2^{-\theta(N-c_0-M)}\|B\|_{\BMO^p_{V,U}}\right\}.       \label{eq:CLLV-remainder-norm}
\end{align}
Use the identity \eqref{eq:CLLV-exact-decomposition}:
\begin{equation}
 \|C_{B,T}-K_N\|_{X\to Y}
 \leq C_T\left\{\epsilon_{\lfloor(N-c_0)/2\rfloor-c}(B)
       +2^{-\theta(N-c_0-\lfloor(N-c_0)/2\rfloor)}
                    \|B\|_{\BMO^p_{V,U}}\right\}.     \label{eq:CLLV-compact-approximation}
\end{equation}
The operator $K_N$ is compact by
Corollary~\ref{cor:CLLV-compact-approximant}.  Hence
\begin{align*}
 \|C_{B,T}\|_{\ess,X\to Y}
 &\leq C_T\liminf_{N\to\infty}
 \left\{\epsilon_{\lfloor(N-c_0)/2\rfloor-c}(B)
       +2^{-\theta(N-c_0-\lfloor(N-c_0)/2\rfloor)}
                    \|B\|_{\BMO^p_{V,U}}\right\} =C_T\mathcal T_{V,U,p}(B),
\end{align*}
where the last equality follows from \eqref{eq:epsdelta}.  If
$B\in\VMO^p_{V,U}$, the right side of
\eqref{eq:CLLV-compact-approximation} tends to zero.  Thus
$C_{B,T}$ is an operator-norm limit of compact operators.
\end{proof}

\begin{remark}
The dyadic grids are used only to estimate the three remainders in
\eqref{eq:CLLV-exact-decomposition}.
\end{remark}

\subsection{Necessity and the lower essential norm}\label{sec:lower}

Choose the integer $M$ in Lemma~\ref{lem:reciprocal-tests}
sufficiently large below.  Index the
members of $\mathscr R_M(Q)$ by their relative positions
$\rho\in\{1,\ldots,M\}^d$ and write $R_\rho(Q)$ for the corresponding
subcube.

\begin{lemma}\label{lem:weaknull}
There exist $r>1$ and $C<\infty$ such that
\begin{equation}
 \bigg(\avg_S
 \|U^{1/p}(x)\mathcal U_S^{-1}\|_{\op}^{pr}\,dx
 \bigg)^{1/(pr)}\le C                               \label{eq:reverseholder}
\end{equation}
for every cube $S$.  For each fixed relative position $\rho$ and
column index $s$, the tests
$g_{Q_n,R_\rho(Q_n),s}$ in Lemma~\ref{lem:reciprocal-tests} converge
weakly to zero in $L^p(U)$ whenever
\begin{enumerate}[label=(\roman*)]
\item $\ell(Q_n)\to0$;
\item $\ell(Q_n)\to\infty$;
\item $0<c\le\ell(Q_n)\le C_0<\infty$ and
$\dist(Q_n,0)\to\infty$.
\end{enumerate}
\end{lemma}

\begin{proof}
For every nonzero $e\in\mathbb C^m$, the directional weight
$$
 w_e(x)=|U^{1/p}(x)e|^p
$$
belongs to scalar $A_\infty$, with its $A_\infty$ characteristic
bounded uniformly in $e$; see the directional reverse H\"older
calculation in \cite[proof of Corollary~1.16]{CUIM}.  Consequently,
there are common constants $r>1$ and $C$ such that
\begin{equation}
 \bigg(\avg_S|U^{1/p}(x)e|^{pr}\,dx\bigg)^{1/(pr)}
 \leq C\bigg(\avg_S|U^{1/p}(x)e|^p\,dx\bigg)^{1/p}
                                                        \label{eq:directional-reverseholder}
\end{equation}
for every cube $S$ and every $e\in\mathbb C^m$.  If
$e_1,\ldots,e_m$ is the standard basis, then
$$
 \|U^{1/p}(x)\mathcal U_S^{-1}\|_{\op}^{pr}
 \lesssim_m \sum_{k=1}^m
 |U^{1/p}(x)\mathcal U_S^{-1}e_k|^{pr}.
$$
Apply \eqref{eq:directional-reverseholder} to
$\mathcal U_S^{-1}e_k$ and use the reducing normalization
$$
 \bigg(\avg_S|U^{1/p}(x)\mathcal U_S^{-1}e_k|^p\,dx\bigg)^{1/p}
 \simeq |\mathcal U_S\mathcal U_S^{-1}e_k|=1.
$$
After summing over $k$, this proves \eqref{eq:reverseholder}.

For later use, suppose that $Q\subset S$ and $|S|\leq C_1|Q|$.
Apply \eqref{eq:weighted-averaging-operator-bdd} on $S$ to
$f=\one_Q\mathcal U_Q^{-1}e$.  Since
$$
 \mathsf A_Sf=\frac{|Q|}{|S|}\one_S\mathcal U_Q^{-1}e,
$$
we obtain
\begin{align}
 \bigg(\avg_S|U^{1/p}(x)\mathcal U_Q^{-1}e|^p\,dx\bigg)^{1/p}
 &\leq [U]_{\mathcal A_p}^{1/p}
 \bigg(\frac{|S|}{|Q|}\bigg)^{1/p'}
 \bigg(\avg_Q|U^{1/p}(x)\mathcal U_Q^{-1}e|^p\,dx\bigg)^{1/p}
 \lesssim_{U,C_1}|e|.                              \label{eq:comparable-reducing-weak-null}
\end{align}
By the reducing property on $S$, this gives
$\|\mathcal U_S\mathcal U_Q^{-1}\|_{\op}\lesssim_{U,C_1}1$.

Put $F_n=U^{1/p}g_{Q_n,R_\rho(Q_n),s}$.  Let $S_n$ be the cube in
Lemma~\ref{lem:reciprocal-tests} that contains
$Q_n\cup Q_n^\sharp$ and satisfies $|S_n|\lesssim |Q_n|$.  The bound
for the reciprocal multiplier and
\eqref{eq:comparable-reducing-weak-null} give the pointwise estimate
\begin{equation}
 |F_n(x)|\lesssim |Q_n|^{-1/p}\mathbf1_{Q_n^\sharp}(x)
 \|U^{1/p}(x)\mathcal U_{S_n}^{-1}\|_{\op}.          \label{eq:weak-null-pointwise}
\end{equation}
It follows from H\"older's inequality and \eqref{eq:reverseholder}
that, for every measurable set $E$ of finite measure,
\begin{align}
 \int_E|F_n(x)|^p\,dx
 &\lesssim |Q_n|^{-1}|E\cap Q_n^\sharp|^{1/r'}
 \bigg(\int_{S_n}
 \|U^{1/p}(x)\mathcal U_{S_n}^{-1}\|_{\op}^{pr}\,dx
 \bigg)^{1/r}\lesssim
 \bigg(\frac{|E|}{|Q_n|}\bigg)^{1/r'}.             \label{eq:weak-null-local-mass}
\end{align}

Weak convergence in $L^p(U)$ is equivalent to weak convergence of
$F_n$ in unweighted $L^p$.  The functions $F_n$ are uniformly bounded
in $L^p$ by Lemma~\ref{lem:reciprocal-tests}.  In case (i), their
supports have measure tending to zero.  In case (iii), the reciprocal
geometry gives
$$
 \dist(Q_n^\sharp,0)
 \geq\dist(Q_n,0)-C\ell(Q_n)\longrightarrow\infty.
$$
Thus their supports escape every compact set.  H\"older's inequality and absolute
continuity of the $L^{p'}$ norm give weak convergence in both cases.
In case (ii), \eqref{eq:weak-null-local-mass} tends to zero for every
bounded $E$.  Test first against compactly supported $L^{p'}$
functions and then use the uniform global $L^p$ bound to approximate
an arbitrary $L^{p'}$ function by compactly supported ones.
\end{proof}

\begin{proposition}\label{prop:lower}
Let $1<p<\infty$, let $U,V$ be $m\times m$ matrix $\mathcal A_p$
weights and  let $T$ be an $L^2$-bounded Calder\'on--Zygmund operator
with a $\delta$-standard scalar kernel.  Assume that the kernel is
two-sided non-degenerate in the sense of
Definition~\ref{def:two-sided-nondegeneracy}.  If
$B\in\BMO^p_{V,U}$, then
$$
 \mathcal T_{V,U,p}(B)\lesssim \|C_{B,T}\|_{\ess}.
$$
Here the essential norm is taken from $L^p(U;\C^m)$ to
$L^p(V;\C^m)$.  In particular, compactness forces the three
$\omega$-tails, or equivalently the six primal and dual tails, to
vanish as in Definition~\ref{def:vmo}.
\end{proposition}

\begin{proof}
Write $\tau=\mathcal T_{V,U,p}(B)$.  If $\tau=0$, there is nothing to
prove.  Choose cubes $Q_n$ in one of the three tail regimes such that
$\omega_{Q_n}(B)\to\tau$.  After passing to a subsequence, one of
$\alpha_{Q_n}(B)$ and $\beta_{Q_n}(B)$ is at least
$\frac12\omega_{Q_n}(B)$ for every $n$.  We first treat the primal
case.

Let $S_n=S_{Q_n}$ be the enlarged cube in
Lemma~\ref{lem:reciprocal-tests}.  The geometry is uniform.  Hence
\begin{equation}\label{eq:enlarged-tail-limsup}
 \limsup_{n\to\infty}\omega_{S_n}(B)\leq\tau.
\end{equation}
For the small and large regimes, this follows from
$\ell(S_n)\simeq \ell(Q_n)$.  In the far regime, pass to a further
subsequence.  If the side lengths tend to zero or infinity, use the
small or large defect.  If they remain between two positive constants,
then $S_n$ also escapes to infinity, so use the far defect.  This is
the only point where a far companion cube may be absorbed into a
large-cube tail rather than another far tail.

Let $K_0:L^p(U)\to L^p(V)$ be compact.  Lemma~\ref{lem:weaknull}
and the finiteness of the index sets for $\rho$ and $s$ give
\begin{equation}\label{eq:compact-reciprocal-vanishing}
 \max_{\rho,s}
 \|K_0g_{Q_n,R_\rho(Q_n),s}\|_{L^p(V)}\longrightarrow0.
\end{equation}
Use $C_{B,T}=(C_{B,T}-K_0)+K_0$ in
\eqref{eq:localized-reciprocal-alpha}.  The uniform test bound and
\eqref{eq:compact-reciprocal-vanishing} yield
\begin{equation}\label{eq:primal-essential-preabsorb}
 \limsup_{n\to\infty}\alpha_{Q_n}(B)
 \leq C_{\mathrm{test},M}\|C_{B,T}-K_0\|
      +C_A M^{-\delta}\tau.
\end{equation}
Choose the fixed integer $M$ so large that the last coefficient is at
most $1/4$.  Since the left side is at least $\tau/2$, the last term
can be absorbed.  Taking the infimum over $K_0$ gives
$$
 \tau\lesssim \|C_{B,T}\|_{\ess}
$$
in the primal case.

For the dual case, put
$$
 V^\#=V^{-p'/p},\qquad U^\#=U^{-p'/p}.
$$
The adjoint identity is
\begin{equation}\label{eq:general-commutator-adjoint}
 C_{B,T}^*=-C_{B^*,T^*}:
 L^{p'}(V^\#)\longrightarrow L^{p'}(U^\#).
\end{equation}
For the pair $(V^\#,U^\#)$ at exponent $p'$, the primal oscillation of
$B^*$ is
$$
 \bigg(\avg_Q
 \|U^{-1/p}(B^*-B_Q^*)(\mathcal V'_Q)^{-1}\|_{\op}^{p'}
 \,dx\bigg)^{1/p'}
 =\beta_Q(B).
$$
Its dual oscillation is $\alpha_Q(B)$.  Thus the new $\omega_Q$ and
its three tail defects agree with those of the original pair when the
reducing matrices $\mathcal V'_Q$ and $\mathcal U_Q$ are used.
The kernel of $T^*$ is $K^*(x,y)=\overline{K(y,x)}$.  Hence the
condition
$$
 |K(y_-,x)|\gtrsim r^{-d},
$$
which is the second inequality in
\eqref{eq:two-sided-nondegeneracy}, gives
$$
 |K^*(x,y_-)|\gtrsim r^{-d}.
$$
This is the condition required to apply
Lemma~\ref{lem:reciprocal-tests} to $T^*$.
If $K_0:L^p(U)\to L^p(V)$ is compact, then $K_0^*$ is compact and
$$
 \|C_{B,T}^*-K_0^*\|=\|C_{B,T}-K_0\|.
$$
Apply the primal argument at exponent $p'$ to $C_{B^*,T^*}$ and use
\eqref{eq:general-commutator-adjoint}.  This proves the same lower
bound in the dual case.
\end{proof}

If $C_{B,T}$ is compact, then it is bounded.  Theorem
\ref{thm:two-matrix-boundedness} first gives $B\in\BMO^p_{V,U}$.
Proposition~\ref{prop:lower} then gives the three $\omega$-conditions,
or equivalently all six primal and dual vanishing conditions.
Conversely, Proposition~\ref{prop:upper} gives compactness when these
conditions hold.  The two propositions also give
\eqref{eq:essential}.  This proves Theorem~\ref{thm:main}.

\section{Schatten classes for Riesz-transform commutators}

\subsection{The $L^2$ specialization}

Let $d\geq2$, let $U,V$ be $m\times m$ matrix $\mathcal A_2$ weights
on $\R^d$ and  let $B:\R^d\to\C^{m\times m}$ be locally integrable.
Thus a matrix weight $W$ is positive definite almost everywhere,
$W,W^{-1}$ are locally integrable.  In this part we use the
equivalent characteristic
\begin{equation}
 [W]_{\mathcal A_2}
 =\sup_Q
 \|\langle W\rangle_Q^{1/2}
   \langle W^{-1}\rangle_Q^{1/2}\|_{\op}^2<\infty.
 \label{sch:eq:matrix-A2-definition}
\end{equation}
If $[W]_{\mathcal A_2}^{\rm int}$ denotes the quantity in
\eqref{eq:matrix-Ap} with $p=2$, then
$$
 [W]_{\mathcal A_2}\leq [W]_{\mathcal A_2}^{\rm int}
 \leq m[W]_{\mathcal A_2}.
$$
The first inequality follows from convexity of the operator norm and 
the second follows by using the Hilbert--Schmidt norm and taking the
trace.  All constants below are unchanged up to a factor depending on
$m$.
Fix $1\leq j\leq d$ and  let $R_j$ be the $j$th Riesz transform,
with kernel
$$
 K_j(x,y)=c_d\frac{x_j-y_j}{|x-y|^{d+1}}.
$$
Put $C_{B,R_j}=[M_B,R_j\otimes I_m]$.
Schatten classes are defined only between Hilbert spaces.  Thus the
Lebesgue exponent in this part is fixed at $2$.  The unitary maps
$$
 J_Uf=U^{1/2}f,
 \qquad
 J_Vg=V^{1/2}g
$$
identify $L^2(U;\C^m)$ and $L^2(V;\C^m)$ with
$L^2(\R^d;\C^m)$.  Hence
\begin{equation}
 \|C_{B,R_j}\|_{\Sch^{r,q}(L^2(U),L^2(V))}
 =\|\mathcal C_{B,R_j}^{V,U}\|_{\Sch^{r,q}(L^2,L^2)},
 \qquad
 \mathcal C_{B,R_j}^{V,U}
 =M_{V^{1/2}}C_{B,R_j}M_{U^{-1/2}}.
 \label{sch:eq:unitary-conjugation}
\end{equation}
Off the diagonal, the kernel of $\mathcal C_{B,R_j}^{V,U}$ is
\begin{equation}
 V^{1/2}(x)(B(x)-B(y))K_j(x,y)U^{-1/2}(y).
 \label{sch:eq:weighted-commutator-kernel}
\end{equation}

For a bounded operator $A:H_1\to H_2$, let
$$
 s_n(A)=\inf\{\|A-F\|:\operatorname{rank}F<n\},
 \qquad n\geq1,
$$
be its approximation numbers.  Thus $A$ is compact if and only if
$s_n(A)\to0$ and  for a compact operator these are its singular
values, listed in nonincreasing order.  For $r>0$ and
$0<q\leq\infty$, $\Sch^{r,q}(H_1,H_2)$ denotes the
Schatten--Lorentz class defined by
$$
 \|A\|_{\Sch^{r,q}}=\|\{s_n(A)\}_{n\geq1}\|_{\ell^{r,q}}.
$$
For any countable family, its $\ell^{r,q}$ quasi-norm is defined by
the decreasing rearrangement of its absolute values.  We use the
conventions $\Sch^r=\Sch^{r,r}$ and $\ell^r=\ell^{r,r}$.  When
$0<q<1$, these are quasi-norms.
Membership in any of these classes implies compactness.

\subsection{The matrix two-weight oscillation sequences}

At exponent $2$ we take
$
 \mathcal U_Q=\langle U\rangle_Q^{1/2},
 \ \ 
 \mathcal U'_Q=\langle U^{-1}\rangle_Q^{1/2},
 \ \ 
 \mathcal V_Q=\langle V\rangle_Q^{1/2},
 \ \ 
 \mathcal V'_Q=\langle V^{-1}\rangle_Q^{1/2}.
$
Here $B_Q=\avg_QB$.
Define the full-cube oscillations
\begin{align}
 \alpha_Q^{\HS}(B)
 &=\bigg(\avg_Q
 \|V^{1/2}(x)(B(x)-B_Q)\mathcal U_Q^{-1}\|_{\HS}^{2}\,dx
 \bigg)^{1/2},                                      \label{sch:eq:primal-osc}\\
 \beta_Q^{\HS}(B)
 &=\bigg(\avg_Q
 \|U^{-1/2}(x)(B^*(x)-B_Q^*)(\mathcal V'_Q)^{-1}\|_{\HS}^{2}\,dx
 \bigg)^{1/2},                                      \label{sch:eq:dual-osc}\\
 \omega_Q^{\HS}(B)&=\max\{\alpha_Q^{\HS}(B),\beta_Q^{\HS}(B)\}.                 \label{sch:eq:Omega}
\end{align}
The Hilbert--Schmidt norm in \eqref{sch:eq:primal-osc}--\eqref{sch:eq:dual-osc}
is convenient for Schatten estimates.  Since the fibre dimension $m$
is fixed, it is equivalent to the operator norm.  Thus $\omega_Q^{\HS}(B)$
is equivalent to the operator-norm quantity $\omega_Q(B)$ used in Parts~I and~II.

Let $\D^1,\ldots,\D^{3^d}$ be fixed adjacent dyadic systems.  For
$0<r<\infty$ and $0<q\leq\infty$, put
\begin{equation}
 \|B\|_{\mathfrak B^{r,q}_{V,U}}
 =\sum_{t=1}^{3^d}
 \big\|\{\omega_Q^{\HS}(B):Q\in\D^t\}\big\|_{\ell^{r,q}}.
 \label{sch:eq:matrix-Besov-osc}
\end{equation}
We write $\mathfrak B^r_{V,U}=\mathfrak B^{r,r}_{V,U}$ and
$$
 \|B\|_{\mathfrak W^d_{V,U}}
 =\|B\|_{\mathfrak B^{d,\infty}_{V,U}}.
$$
These are seminorms modulo constant matrices.

\begin{lemma}
\label{sch:lem:sequence-implies-VMO}
Let $0<r<\infty$ and $0<q\leq\infty$.  If
$B\in\mathfrak B^{r,q}_{V,U}$, then $B\in\VMO^2_{V,U}$.
\end{lemma}

\begin{proof}
Suppose that $R\subset Q$ and $|Q|\leq A|R|$.  By
Lemma~\ref{lem:localcomparison}, the equivalence of the operator and
Hilbert--Schmidt norms and  enlargement of the two positive integrals
in \eqref{eq:pairwise},
\begin{equation}
 \omega_R^{\HS}(B)\lesssim_m\widehat\omega_R(B)
 \leq A\widehat\omega_Q(B)
 \lesssim_m\omega_Q^{\HS}(B).                       \label{sch:eq:comparable-cube-oscillation}
\end{equation}
For every cube $R$, choose a cube $Q$ in one of the fixed adjacent
systems such that $R\subset Q$ and $|Q|\leq C_d|R|$.  For each
$\epsilon>0$, an $\ell^{r,q}$ sequence has only finitely many terms
larger than $\epsilon$.  It follows from
\eqref{sch:eq:comparable-cube-oscillation} that every cube $R$ with
$\omega_R^{\HS}(B)>\epsilon$ is contained in and  has side length
comparable to, one of finitely many cubes.  Such cubes $R$ cannot have
side length tending to zero or infinity, or tend to infinity while
their side lengths remain bounded.  The three vanishing conditions
follow.  Since $\ell^{r,q}\subset\ell^\infty$, the same comparison
also gives $B\in\BMO^2_{V,U}$.  The operator and Hilbert--Schmidt norms
are equivalent in the fixed matrix dimension, so
$B\in\VMO^2_{V,U}$.
\end{proof}

\subsection{The Schatten theorems}

For a locally integrable $B$, we first interpret $C_{B,R_j}$ by the
separated-support form
$$
 \langle C_{B,R_j}f,g\rangle
 =\iint g(x)^*(B(x)-B(y))K_j(x,y)f(y)\,dy\,dx
$$
for bounded compactly supported vector functions whose supports have
positive distance.  A Schatten assertion below means that this form
has an extension from $L^2(U)$ to $L^2(V)$ in the stated class.  Such
a compact extension is unique; this follows from
Lemma~\ref{sch:lem:compact-local} below.

\begin{theorem}
\label{sch:thm:target}
Let $d\geq2$.  Let $U,V\in\mathcal A_2$ be matrix weights and let
$B:\R^d\to\C^{m\times m}$ be locally integrable.  Fix
$1\leq j\leq d$.  Then the following statements hold.

\begin{enumerate}
\item If $d<r<\infty$, then
\begin{equation}
 C_{B,R_j}\in\Sch^r(L^2(U;\C^m),L^2(V;\C^m))
 \quad\Longleftrightarrow\quad
 B\in\mathfrak B^r_{V,U},
 \label{sch:eq:strong-target}
\end{equation}
and
\begin{equation}
 \|C_{B,R_j}\|_{\Sch^r(L^2(U),L^2(V))}
 \simeq \|B\|_{\mathfrak B^r_{V,U}}.
 \label{sch:eq:strong-norm-target}
\end{equation}

\item If $0<r\leq d$, then
\begin{equation}
 C_{B,R_j}\in\Sch^r(L^2(U;\C^m),L^2(V;\C^m))
 \quad\Longleftrightarrow\quad
 \begin{gathered}\text{there is a fixed }B_0\in\C^{m\times m}\text{ such that}\\B(x)=B_0\text{ for almost every }x\in\R^d\end{gathered}.
 \label{sch:eq:subcritical-target}
\end{equation}

\item At the critical index,
\begin{equation}
 C_{B,R_j}\in\Sch^{d,\infty}(L^2(U;\C^m),L^2(V;\C^m))
 \quad\Longleftrightarrow\quad
 B\in\mathfrak W^d_{V,U},
 \label{sch:eq:weak-target}
\end{equation}
and
\begin{equation}
 \|C_{B,R_j}\|_{\Sch^{d,\infty}(L^2(U),L^2(V))}
 \simeq \|B\|_{\mathfrak W^d_{V,U}}.
 \label{sch:eq:weak-norm-target}
\end{equation}
\end{enumerate}

The constants in the norm equivalences depend only on
$d,m,r,[U]_{\mathcal A_2}$ and  $[V]_{\mathcal A_2}$ and  on the
normalization of $R_j$.
\end{theorem}

\begin{corollary}
\label{sch:cor:Schatten-implies-VMO}
Under the assumptions of Theorem~\ref{sch:thm:target}, either of the
conditions
$$
 B\in\mathfrak B^r_{V,U}\quad (r>d),
 \qquad\text{or}\qquad
 B\in\mathfrak W^d_{V,U}
$$
implies $B\in\VMO^2_{V,U}$.  No converse is asserted.
\end{corollary}

\begin{proof}
This follows from Lemma~\ref{sch:lem:sequence-implies-VMO}, with
$(r,q)=(r,r)$ or $(d,\infty)$.
\end{proof}

\subsection{The Haar characterization of the oscillation spaces}

Let $h_Q^\varepsilon$, $\varepsilon\in\{0,1\}^d$, be the
$L^2$-normalized tensor Haar functions.  The signature
$\mathbf1=(1,\ldots,1)$ denotes the noncancellative function
$$
 h_Q^{\mathbf1}=|Q|^{-1/2}\one_Q.
$$
For a cancellative signature, define
$$
 \widehat B(Q,\varepsilon)
 =\int_QB(x)h_Q^\varepsilon(x)\,dx.
$$
For the cancellative signatures, put
$$
 \Delta_QB
 =\sum_{\varepsilon\neq\mathbf1}
   \widehat B(Q,\varepsilon)h_Q^\varepsilon.
$$
We write $\operatorname{ch}^k(Q)$ for the $k$th-generation
dyadic descendants of $Q$, with
$\operatorname{ch}^0(Q)=\{Q\}$.

There are two Haar forms.  For $Q\in\D^t$, put
\begin{equation}
 \mathfrak h_Q^+(B)=
 \bigg(
 \sum_{\varepsilon\neq\mathbf1}
 \bigg\|
 \mathcal V_Q
 \frac{\widehat B(Q,\varepsilon)}{|Q|^{1/2}}
 \mathcal U_Q^{-1}
 \bigg\|_{\HS}^{2}
 \bigg)^{1/2}.
 \label{sch:eq:Haar-primal}
\end{equation}
The adjoint-side coefficient is
\begin{equation}
 \mathfrak h_Q^-(B)=
 \bigg(
 \sum_{\varepsilon\neq\mathbf1}
 \bigg\|
 \mathcal U'_Q
 \frac{\widehat B(Q,\varepsilon)^*}{|Q|^{1/2}}
 (\mathcal V'_Q)^{-1}
 \bigg\|_{\HS}^{2}
 \bigg)^{1/2}.
 \label{sch:eq:Haar-dual}
\end{equation}
Set
\begin{equation}
 \mathfrak h_Q(B)=\max\{\mathfrak h_Q^+(B),
                              \mathfrak h_Q^-(B)\}.
 \label{sch:eq:Haar-pair}
\end{equation}

The two terms in \eqref{sch:eq:Haar-pair} keep the primal and adjoint
estimates separate.  At exponent $2$, however, they are equivalent.

\begin{lemma}
\label{sch:lem:two-Haar-comparable}
For every $Q$,
$$
 \mathfrak h_Q^+(B)\simeq\mathfrak h_Q^-(B).
$$
The constants depend only on
$[U]_{\mathcal A_2}$ and $[V]_{\mathcal A_2}$.
\end{lemma}

\begin{proof}
For one Haar coefficient $A=\widehat B(Q,\varepsilon)/|Q|^{1/2}$,
$
 \|\mathcal U'_QA^*(\mathcal V'_Q)^{-1}\|_{\HS}
 =
 \|(\mathcal V'_Q)^{-1}A\mathcal U'_Q\|_{\HS}.
$
Put $X=\mathcal V_QA\mathcal U_Q^{-1}$.  Then
$$
 (\mathcal V'_Q)^{-1}A\mathcal U'_Q
 =
 (\mathcal V'_Q)^{-1}\mathcal V_Q^{-1}
 X\mathcal U_Q\mathcal U'_Q.
$$
Jensen's matrix inequality gives
$$
 \|(\mathcal V'_Q)^{-1}\mathcal V_Q^{-1}\|_{\op}\leq1,
 \qquad
 \|(\mathcal U_Q\mathcal U'_Q)^{-1}\|_{\op}\leq1.
$$
The matrix $\mathcal A_2$ condition gives
$$
 \|\mathcal V_Q\mathcal V'_Q\|_{\op}
 \leq[V]_{\mathcal A_2}^{1/2},
 \qquad
 \|\mathcal U_Q\mathcal U'_Q\|_{\op}
 \leq[U]_{\mathcal A_2}^{1/2}.
$$
Multiplication on the left or right is bounded on the
Hilbert--Schmidt class by the corresponding operator norm.  These four
estimates prove the result for each $\varepsilon$.  Sum their squares.
\end{proof}

For one martingale generation, put
\begin{align}
 \rho_Q^+(B)&=\bigg(\avg_Q
 \|V^{1/2}(x)\Delta_QB(x)\mathcal U_Q^{-1}\|_{\HS}^2\,dx
 \bigg)^{1/2},\label{sch:eq:rho-plus}\\
 \rho_Q^-(B)&=\bigg(\avg_Q
 \|U^{-1/2}(x)(\Delta_QB(x))^*(\mathcal V'_Q)^{-1}\|_{\HS}^2\,dx
 \bigg)^{1/2}.\label{sch:eq:rho-minus}
\end{align}

\begin{lemma}
\label{sch:lem:scalar-projections}
Let $W\in\mathcal A_2$.  For every nonzero $e\in\C^m$, the function
$$
 w_e(x)=\langle W(x)e,e\rangle
$$
belongs to the scalar class $A_2$ and 
$
 [w_e]_{A_2}\leq m[W]_{\mathcal A_2}.
$

Consequently, there are $C_W<\infty$ and $\eta_W>0$, independent
of $e$, such that
\begin{equation}
 \frac{\int_Ew_e}{\int_Qw_e}
 \leq C_W\bigg(\frac{|E|}{|Q|}\bigg)^{\eta_W},
 \qquad E\subset Q.
 \label{sch:eq:scalarized-Ainfty}
\end{equation}
\end{lemma}

\begin{proof}
Fix a cube $Q$ and put $M=\langle W\rangle_Q$ and
$a=M^{1/2}e$.  Since $|a|^2=\avg_Qw_e$,
\begin{align*}
 \frac1{w_e(x)}
 &=\frac1{|W^{1/2}(x)M^{-1/2}a|^2}
 \leq\frac{\|M^{1/2}W^{-1/2}(x)\|_{\op}^2}{|a|^2}.
\end{align*}
It follows that
\begin{align*}
 \bigg(\avg_Qw_e\bigg)\bigg(\avg_Qw_e^{-1}\bigg)
 &\leq\avg_Q\|M^{1/2}W^{-1/2}(x)\|_{\op}^2
 \leq\operatorname{tr}\!\bigg(M^{1/2}
           \langle W^{-1}\rangle_QM^{1/2}\bigg)
 \leq m[W]_{\mathcal A_2}.
\end{align*}
The scalar $A_\infty$ estimate for $w_e$ gives
\eqref{sch:eq:scalarized-Ainfty}.
\end{proof}

\begin{lemma}
\label{sch:lem:local-haar}
For every dyadic cube $Q$,
\begin{equation}
 \rho_Q^+(B)\simeq \mathfrak h_Q^+(B),
 \qquad
 \rho_Q^-(B)\simeq \mathfrak h_Q^-(B).
 \label{sch:eq:local-haar-equivalence}
\end{equation}
The constants depend only on $d,m$, the matrix $\mathcal A_2$ constants
of $U,V$ and  the chosen normalization of the Haar functions.
\end{lemma}

\begin{proof}
We prove the first estimate.  Let $\operatorname{ch}(Q)$ be the
$2^d$ dyadic children of $Q$.  The function $\Delta_QB$ is constant on
each child.  Write this value as $D_E$ on $E\in\operatorname{ch}(Q)$.
Then
\begin{align}
 \rho_Q^+(B)^2
 =\frac1{|Q|}\sum_{E\in\operatorname{ch}(Q)}
   \int_E\|V^{1/2}(x)D_E\mathcal U_Q^{-1}\|_{\HS}^2\,dx=2^{-d}\sum_{E\in\operatorname{ch}(Q)}
   \|\langle V\rangle_E^{1/2}D_E\mathcal U_Q^{-1}\|_{\HS}^2.
 \label{sch:eq:child-trace}
\end{align}
By Lemma~\ref{sch:lem:scalar-projections}, the scalar function
$x\mapsto\langle V(x)e,e\rangle$ is an $A_2$ weight with a constant
uniform in $e$.  Scalar doubling, applied to a child
$E$ and its parent $Q$, gives
$$
 c\,\langle V\rangle_Q\leq \langle V\rangle_E
 \leq 2^d\langle V\rangle_Q
$$
in the Loewner order.  Hence \eqref{sch:eq:child-trace} is comparable to
$$
 2^{-d}\sum_{E\in\operatorname{ch}(Q)}
 \|\mathcal V_QD_E\mathcal U_Q^{-1}\|_{\HS}^2.
$$
The map from the mean-zero child values $\{D_E\}_E$ to the Haar
coefficients
$\{|Q|^{-1/2}\widehat B(Q,\varepsilon)\}_{\varepsilon\ne\mathbf1}$
is a fixed unitary map, up to the harmless normalization of the Haar
basis.  Apply it separately to every scalar entry of
$\mathcal V_QD_E\mathcal U_Q^{-1}$.  Parseval's identity gives the
first estimate in \eqref{sch:eq:local-haar-equivalence}.

For the second estimate, repeat the argument with the matrix weight
$U^{-1}$ and the child values $D_E^*$.  The inverse of a matrix
$\mathcal A_2$ weight is again a matrix $\mathcal A_2$ weight.  This
gives the second estimate.
\end{proof}

We now take the dyadic Schatten--Besov seminorm to be
\begin{equation}
 \|B\|_{\mathfrak b^{r,q}_{V,U}(\D)}
 =\|\{\mathfrak h_Q(B):Q\in\D\}\|_{\ell^{r,q}},
 \qquad
 \|B\|_{\mathfrak b^{r}_{V,U}(\D)}
 =\|B\|_{\mathfrak b^{r,r}_{V,U}(\D)}.
 \label{sch:eq:dyadic-Haar-Besov}
\end{equation}
The continuous seminorm is the sum over a fixed finite family of
adjacent grids:
\begin{equation}
 \|B\|_{\mathfrak{SB}^{r,q}_{V,U}}
 =\sum_{t=1}^{3^d}\|B\|_{\mathfrak b^{r,q}_{V,U}(\D^t)}.
 \label{sch:eq:continuous-Haar-Besov}
\end{equation}
We write $\mathfrak{SB}^r_{V,U}=\mathfrak{SB}^{r,r}_{V,U}$ and
$\mathfrak{SW}^d_{V,U}=\mathfrak{SB}^{d,\infty}_{V,U}$.
Lemma~\ref{sch:lem:local-haar} shows that this definition
is  the $\ell^{r,q}$ norm of the two one-generation weighted
martingale oscillations.  No matrix John--Nirenberg theorem is needed
for this identity.

\begin{theorem}
\label{sch:prop:matrix-besov-JN}
For $1<r<\infty$ and $0<q\leq\infty$,
\begin{equation}
 \sum_{t=1}^{3^d}
 \|\{\mathfrak h_Q(B):Q\in\D^t\}\|_{\ell^{r,q}}
 \simeq
 \|B\|_{\mathfrak B^{r,q}_{V,U}}.
 \label{sch:eq:matrix-besov-JN}
\end{equation}
The constants may depend on $d,m,r,q,[U]_{\mathcal A_2}$ and
$[V]_{\mathcal A_2}$, but not on $B$.
\end{theorem}

\begin{proof}
It is enough to work with one dyadic system $\D$.  Let $P\subset Q$
be dyadic cubes.  Define
\begin{equation}
 \vartheta_U(P,Q)
 =\frac1{|Q|}\operatorname{tr}\!\bigg(
 \mathcal U_Q^{-1}U(P)\mathcal U_Q^{-1}\bigg),
 \qquad U(P)=\int_PU(x)\,dx.
 \label{sch:eq:theta-U}
\end{equation}
For every $k\geq0$,
\begin{equation}
 \sum_{P\in\operatorname{ch}^k(Q)}\vartheta_U(P,Q)=m.
 \label{sch:eq:theta-sum}
\end{equation}
To obtain decay, choose the standard basis $e_1,\ldots,e_m$ and apply
\eqref{sch:eq:scalarized-Ainfty} to
$$
 x\longmapsto
 \langle U(x)\mathcal U_Q^{-1}e_j,
                 \mathcal U_Q^{-1}e_j\rangle.
$$
Its integral over $Q$ is $|Q|$.  Hence, if
$P\in\operatorname{ch}^k(Q)$,
\begin{equation}
 0\leq\vartheta_U(P,Q)\leq C\,2^{-kd\eta_U}.
 \label{sch:eq:theta-decay}
\end{equation}
Also,
\begin{align}
 \|\mathcal U_P\mathcal U_Q^{-1}\|_{\op}^2
 &\leq
 \operatorname{tr}\!\bigg(
 \mathcal U_Q^{-1}\langle U\rangle_P
 \mathcal U_Q^{-1}\bigg)
 =\frac{|Q|}{|P|}\vartheta_U(P,Q).
 \label{sch:eq:reducing-matrix-comparison}
\end{align}

We prove the upper estimate for the primal oscillation.  For
$P\in\D$, write $D_P=\Delta_PB$.  For $Q\in\D$ and $k\geq0$, put
$$
 S_{Q,k}=\sum_{P\in\operatorname{ch}^k(Q)}D_P.
$$
The cubes in this sum are disjoint.  Lemma~\ref{sch:lem:local-haar} and
\eqref{sch:eq:reducing-matrix-comparison} give
\begin{align}
 \frac1{|Q|}\int_Q
 \|V^{1/2}(x)S_{Q,k}(x)\mathcal U_Q^{-1}\|_{\HS}^2\,dx&=\sum_{P\in\operatorname{ch}^k(Q)}\frac{|P|}{|Q|}
 \avg_P\|V^{1/2}D_P\mathcal U_P^{-1}
              \mathcal U_P\mathcal U_Q^{-1}\|_{\HS}^2\notag\\
 &\leq
 C\sum_{P\in\operatorname{ch}^k(Q)}
 \vartheta_U(P,Q)\,\mathfrak h_P^+(B)^2.
 \label{sch:eq:generation-bound}
\end{align}
For a nonnegative sequence $c=\{c_P\}_{P\in\D}$, define
$$
 \mathcal H_kc(Q)=
 \bigg(\sum_{P\in\operatorname{ch}^k(Q)}
 \vartheta_U(P,Q)c_P^2\bigg)^{1/2}.
$$
We next estimate this operator on sequence spaces.

Let $s\geq2$.  By \eqref{sch:eq:theta-sum},
\begin{align*}
 \|\mathcal H_kc\|_{\ell^s}^s
 &\leq m^{s/2-1}
 \sum_Q\sum_{P\in\operatorname{ch}^k(Q)}
 \vartheta_U(P,Q)c_P^s
 \leq C2^{-kd\eta_U}\sum_Pc_P^s.
\end{align*}
Each $P$ has only one ancestor of generation $k$ in the last sum.  Thus
\begin{equation}
 \|\mathcal H_kc\|_{\ell^s}
 \leq C_s2^{-kd\eta_U/s}\|c\|_{\ell^s},
 \qquad s\geq2.
 \label{sch:eq:Hk-large-s}
\end{equation}
If $0<s<2$, then $(\sum_jx_j)^{s/2}\leq\sum_jx_j^{s/2}$.
By \eqref{sch:eq:theta-decay},
\begin{align*}
 \|\mathcal H_kc\|_{\ell^s}^s
 &\leq\sum_Q\sum_{P\in\operatorname{ch}^k(Q)}
 \vartheta_U(P,Q)^{s/2}c_P^s
 \leq C_s2^{-kd\eta_Us/2}\sum_Pc_P^s.
\end{align*}
Consequently,
\begin{equation}
 \|\mathcal H_kc\|_{\ell^s}
 \leq C_s2^{-kd\eta_U/2}\|c\|_{\ell^s},
 \qquad 0<s<2.
 \label{sch:eq:Hk-small-s}
\end{equation}
Choose $1<s_0<r<s_1$.  The real $K$-method and Holmstedt's formula
\cite{Holmstedt} apply for $0<q\leq\infty$.  Interpolating
\eqref{sch:eq:Hk-large-s} and \eqref{sch:eq:Hk-small-s} gives numbers
$C<\infty$ and $\gamma>0$ such that
\begin{equation}
 \|\mathcal H_kc\|_{\ell^{r,q}}
 \leq C2^{-k\gamma}\|c\|_{\ell^{r,q}}.
 \label{sch:eq:Hk-Lorentz}
\end{equation}
This also covers $q=\infty$.

For a finite martingale expansion,
$$
 (B-B_Q)\one_Q=\sum_{k\geq0}S_{Q,k},
$$
where only finitely many terms occur.  Minkowski's inequality and
\eqref{sch:eq:generation-bound} give
\begin{equation}
 \alpha_Q^{\HS}(B)\leq C\sum_{k\geq0}
 \mathcal H_k(\mathfrak h^+(B))(Q).
 \label{sch:eq:aQ-Hk}
\end{equation}
Let $0<\tau\leq\min\{1,q\}$ be an exponent for the Lorentz
quasi-triangle inequality.  Raise the quasi-norm to the power
$\tau$, apply that inequality to the sum over $k$ and  use the
geometric factor in \eqref{sch:eq:Hk-Lorentz}.  More precisely,
$$
 \|\{\alpha_Q^{\HS}(B)\}\|_{\ell^{r,q}}^\tau
 \leq C\sum_{k\geq0}
 \|\mathcal H_k(\mathfrak h^+(B))\|_{\ell^{r,q}}^\tau
 \leq C\|\{\mathfrak h_Q^+(B)\}\|_{\ell^{r,q}}^\tau
 \sum_{k\geq0}2^{-k\gamma\tau}.
$$
This gives
\begin{equation}
 \|\{\alpha_Q^{\HS}(B)\}_{Q\in\D}\|_{\ell^{r,q}}
 \leq C\|\{\mathfrak h_Q^+(B)\}_{Q\in\D}\|_{\ell^{r,q}}.
 \label{sch:eq:primal-global-upper}
\end{equation}
For a general locally integrable $B$, define
$$
 \alpha_{Q,N}^{\HS}
 =\bigg(\frac1{|Q|}\int_Q
 \bigg\|V^{1/2}(x)
 \sum_{k=0}^NS_{Q,k}(x)\mathcal U_Q^{-1}
 \bigg\|_{\HS}^2\,dx\bigg)^{1/2}.
$$
On a finite set of cubes, the preceding argument gives the
sequence estimate for $\{\alpha_{Q,N}^{\HS}\}$, uniformly in $N$.  Exhaust
the countable grid to obtain the same estimate on all of $\D$.
Ordinary martingale convergence gives
$$
 \sum_{k=0}^NS_{Q,k}(x)\longrightarrow B(x)-B_Q
 \quad\text{for almost every }x\in Q.
$$
Hence $\alpha_Q^{\HS}(B)\leq\liminf_N\alpha_{Q,N}^{\HS}$ by Fatou's lemma.  The Fatou
property of $\ell^{r,q}$, followed by the uniform estimate, proves
\eqref{sch:eq:primal-global-upper}.

Apply the same proof to $B^*$, with $U^{-1}$ in place of $V$ and
$V^{-1}$ in place of $U$.  Since the inverse of a matrix
$\mathcal A_2$ weight belongs to $\mathcal A_2$, we obtain
\begin{equation}
 \|\{\beta_Q^{\HS}(B)\}_{Q\in\D}\|_{\ell^{r,q}}
 \leq C\|\{\mathfrak h_Q^-(B)\}_{Q\in\D}\|_{\ell^{r,q}}.
 \label{sch:eq:dual-global-upper}
\end{equation}

It remains to prove the reverse estimates.  Put
$F=(B-B_Q)\mathcal U_Q^{-1}$.  Then $F_Q=0$ and, on every child
$E$ of $Q$,
$$
 \Delta_QF=F_E=\avg_EF.
$$
The matrix Cauchy--Schwarz inequality gives
\begin{align*}
 \|\langle V\rangle_E^{1/2}F_E\|_{\HS}^2
 &\leq
 \|\langle V\rangle_E^{1/2}
       \langle V^{-1}\rangle_E^{1/2}\|_{\op}^2
 \avg_E\|V^{1/2}(x)F(x)\|_{\HS}^2
 \leq [V]_{\mathcal A_2}
 \avg_E\|V^{1/2}(x)F(x)\|_{\HS}^2.
\end{align*}
Sum over the children.  By Lemma~\ref{sch:lem:local-haar},
\begin{equation}
 \mathfrak h_Q^+(B)\leq C\rho_Q^+(B)\leq C\alpha_Q^{\HS}(B).
 \label{sch:eq:primal-reverse}
\end{equation}
The same calculation with $B^*$, $U^{-1}$ and  $V^{-1}$ gives
\begin{equation}
 \mathfrak h_Q^-(B)\leq C\beta_Q^{\HS}(B).
 \label{sch:eq:dual-reverse}
\end{equation}
Combining \eqref{sch:eq:primal-global-upper}--\eqref{sch:eq:dual-reverse} proves the result on $\D$.  Sum over the adjacent dyadic systems.
\end{proof}

\begin{corollary}
For $1<r<\infty$ and $0<q\leq\infty$,
$
 \mathfrak B^{r,q}_{V,U}=\mathfrak{SB}^{r,q}_{V,U}
$
with equivalent seminorms.  In particular,
$
 \mathfrak W^d_{V,U}=\mathfrak{SW}^d_{V,U}.
$
\end{corollary}

\subsection{NWO estimates}

We use one scalar result of Rochberg and Semmes.  All extensions to
$\C^m$, all matrix coefficients and  all Lorentz estimates are proved
in this section.  The matrix-weight maximal estimates are proved below
before each application.

Let $\{e_\xi\}_{\xi\in\Xi}$ be a family in
$L^2(\R^d;\C^n)$.  Assign a cube $S_\xi$ to every index and suppose
that $\operatorname{supp} e_\xi\subset S_\xi$.  Define
\begin{equation}
 \mathcal M_eh(x)
 =\sup_{\xi\in\Xi}
 \frac{|\langle h,e_\xi\rangle|}{|S_\xi|^{1/2}}
 \one_{S_\xi}(x).
 \label{sch:eq:NWO-maximal-operator}
\end{equation}
We call $\{e_\xi\}$ an NWO family if
\begin{equation}
 \|\mathcal M_eh\|_2\leq C_e\|h\|_2.
 \label{sch:eq:NWO-definition}
\end{equation}
The controlling cubes may be drawn from a fixed finite collection of
dyadic systems.  Repeated controlling cubes are also allowed
with fixed finite multiplicity.  These changes affect only the
constant in \eqref{sch:eq:NWO-definition}.  When an operator-ideal
estimate is applied, we partition the indices by dyadic system and by
a finite repetition color.  We then apply the scalar estimate on each
subfamily and use the finite quasi-triangle inequality.

For vectors $x$ and $y$ in Hilbert spaces, we use the convention
$$
 (x\otimes y)h=\langle h,y\rangle x.
$$

\subsubsection{The scalar theorem and its Lorentz consequence}

We first state the only NWO operator-ideal estimate taken from
Rochberg--Semmes \cite[(1.10) and Section~3]{RS}.

\begin{theorem}
\label{sch:thm:scalar-NWO}
Let $\{e_\xi\}$ and $\{f_\xi\}$ be scalar NWO families in two
copies of $L^2(\R^d)$.  Assume that, for every $\xi$, both
functions are supported in the same controlling cube $S_\xi$ and 
that both NWO bounds are defined with this cube.  If
$1<s<\infty$ and $A$ is a compact operator between these spaces,
then
\begin{equation}
 \bigg(\sum_{\xi\in\Xi}
 |\langle Ae_\xi,f_\xi\rangle|^s\bigg)^{1/s}
 \leq C_s\|A\|_{\Sch^s}.
 \label{sch:eq:scalar-NWO-analysis}
\end{equation}
The constant depends on $s$ and on the two NWO constants.
\end{theorem}

The synthesis estimate needed below follows from
Theorem~\ref{sch:thm:scalar-NWO}.  We include the argument.  Let
$\lambda=\{\lambda_\xi\}$ have finite support and put
$$
 T_\lambda=\sum_\xi\lambda_\xi f_\xi\otimes e_\xi.
$$
Fix $1<s<\infty$ and write $s'=s/(s-1)$.  For a finite-rank
operator $A$ between these spaces,
\begin{align*}
 |\operatorname{Tr}(A^*T_\lambda)|
 &=\left|\sum_\xi\lambda_\xi
       \overline{\langle Ae_\xi,f_\xi\rangle}\right|
 \leq \|\lambda\|_{\ell^s}
 \bigg(\sum_\xi|\langle Ae_\xi,f_\xi\rangle|^{s'}\bigg)^{1/s'}
 \leq C_{s'}\|\lambda\|_{\ell^s}\|A\|_{\Sch^{s'}}.
\end{align*}
Trace duality gives
\begin{equation}
 \|T_\lambda\|_{\Sch^s}
 \leq C_s\|\lambda\|_{\ell^s}.
 \label{sch:eq:scalar-NWO-synthesis}
\end{equation}
Indeed, finite-rank operators are dense in $\Sch^{s'}$, so it is
enough to take the supremum over such $A$ with
$\|A\|_{\Sch^{s'}}\leq1$.

We next pass to Lorentz ideals.  This step will be used at the weak
critical index.  Fix
$$
 1<s_0<r<s_1<\infty,
 \qquad
 \frac1r=\frac{1-\theta}{s_0}+\frac\theta{s_1}.
$$
Apply the same real $K$-method to the linear maps
$
 \lambda\longmapsto T_\lambda,
 \ 
 A\longmapsto\{\langle Ae_\xi,f_\xi\rangle\}_\xi.
$
The identities
\begin{align*}
 (\ell^{s_0},\ell^{s_1})_{\theta,q}
 &=\ell^{r,q}\qquad{\rm and}
\qquad (\Sch^{s_0},\Sch^{s_1})_{\theta,q}
 =\Sch^{r,q}
\end{align*}
follow from the corresponding $K$-functional formulas for decreasing
sequences and singular values.  When $0<q<1$, the resulting spaces are
quasi-Banach spaces.  These identities give, for $1<r<\infty$ and
$0<q\leq\infty$,
\begin{align}
 \left\|\sum_\xi\lambda_\xi f_\xi\otimes e_\xi
 \right\|_{\Sch^{r,q}}
 &\leq C_{r,q}\|\lambda\|_{\ell^{r,q}},
 \label{sch:eq:scalar-NWO-Lorentz-synthesis}\\
 \big\|\{\langle Ae_\xi,f_\xi\rangle\}_\xi
 \big\|_{\ell^{r,q}}
 &\leq C_{r,q}\|A\|_{\Sch^{r,q}}.
 \label{sch:eq:scalar-NWO-Lorentz-analysis}
\end{align}
For each endpoint $s_i$, the synthesis estimate on finitely
supported sequences extends by completion to a bounded map
$\ell^{s_i}\to\Sch^{s_i}$.  We interpolate these extensions.  When
$q=\infty$, only the finite-sequence synthesis estimate is used
below.
The first estimate is initially stated for finitely supported
$\lambda$.  The second follows for an infinite family by applying the
finite estimate to increasing finite sets and using the Fatou property
of $\ell^{r,q}$.

We now prove the vector-valued form used in the paper.  

\begin{lemma}
\label{sch:lem:finite-fibre-NWO}
Let
$
 H_n=L^2(\R^d;\C^n),
 \ H_m=L^2(\R^d;\C^m).
$
Let $\{e_\xi\}\subset H_n$ and
$\{f_\xi\}\subset H_m$ be NWO families with a common controlling
cube $S_\xi$ at every index.  If
$1<r<\infty$ and $0<q\leq\infty$, then
for every finitely supported scalar sequence $\lambda$ and every
$A\in\Sch^{r,q}(H_n,H_m)$,
\begin{align}
 \left\|\sum_\xi\lambda_\xi f_\xi\otimes e_\xi
 \right\|_{\Sch^{r,q}(H_n,H_m)}
 &\leq C\|\lambda\|_{\ell^{r,q}},
 \label{sch:eq:finite-fibre-synthesis}\\
 \big\|\{\langle Ae_\xi,f_\xi\rangle\}_\xi
 \big\|_{\ell^{r,q}}
 &\leq C\|A\|_{\Sch^{r,q}(H_n,H_m)}.
 \label{sch:eq:finite-fibre-analysis}
\end{align}
The constant depends only on
$m,n,r,q$ and the two NWO constants.
\end{lemma}

\begin{proof}
Let $\{\mathbf e_a^{(m)}\}_{a=1}^m$ and
$\{\mathbf e_b^{(n)}\}_{b=1}^n$ be the standard bases.  Write
$$
 f_\xi=\sum_{a=1}^mf_\xi^a\mathbf e_a^{(m)},
 \qquad
 e_\xi=\sum_{b=1}^ne_\xi^b\mathbf e_b^{(n)}.
$$
Let $I_a^{(m)}:L^2\to H_m$ and
$P_a^{(m)}:H_m\to L^2$ be the coordinate injection and projection
in the range.  Define $I_b^{(n)}:L^2\to H_n$ and
$P_b^{(n)}:H_n\to L^2$ in the domain by
$$
 I_a^{(m)}h=h\mathbf e_a^{(m)},
 \quad P_a^{(m)}H=H_a,
 \qquad
 I_b^{(n)}h=h\mathbf e_b^{(n)},
 \quad P_b^{(n)}H=H_b.
$$
Every coordinate family is scalar NWO.  Indeed, if
$h\in L^2(\R^d)$, then
$$
 |\langle h,e_\xi^b\rangle|
 =|\langle I_b^{(n)}h,e_\xi\rangle|.
$$
Thus
$$
 \mathcal M_{e^b}h\leq\mathcal M_e(I_b^{(n)}h),
 \qquad
 \|\mathcal M_{e^b}h\|_2\leq C_e\|h\|_2.
$$
Likewise,
$$
 |\langle h,f_\xi^a\rangle|
 =|\langle I_a^{(m)}h,f_\xi\rangle|,
$$
so every family $\{f_\xi^a\}$ is scalar NWO with constant at most
$C_f$.

The $(a,b)$-block of
$$
 T_\lambda=\sum_\xi\lambda_\xi f_\xi\otimes e_\xi
$$
is
$$
 P_a^{(m)}T_\lambda I_b^{(n)}
 =\sum_\xi\lambda_\xi f_\xi^a\otimes e_\xi^b.
$$
Equation \eqref{sch:eq:scalar-NWO-Lorentz-synthesis} gives
$$
 \|P_a^{(m)}T_\lambda I_b^{(n)}\|_{\Sch^{r,q}}
 \leq C\|\lambda\|_{\ell^{r,q}}.
$$
Since
$$
 T_\lambda=\sum_{a=1}^m\sum_{b=1}^n
 I_a^{(m)}(P_a^{(m)}T_\lambda I_b^{(n)})P_b^{(n)},
$$
the ideal property and the quasi-triangle inequality in
$\Sch^{r,q}$ prove \eqref{sch:eq:finite-fibre-synthesis}.

For the second estimate, put
$A_{ab}=P_a^{(m)}AI_b^{(n)}$.  Then
\begin{equation}
 \langle Ae_\xi,f_\xi\rangle
 =\sum_{a=1}^m\sum_{b=1}^n
 \langle A_{ab}e_\xi^b,f_\xi^a\rangle.
 \label{sch:eq:finite-fibre-coordinate-pairing}
\end{equation}
By \eqref{sch:eq:scalar-NWO-Lorentz-analysis} and the ideal property,
$$
 \big\|\{\langle A_{ab}e_\xi^b,f_\xi^a\rangle\}_\xi
 \big\|_{\ell^{r,q}}
 \leq C\|A_{ab}\|_{\Sch^{r,q}}
 \leq C\|A\|_{\Sch^{r,q}}.
$$
Apply the quasi-triangle inequality to the finite sum in
\eqref{sch:eq:finite-fibre-coordinate-pairing}.  This proves
\eqref{sch:eq:finite-fibre-analysis}.
\end{proof}

\subsubsection{Matrix-valued families and matrix coefficients}

Let $F_\xi$ be an $m\times n$ matrix-valued function and  let
$S_\xi^F$ contain its support.  Define
\begin{equation}
 \mathcal M_Fh(x)
 =\sup_\xi
 \frac{\left|\int_{\R^d}F_\xi(y)^*h(y)\,dy\right|_{\C^n}}
 {|S_\xi^F|^{1/2}}\one_{S_\xi^F}(x).
 \label{sch:eq:matrix-NWO-maximal}
\end{equation}
Suppose that
\begin{equation}
 \|\mathcal M_Fh\|_2\leq C_F\|h\|_2.
 \label{sch:eq:matrix-NWO-maximal-bound}
\end{equation}
For any unit vectors $u_\xi\in\C^n$, which may depend on $\xi$,
\begin{align*}
 |\langle h,F_\xi u_\xi\rangle|
 &=\left|u_\xi^*\int_{\R^d}F_\xi(y)^*h(y)\,dy\right|
 \leq\left|\int_{\R^d}F_\xi(y)^*h(y)\,dy\right|_{\C^n}.
\end{align*}
Hence $\{F_\xi u_\xi\}$ is an NWO family with constant at most
$C_F$.  This observation permits the singular vectors below to vary
with every index.

We now prove the form used in the Riesz upper estimate.  The two
families use one common controlling cube at every index.

\begin{proposition}
\label{sch:prop:matrix-NWO}
Let $\mathcal K$ be finite.  Let $F_{Q,\kappa}$ and
$G_{Q,\kappa}$ be $m\times m$ matrix-valued functions, with a
common controlling cube $S_{Q,\kappa}$, such that
$$
 \operatorname{supp}F_{Q,\kappa}\subset S_{Q,\kappa},
 \qquad
 \operatorname{supp}G_{Q,\kappa}\subset S_{Q,\kappa}.
$$
Suppose that
\begin{align}
 \left\|\sup_{Q,\kappa}
 \frac{\left|\int F_{Q,\kappa}(y)^*f(y)\,dy\right|}
 {|S_{Q,\kappa}|^{1/2}}
 \one_{S_{Q,\kappa}}\right\|_2
 &\leq C_F\|f\|_2,
 \label{sch:eq:matrix-NWO-F-bound}\\
 \left\|\sup_{Q,\kappa}
 \frac{\left|\int G_{Q,\kappa}(y)^*g(y)\,dy\right|}
 {|S_{Q,\kappa}|^{1/2}}
 \one_{S_{Q,\kappa}}\right\|_2
 &\leq C_G\|g\|_2.
 \label{sch:eq:matrix-NWO-G-bound}
\end{align}
Let $A_{Q,\kappa}\in\C^{m\times m}$.  For a finite set
$\mathcal F$ of pairs $(Q,\kappa)$, put
$$
 \Theta_{F,A,G}^{\mathcal F}h
 =\sum_{(Q,\kappa)\in\mathcal F}
 F_{Q,\kappa}A_{Q,\kappa}
 \int_{\R^d}G_{Q,\kappa}(y)^*h(y)\,dy.
$$
Then, for $1<r<\infty$ and $0<q\leq\infty$,
\begin{equation}
 \|\Theta_{F,A,G}^{\mathcal F}\|_{\Sch^{r,q}}
 \leq C
 \left\|\left\{
 \bigg(\sum_{\kappa\in\mathcal K}
 \|A_{Q,\kappa}\|_{\HS}^2\bigg)^{1/2}
 \right\}_Q\right\|_{\ell^{r,q}}.
 \label{sch:eq:matrix-NWO-upper}
\end{equation}
The constant is independent of $\mathcal F$ and depends only on
$r,q,m,|\mathcal K|,C_F$ and  $C_G$.
\end{proposition}

\begin{proof}
Replace $A_{Q,\kappa}$ by zero when
$(Q,\kappa)\notin\mathcal F$.  This does not change
$\Theta_{F,A,G}^{\mathcal F}$ and  it can only decrease the
right-hand side of \eqref{sch:eq:matrix-NWO-upper}.  We may therefore
assume that the coefficient family is finitely supported.

For every $(Q,\kappa)$, choose a singular-value decomposition
$$
 A_{Q,\kappa}
 =\sum_{\rho=1}^m s_{Q,\kappa,\rho}
 u_{Q,\kappa,\rho}v_{Q,\kappa,\rho}^*.
$$
Then
\begin{equation}
 \Theta_{F,A,G}^{\mathcal F}
 =\sum_{(Q,\kappa)\in\mathcal F}\sum_{\rho=1}^m
 s_{Q,\kappa,\rho}
 (F_{Q,\kappa}u_{Q,\kappa,\rho})
 \otimes(G_{Q,\kappa}v_{Q,\kappa,\rho}).
 \label{sch:eq:matrix-NWO-SVD-expansion}
\end{equation}
The singular vectors depend on $(Q,\kappa,\rho)$.  Equations
\eqref{sch:eq:matrix-NWO-F-bound} and
\eqref{sch:eq:matrix-NWO-G-bound}, together with
\eqref{sch:eq:matrix-NWO-maximal}--
\eqref{sch:eq:matrix-NWO-maximal-bound}, show that the two vector
families in \eqref{sch:eq:matrix-NWO-SVD-expansion} are NWO with
uniform constants.  Lemma~\ref{sch:lem:finite-fibre-NWO} gives
\begin{equation}
 \|\Theta_{F,A,G}^{\mathcal F}\|_{\Sch^{r,q}}
 \leq C
 \|\{s_{Q,\kappa,\rho}\}_{Q,\kappa,\rho}\|_{\ell^{r,q}}.
 \label{sch:eq:matrix-NWO-flattened}
\end{equation}

It remains to group the finite set of singular values over each
$Q$.  Put $N=m|\mathcal K|$ and
$$
 a_Q=\bigg(\sum_{\kappa,\rho}s_{Q,\kappa,\rho}^2\bigg)^{1/2}
 =\bigg(\sum_\kappa\|A_{Q,\kappa}\|_{\HS}^2\bigg)^{1/2}.
$$
For $t>0$, let
$$
 n_a(t)=\#\{Q:a_Q>t\},
 \qquad
 n_s(t)=\#\{(Q,\kappa,\rho):s_{Q,\kappa,\rho}>t\}.
$$
Since $s_{Q,\kappa,\rho}\leq a_Q$,
$$
 n_s(t)\leq N n_a(t).
$$
If $a_Q>t$, at least one singular value satisfies
$s_{Q,\kappa,\rho}>t/\sqrt N$.  Hence
$$
 n_a(t)\leq n_s(t/\sqrt N).
$$
The distribution-function description of Lorentz sequence norms now
gives
\begin{equation}
 \|\{s_{Q,\kappa,\rho}\}\|_{\ell^{r,q}}
 \simeq_{m,|\mathcal K|,r,q}
 \|\{a_Q\}\|_{\ell^{r,q}}.
 \label{sch:eq:finite-block-Lorentz-equivalence}
\end{equation}
Combining \eqref{sch:eq:matrix-NWO-flattened} and
\eqref{sch:eq:finite-block-Lorentz-equivalence} proves
\eqref{sch:eq:matrix-NWO-upper}.
\end{proof}

We also record the analysis form.  It will be used in the lower
estimate.

\begin{lemma}
\label{sch:lem:NWO-diagonal}
Let $\{e_\xi\}$ and $\{f_\xi\}$ be NWO families in
$L^2(\R^d;\C^n)$ and $L^2(\R^d;\C^m)$, respectively, with a common
controlling cube at every index.  If
$1<r<\infty$, $0<q\leq\infty$ and 
$$
 A\in\Sch^{r,q}\bigl(
 L^2(\R^d;\C^n),L^2(\R^d;\C^m)\bigr),
$$
then
\begin{equation}
 \big\|\{\langle Ae_\xi,f_\xi\rangle\}_\xi
 \big\|_{\ell^{r,q}}
 \leq C\|A\|_{\Sch^{r,q}}.
 \label{sch:eq:NWO-diagonal}
\end{equation}
The same estimate holds for a fixed finite collection of pairs of NWO
families.
\end{lemma}

\begin{proof}
Equation \eqref{sch:eq:NWO-diagonal} is
\eqref{sch:eq:finite-fibre-analysis}.  For a finite collection, apply
that estimate to each pair, place the resulting sequences on disjoint
copies of the index set and  use the finite quasi-triangle inequality
in $\ell^{r,q}$.  The number of pairs is fixed, so the resulting
constant is independent of $A$.
\end{proof}

\subsection{Matrix-weight-normalized NWO families}

We now prove the maximal estimates used in the Riesz upper and lower
bounds.  We begin with an $L^{2+\sigma}$ criterion.

\begin{lemma}
\label{sch:lem:NWO-Lp-criterion}
Let $0<\sigma<\infty$ and  let $H$ be a finite-dimensional
Hilbert space.  Suppose that $e_Q\in L^{2+\sigma}(\R^d;H)$ is
supported on a cube $S_Q$ and
\begin{equation}
 \|e_Q\|_{2+\sigma}
 \leq C_0|S_Q|^{1/(2+\sigma)-1/2}.
 \label{sch:eq:NWO-Lp-size}
\end{equation}
Then, for every $f\in L^2(\R^d;H)$,
\begin{equation}
 \left\|\sup_Q
 \frac{|\langle f,e_Q\rangle|}{|S_Q|^{1/2}}
 \one_{S_Q}\right\|_2
 \leq C_{\sigma,d}C_0\|f\|_2.
 \label{sch:eq:NWO-maximal-from-Lp}
\end{equation}
The cubes $S_Q$ need not be distinct.
\end{lemma}

\begin{proof}
Put $s=2+\sigma$ and $s'=s/(s-1)<2$.  For $x\in S_Q$,
H\"older's inequality and \eqref{sch:eq:NWO-Lp-size} give
$$
 \frac{|\langle f,e_Q\rangle|}{|S_Q|^{1/2}}
 \leq C_0\bigg(\avg_{S_Q}|f|_H^{s'}\bigg)^{1/s'}
 \leq C_0\bigl(M(|f|_H^{s'})(x)\bigr)^{1/s'}.
$$
Since $2/s'>1$, the Hardy--Littlewood maximal theorem proves
\eqref{sch:eq:NWO-maximal-from-Lp}.
\end{proof}

The next lemma gives the matrix maximal estimate required in
Proposition~\ref{sch:prop:matrix-NWO}.

\begin{lemma}
\label{sch:lem:matrix-column-NWO}
Let $F_Q$ be $m\times m$ matrix-valued functions supported on
$S_Q$.  Suppose that, for some $s>2$,
\begin{equation}
 \|F_Qu\|_{L^s}
 \leq C_0|S_Q|^{1/s-1/2}|u|,
 \qquad u\in\C^m.
 \label{sch:eq:matrix-column-Ls}
\end{equation}
Then
\begin{equation}
 \left\|\sup_Q
 \frac{\left|\int_{\R^d}F_Q(y)^*f(y)\,dy\right|_{\C^m}}
 {|S_Q|^{1/2}}\one_{S_Q}\right\|_2
 \leq C_{d,m,s}C_0\|f\|_2.
 \label{sch:eq:matrix-column-maximal}
\end{equation}
Consequently, for every choice of unit vectors $u_Q$, the family
$\{F_Qu_Q\}$ is NWO with a constant independent of the choices.
\end{lemma}

\begin{proof}
Let $\mathcal N_m$ be a fixed $1/2$-net of the unit sphere of
$\C^m$, with $\#\mathcal N_m\leq5^{2m}$.  For every
$z\in\C^m$,
\begin{equation}
 |z|\leq2\max_{u\in\mathcal N_m}|u^*z|.
 \label{sch:eq:finite-net-vector-norm}
\end{equation}
For each $u\in\mathcal N_m$, the family $\{F_Qu\}$ satisfies
Lemma~\ref{sch:lem:NWO-Lp-criterion}.  Hence
\begin{align*}
 &\left\|\sup_Q
 \frac{\left|\int F_Q^*f\right|}{|S_Q|^{1/2}}
 \one_{S_Q}\right\|_2\leq
 2\sum_{u\in\mathcal N_m}
 \left\|\sup_Q
 \frac{|\langle f,F_Qu\rangle|}{|S_Q|^{1/2}}
 \one_{S_Q}\right\|_2
 \leq C_{d,m,s}C_0\|f\|_2.
\end{align*}
This proves \eqref{sch:eq:matrix-column-maximal}.  Finally,
$$
 |\langle f,F_Qu_Q\rangle|
 \leq\left|\int F_Q^*f\right|,
$$
which proves the last assertion.
\end{proof}

We next prove the estimates for the normalized columns of a matrix
weight.  They are uniform in the direction and also hold on a nearby
cube.

\begin{proposition}
\label{sch:prop:concrete-matrix-NWO}
Let $W\in\mathcal A_2$ and  put
$$
 \mathcal W_Q=\langle W\rangle_Q^{1/2},
 \qquad
 \mathcal W'_Q=\langle W^{-1}\rangle_Q^{1/2}.
$$
There is $\sigma_W>0$ with the following property.  Let $Q^*$ be
a cube with $|Q^*|=|Q|$.  Suppose that a
cube $S$ contains $Q\cup Q^*$ and $|S|\leq A|Q|$.  Then, for
every $u\in\C^m$,
\begin{align}
 \|W^{1/2}\mathcal W_Q^{-1}u\|_{L^{2+\sigma_W}(Q^*)}
 &\leq C_{W,A}|Q|^{1/(2+\sigma_W)}|u|,
 \label{sch:eq:directional-primal-RH}\\
 \|W^{-1/2}\mathcal W_Qu\|_{L^{2+\sigma_W}(Q^*)}
 &\leq C_{W,A}|Q|^{1/(2+\sigma_W)}|u|.
 \label{sch:eq:directional-dual-RH}
\end{align}
The same conclusions hold with $Q^*=Q$.  The constants depend only
on $d,m,[W]_{\mathcal A_2}$ and  $A$.
\end{proposition}

\begin{proof}
By homogeneity, assume $|u|=1$.  Set
$$
 w_{Q,u}(x)=|W^{1/2}(x)\mathcal W_Q^{-1}u|^2.
$$
Its average on $Q$ is one.  Lemma~\ref{sch:lem:scalar-projections}
shows that $w_{Q,u}\in A_2$, with a constant independent of $Q$
and $u$.  The scalar $A_2$ inequality gives
\begin{align*}
 \avg_Sw_{Q,u}
 &\leq[w_{Q,u}]_{A_2}
       \bigg(\avg_Sw_{Q,u}^{-1}\bigg)^{-1}
 \leq A[w_{Q,u}]_{A_2}
       \bigg(\avg_Qw_{Q,u}^{-1}\bigg)^{-1}
 \leq A[w_{Q,u}]_{A_2}\avg_Qw_{Q,u}
 \leq C_{W,A}.
\end{align*}
Since $Q^*\subset S$ and $|S|\leq A|Q^*|$,
\begin{equation}
 \avg_{Q^*}w_{Q,u}\leq C_{W,A}.
 \label{sch:eq:nearby-primal-average}
\end{equation}
Scalar reverse H\"older, applied on $Q^*$, gives a number
$\varepsilon_W>0$, independent of $Q,Q^*$ and  $u$, such that
$$
 \bigg(\avg_{Q^*}w_{Q,u}^{1+\varepsilon_W}\bigg)^{1/(1+\varepsilon_W)}
 \leq C_{W,A}.
$$
This proves \eqref{sch:eq:directional-primal-RH} with
$\sigma_W=2\varepsilon_W$.

For the second estimate, put
$$
 \widetilde w_{Q,u}(x)
 =|W^{-1/2}(x)\mathcal W_Qu|^2.
$$
The dual matrix weight $W^{-1}$ also belongs to $\mathcal A_2$.  Thus
this is a scalar projection of $W^{-1}$, again with a uniform scalar
$A_2$ constant.  Moreover,
\begin{align}
 \avg_Q\widetilde w_{Q,u}
 &=|\mathcal W'_Q\mathcal W_Qu|^2
 \leq\|\mathcal W'_Q\mathcal W_Q\|_{\op}^2
 =\|\mathcal W_Q\mathcal W'_Q\|_{\op}^2
 \leq[W]_{\mathcal A_2}.
 \label{sch:eq:dual-directional-average}
\end{align}
The equality of the two operator norms holds because the two products
are adjoints.
The scalar $A_2$ inequality, the inclusion $Q\subset S$ and 
\eqref{sch:eq:dual-directional-average} give
\begin{align*}
 \avg_S\widetilde w_{Q,u}
 &\leq A[\widetilde w_{Q,u}]_{A_2}
       \bigg(\avg_Q\widetilde w_{Q,u}^{-1}\bigg)^{-1}
 \leq A[\widetilde w_{Q,u}]_{A_2}
       \avg_Q\widetilde w_{Q,u}
 \leq C_{W,A},
\end{align*}
and hence
$$
 \avg_{Q^*}\widetilde w_{Q,u}\leq C_{W,A}.
$$
The reverse H\"older inequality for the scalar projections of
$W^{-1}$ proves \eqref{sch:eq:directional-dual-RH}, after decreasing
$\sigma_W$ if necessary.

When $Q^*=Q$, the proof is the same with $S=Q$.  The scalar
projection constants and the reverse H\"older exponents depend only on
the stated quantities.
\end{proof}

Under the hypotheses of
Proposition~\ref{sch:prop:concrete-matrix-NWO}, multiplication by a
bounded scalar function preserves the required column estimates.
Together with Lemma~\ref{sch:lem:matrix-column-NWO}, this gives the
following consequence.  If
$|\varphi_Q|\leq C|Q|^{-1/2}$, then the matrix functions
$$
 \one_{Q^*}\varphi_QW^{1/2}\mathcal W_Q^{-1},
 \qquad
 \one_{Q^*}\varphi_QW^{-1/2}\mathcal W_Q
$$
satisfy the matrix maximal estimate, with $Q^*$ as the controlling
cube.  This is the form used below.

Put $\Sigma=\{0,1\}^d$ and $\Sigma_0=\Sigma\setminus\{\mathbf1\}$.

\subsection{The direct Riesz upper estimate}

\subsubsection{Whitney blocks and Fourier coefficients}

Let
$
 \mathcal O=(\R^d\times\R^d)\setminus
 \{(x,x):x\in\R^d\}.
$
Choose a dyadic Whitney decomposition of $\mathcal O$ by product
cubes $Q\times R$.  The cubes $Q$ and $R$ have the same side
length and
\begin{equation}
 c_d\ell(Q)\leq\operatorname{dist}(Q,R)
 \leq C_d\ell(Q).
 \label{sch:eq:Whitney-separation}
\end{equation}
For each $Q$, there are at most $M_d$ possible cubes $R$.
Denote them by $R_{Q,s}$, $1\leq s\leq M_d$.  The maps
$Q\mapsto R_{Q,s}$ have uniformly bounded fibres.
Indeed, for fixed $Q$, \eqref{sch:eq:Whitney-separation} confines $R$
to an annulus containing only $O_d(1)$ dyadic cubes of side length
$\ell(Q)$.  The same argument with $Q$ and $R$ interchanged proves
the bounded-fibre assertion.

\begin{lemma}
\label{sch:lem:Whitney-Fourier}
For every integer $N>0$, there are scalar functions
$\Phi_{\ell,Q}$ and $\Psi_{\ell,R_{Q,s}}$, supported on $Q$
and $R_{Q,s}$, respectively and  of modulus at most one, such that
\begin{equation}
 K_j(x,y)\one_Q(x)\one_{R_{Q,s}}(y)
 =\frac1{|Q|}\sum_{\ell\in\mathbb Z^{2d}}
 \Upsilon^j_{\ell,Q,s}
 \Phi_{\ell,Q}(x)\Psi_{\ell,R_{Q,s}}(y),
 \label{sch:eq:Whitney-Fourier}
\end{equation}
where
\begin{equation}
 |\Upsilon^j_{\ell,Q,s}|
 \leq C_{N,d}(1+|\ell|)^{-N}.
 \label{sch:eq:Fourier-rapid-decay}
\end{equation}
\end{lemma}

\begin{proof}
Choose $\epsilon>0$ so that
$(1+\epsilon)Q$ and $(1+\epsilon)R_{Q,s}$ remain separated.
Take smooth cutoffs supported on these enlarged cubes and equal to
one on $Q$ and $R_{Q,s}$.  After rescaling both cubes to fixed
unit cubes, expand the localized kernel in a Fourier series in
$2d$ variables.  On the enlarged product,
$|x-y|\simeq\ell(Q)$ and  for all multi-indices
$\alpha,\beta$,
$$
 |\partial_x^\alpha\partial_y^\beta K_j(x,y)|
 \leq C_{\alpha,\beta,d}
 \ell(Q)^{-d-|\alpha|-|\beta|}.
$$
Each derivative falling on a cutoff contributes a factor
$\ell(Q)^{-1}$.  Hence the localized kernel satisfies the same
derivative bound.  Integration by parts $N$ times gives a
coefficient bounded by
$$
 C_{N,d}|Q|^{-1}(1+|\ell|)^{-N}.
$$
Put $|Q|^{-1}$ outside the series and restrict the exponentials to
$Q$ and $R_{Q,s}$.  This proves
\eqref{sch:eq:Whitney-Fourier}--\eqref{sch:eq:Fourier-rapid-decay}.
\end{proof}

\begin{lemma}
\label{sch:lem:Marcinkiewicz-ideal}
For $d>1$, define
\begin{equation}
 |||A|||_{d,\infty}
 =\sup_{N\geq1}N^{1/d-1}\sum_{n=1}^Ns_n(A).
 \label{sch:eq:Marcinkiewicz-norm}
\end{equation}
Then
$
 |||A|||_{d,\infty}\simeq\|A\|_{\Sch^{d,\infty}}.
$
Moreover, $|||\cdot|||_{d,\infty}$ is a convex ideal norm and is
lower semicontinuous for weak operator convergence.
\end{lemma}

\begin{proof}
For every nonnegative decreasing sequence $\{a_n\}$,
$$
 \sup_NN^{1/d}a_N
 \leq\sup_NN^{1/d-1}\sum_{n=1}^Na_n
 \leq \frac d{d-1}\sup_NN^{1/d}a_N.
$$
This proves the norm equivalence.  Ky Fan's inequality proves the
triangle and ideal properties.  Finally,
$$
 \sum_{n=1}^Ns_n(A)
 =\sup_{\{e_n\},\{f_n\}}
 \left|\sum_{n=1}^N\langle Ae_n,f_n\rangle\right|,
$$
where both families are orthonormal.  This variational formula proves
weak-operator lower semicontinuity.
\end{proof}

\begin{lemma}
\label{sch:lem:compact-local}
Let $A$ be bounded on $L^2(\R^d;\C^m)$.  Suppose that
$
 \langle Af,g\rangle=0
$
whenever $f,g$ are bounded and compactly supported and
$$
 \dist(\operatorname{supp}f,\operatorname{supp}g)>0.
$$
Then $A=M_a$ for some essentially bounded matrix function $a$.  If
$A$ is compact, then $A=0$.
\end{lemma}

\begin{proof}
Fix a dyadic grid.  Let $Q$ and $R$ be distinct cubes in the same
generation.  For $\epsilon>0$, remove an $\epsilon$-neighborhood of
the boundary from each cube and denote the resulting inner sets by
$Q_\epsilon$ and $R_\epsilon$.  The two inner sets have positive
distance.  The hypothesis and density of bounded functions give
$$
 \one_{R_\epsilon}A\one_{Q_\epsilon}=0.
$$
The boundaries have measure zero.  Letting $\epsilon\downarrow0$ and
using strong convergence of the multiplication projections gives
$$
 \one_RA\one_Q=0.
$$
Sum strongly over all cubes $R$ in the same generation with $R\ne Q$.
The same argument with $Q$ and $R$ interchanged gives
$$
 \one_{Q^c}A\one_Q=0,
 \qquad
 \one_QA\one_{Q^c}=0.
$$
Thus $A$ commutes with the indicator of every dyadic cube.  These cubes
generate the Borel sigma-algebra.  A monotone-class argument, followed
by approximation by simple functions, shows that $A$ commutes with
every scalar $L^\infty$ multiplier.  The commutant of these scalar multipliers on
$L^2(\R^d;\C^m)$ consists of multiplication by essentially bounded
$m\times m$ matrix functions.  Thus $A=M_a$.

Suppose that $M_a\neq0$.  A finite net of the unit sphere in
$\C^m$ gives a unit vector $v$, a number $\epsilon>0$ and  a
measurable set $E$ of positive finite measure such that
$|a(x)v|\geq\epsilon$ on $E$.  Since Lebesgue measure is
non-atomic, choose pairwise disjoint subsets $E_n\subset E$ of
positive measure.  The functions
$$
 f_n=|E_n|^{-1/2}\one_{E_n}v
$$
are orthonormal, while the functions $M_af_n$ have disjoint
supports and norms at least $\epsilon$.  They have no convergent
subsequence.  Therefore $M_a$ is not compact.  This proves the last
assertion.
\end{proof}

\subsubsection{Grouped Haar descendants}

The localized factors obtained from the Fourier expansion need not
satisfy the NWO maximal estimate.  Uniform $L^2$ bounds do not
imply that estimate.  We therefore sum all Haar functions in one
martingale generation and prove an $L^{2+\sigma}$ bound for the
resulting family.

We use the following elementary geometric refinement.  For every
Whitney pair $(Q,R_{Q,s})$, one can choose a cube $S_Q$ from one of
the fixed adjacent systems such that
\begin{equation}
 Q\cup R_{Q,s}\subset S_Q,
 \qquad \ell(S_Q)\leq C_d\ell(Q).
 \label{sch:eq:adjacent-common-cube}
\end{equation}
There are only finitely many choices for the adjacent system.  After
a further finite coloring, the map $Q\mapsto S_Q$ is one-to-one.
Indeed, for a fixed cube $S$, every $Q$ with $S_Q=S$ is contained in
$S$ and satisfies $\ell(Q)\simeq_d\ell(S)$.  Hence there are only
$O_d(1)$ such cubes and  a fixed finite coloring makes the map
one-to-one.  We
work with one value of $s$, one adjacent system and  one color.  All
constants below are uniform in these finite choices.  We write $S_Q$ for the common cube on this fixed piece.  The
functions
$E_{Q,k}$ also depend on $\ell,s$, the adjacent system and  the
color; these finite or summable indices are suppressed.

Fix $k\geq0$.  For $P\in\operatorname{ch}^k(S_Q)$ and
$\varepsilon\in\Sigma_0$, set
\begin{align}
 H_{P,\varepsilon}
 &=\mathcal V_P
   \frac{\widehat B(P,\varepsilon)}{|P|^{1/2}}
   \mathcal U_P^{-1},                                      
 \label{sch:eq:direct-H-coefficient}\\
 F_{P,Q,\ell,\varepsilon}(x)
 &=V^{1/2}(x)\Phi_{\ell,Q}(x)\one_Q(x)
   h_P^\varepsilon(x)\mathcal V_P^{-1},
 \label{sch:eq:direct-F-family}\\
 G_{Q,\ell,s}(y)
 &=|Q|^{-1/2}U^{-1/2}(y)
   \overline{\Psi_{\ell,R_{Q,s}}(y)}\one_{R_{Q,s}}(y)
   \mathcal U_{S_Q},
 \label{sch:eq:direct-G-family}\\
 A_{P,Q,\varepsilon}
 &=\bigg(\frac{|P|}{|Q|}\bigg)^{1/2}
   H_{P,\varepsilon}\mathcal U_P\mathcal U_{S_Q}^{-1}.
 \label{sch:eq:direct-A-coefficient}
\end{align}
Here and below a function in \eqref{sch:eq:direct-F-family} is understood
to be zero when $P\cap Q=\varnothing$.  Direct multiplication gives
the identity
\begin{align}
 &F_{P,Q,\ell,\varepsilon}(x)A_{P,Q,\varepsilon}
   G_{Q,\ell,s}(y)^*                                      \notag\\
 &\qquad=\frac1{|Q|}V^{1/2}(x)\Phi_{\ell,Q}(x)\one_Q(x)
   \widehat B(P,\varepsilon)h_P^\varepsilon(x)
   \Psi_{\ell,R_{Q,s}}(y)\one_{R_{Q,s}}(y)U^{-1/2}(y).
 \label{sch:eq:direct-exact-factorization}
\end{align}

Define
\begin{equation}
 c_{Q,k}^2
 =\sum_{P\in\operatorname{ch}^k(S_Q)}
   \sum_{\varepsilon\in\Sigma_0}
   \|A_{P,Q,\varepsilon}\|_{\HS}^2.
 \label{sch:eq:grouped-coefficient}
\end{equation}
The next lemma records the coefficient decay and the two column estimates.

\begin{lemma}
\label{sch:lem:grouped-descendant}
Let $2\leq r<\infty$ and $1\leq q\leq\infty$.  There is
$\gamma_U(r)>0$ such that
\begin{equation}
 \|\{c_{Q,k}\}_Q\|_{\ell^{r,q}}
 \leq C2^{-k\gamma_U(r)}
 \|\{\mathfrak h_P^+(B)\}_P\|_{\ell^{r,q}}.
 \label{sch:eq:grouped-coefficient-decay}
\end{equation}
Moreover, for every $2<s_0<2+\sigma_V$, where $\sigma_V$ is the
exponent in Proposition~\ref{sch:prop:concrete-matrix-NWO}, the matrix
functions
\begin{equation}
 E_{Q,k}(x)
 =\frac1{c_{Q,k}}
 \sum_{P\in\operatorname{ch}^k(S_Q)}
 \sum_{\varepsilon\in\Sigma_0}
 F_{P,Q,\ell,\varepsilon}(x)A_{P,Q,\varepsilon}
 \label{sch:eq:grouped-E-family}
\end{equation}
satisfy, with the convention $E_{Q,k}=0$ when $c_{Q,k}=0$,
\begin{equation}
 \|E_{Q,k}u\|_{L^{s_0}}
 \leq C2^{k\delta_V(s_0)}
 |S_Q|^{1/s_0-1/2}|u|,
 \qquad
 \delta_V(s_0)=d\bigg(\frac12-\frac1{s_0}\bigg).
 \label{sch:eq:grouped-E-growth}
\end{equation}
The family $\{G_{Q,\ell,s}\}_Q$ satisfies
\begin{equation}
 \|G_{Q,\ell,s}v\|_{L^{s_1}}
 \leq C|S_Q|^{1/s_1-1/2}|v|
 \label{sch:eq:direct-G-NWO-size}
\end{equation}
for some $s_1>2$.  All constants are independent of
$k,Q,\ell,s$ and of the finite geometric choices.
\end{lemma}

\begin{proof}
The reducing-matrix comparison \eqref{sch:eq:reducing-matrix-comparison},
with $S_Q$ in place of $Q$, gives
$$
 \|\mathcal U_P\mathcal U_{S_Q}^{-1}\|_{\op}^2
 \leq \frac{|S_Q|}{|P|}\vartheta_U(P,S_Q).
$$
Since $|S_Q|\simeq|Q|$, \eqref{sch:eq:direct-A-coefficient} implies
\begin{equation}
 c_{Q,k}^2
 \leq C\sum_{P\in\operatorname{ch}^k(S_Q)}
 \vartheta_U(P,S_Q)\mathfrak h_P^+(B)^2.
 \label{sch:eq:cQ-by-Hk}
\end{equation}
The map $Q\mapsto S_Q$ is one-to-one on the present color.
Therefore \eqref{sch:eq:Hk-Lorentz}, applied to
$\{\mathfrak h_P^+(B)\}$, proves
\eqref{sch:eq:grouped-coefficient-decay}.

We next prove \eqref{sch:eq:grouped-E-growth}.  Apply
\eqref{sch:eq:directional-primal-RH} with $W=V$ and with both cubes
equal to $P$.  Since $|\Phi_{\ell,Q}|\leq1$,
$|h_P^\varepsilon|=|P|^{-1/2}\one_P$ and  restriction to
$P\cap Q$ decreases the norm,
\begin{equation}
 \|F_{P,Q,\ell,\varepsilon}u\|_{L^{s_0}}
 \leq C|P|^{1/s_0-1/2}|u|.
 \label{sch:eq:direct-F-Ls}
\end{equation}
The cubes $P\in\operatorname{ch}^k(S_Q)$ are disjoint.  There are
only $2^d-1$ Haar signatures.  If
$$
 a_P^2=\sum_{\varepsilon\in\Sigma_0}
       \|A_{P,Q,\varepsilon}\|_{\HS}^2,
$$
then \eqref{sch:eq:direct-F-Ls}, disjointness and 
$\|a\|_{\ell^{s_0}}\leq\|a\|_{\ell^2}$ give
\begin{align*}
 \left\|\sum_{P,\varepsilon}
 F_{P,Q,\ell,\varepsilon}A_{P,Q,\varepsilon}u
 \right\|_{L^{s_0}}^{s_0}
 &\leq C\sum_{P\in\operatorname{ch}^k(S_Q)}
 |P|^{1-s_0/2}a_P^{s_0}|u|^{s_0}
 \leq C(2^{-kd}|S_Q|)^{1-s_0/2}
 c_{Q,k}^{s_0}|u|^{s_0}.
\end{align*}
Divide by $c_{Q,k}^{s_0}$ and take the $s_0$th root.  This is
\eqref{sch:eq:grouped-E-growth}.

For \eqref{sch:eq:direct-G-NWO-size}, fix a unit vector $v$ and put
$$
 w_{Q,v}(y)=|U^{-1/2}(y)\mathcal U_{S_Q}v|^2.
$$
This is a scalar projection of $U^{-1}$, with a uniform scalar
$A_2$ constant.  Since $R_{Q,s}\subset S_Q$ and
$|S_Q|\simeq|R_{Q,s}|$,
\begin{align*}
 \avg_{R_{Q,s}}w_{Q,v}
 &\leq C\avg_{S_Q}w_{Q,v}
 =C|\mathcal U'_{S_Q}\mathcal U_{S_Q}v|^2
 \leq C[U]_{\mathcal A_2}.
\end{align*}
Scalar reverse H\"older on $R_{Q,s}$, followed by
$|R_{Q,s}|\simeq|S_Q|$, proves
\eqref{sch:eq:direct-G-NWO-size}.
\end{proof}

For a fixed $r\geq2$, choose $s_0>2$ so close to $2$ that
\begin{equation}
 \delta_V(s_0)<\gamma_U(r).
 \label{sch:eq:growth-below-decay}
\end{equation}
This is possible because a reverse-H\"older estimate remains true
when its exponent is decreased and 
$\delta_V(s_0)\downarrow0$ as $s_0\downarrow2$.
For each fixed descendant generation $k$,
\eqref{sch:eq:grouped-E-growth} and
\eqref{sch:eq:direct-G-NWO-size} hold uniformly in the unit vector.
Lemma~\ref{sch:lem:matrix-column-NWO} therefore shows that the two
matrix families
$$
 \{2^{-k\delta_V(s_0)}E_{Q,k}\}_Q,
 \qquad \{G_{Q,\ell,s}\}_Q
$$
satisfy \eqref{sch:eq:matrix-NWO-F-bound} and
\eqref{sch:eq:matrix-NWO-G-bound}, with controlling cubes $S_Q$.
Apply Proposition~\ref{sch:prop:matrix-NWO} with
$\mathcal K$ consisting of one point, with the first of these
families as $F_Q$, the second as $G_Q$ and  with
$$
 A_Q=2^{k\delta_V(s_0)}c_{Q,k}I_m.
$$
Together with \eqref{sch:eq:grouped-coefficient-decay}, this gives the
following estimate.
Here and in the next display, $Q$ ranges over an arbitrary finite
subcollection and  the constant is independent of that subcollection:
\begin{equation}
 \left\|\sum_QE_{Q,k}c_{Q,k}
       \int G_{Q,\ell,s}(y)^*\,\cdot\,dy\right\|_{\Sch^{r,q}}
 \leq C2^{-k(\gamma_U(r)-\delta_V(s_0))}
 \|\{\mathfrak h_P^+(B)\}_P\|_{\ell^{r,q}}.
 \label{sch:eq:one-generation-Riesz-bound}
\end{equation}
Since $\Upsilon^j_{\ell,Q,s}$ depends on $Q$, we record explicitly
how it is used.  Absorb it into the scalar coefficient in
Proposition~\ref{sch:prop:matrix-NWO}.  Equations
\eqref{sch:eq:Fourier-rapid-decay} and
\eqref{sch:eq:one-generation-Riesz-bound} give
\begin{align}
 &\left\|\sum_Q\Upsilon^j_{\ell,Q,s}E_{Q,k}c_{Q,k}
       \int G_{Q,\ell,s}(y)^*\,\cdot\,dy\right\|_{\Sch^{r,q}}
 \notag\\
 &\qquad\leq
 C_N(1+|\ell|)^{-N}
 2^{-k(\gamma_U(r)-\delta_V(s_0))}
 \|\{\mathfrak h_P^+(B)\}_P\|_{\ell^{r,q}}.
 \label{sch:eq:one-generation-with-Fourier}
\end{align}
The factor $m^{1/2}$ caused by $I_m$ is included in the constant.

\begin{lemma}
\label{sch:lem:weighted-martingale-convergence}
Let $W\in\mathcal A_2$, let $S$ be a cube in a dyadic system and 
let
$$
 \mathbb E_n^SF
 =\sum_{P\in\operatorname{ch}^n(S)}
   \langle F\rangle_P\one_P.
$$
For every matrix-valued $F\in L^2(W;S)$,
\begin{equation}
 \|\mathbb E_n^SF\|_{L^2(W;S)}
 \leq[W]_{\mathcal A_2}^{1/2}\|F\|_{L^2(W;S)}
 \label{sch:eq:weighted-conditional-bound}
\end{equation}
and
\begin{equation}
 \mathbb E_n^SF\longrightarrow F
 \quad\text{in }L^2(W;S).
 \label{sch:eq:weighted-martingale-convergence}
\end{equation}
\end{lemma}

\begin{proof}
Apply matrix Cauchy--Schwarz to each column of $F$.  For every
$P\subset S$,
\begin{align*}
 |P|\,\|\langle W\rangle_P^{1/2}\langle F\rangle_P\|_{\HS}^2
 &\leq
 \|\langle W\rangle_P^{1/2}
   \langle W^{-1}\rangle_P^{1/2}\|_{\op}^2
 \int_P\|W^{1/2}(x)F(x)\|_{\HS}^2\,dx\\
 &\leq[W]_{\mathcal A_2}
 \int_P\|W^{1/2}(x)F(x)\|_{\HS}^2\,dx.
\end{align*}
Sum this estimate over $P\in\operatorname{ch}^n(S)$.  This proves
\eqref{sch:eq:weighted-conditional-bound}.

Finite dyadic step functions are dense in $L^2(W;S)$.  For such a
function, $\mathbb E_n^SF=F$ for all sufficiently large $n$.
Approximate a general $F$ by a finite step function and use the
uniform bound \eqref{sch:eq:weighted-conditional-bound}.  This proves
\eqref{sch:eq:weighted-martingale-convergence}.
\end{proof}

\subsubsection{The all-dimensional upper estimate}

\begin{theorem}
\label{sch:thm:Riesz-upper}
Let $d\geq2$, let $U,V\in\mathcal A_2$ and  fix
$1\leq j\leq d$.  Then
\begin{align}
 \|C_{B,R_j}\|_{\Sch^r(L^2(U),L^2(V))}
 &\leq C\|B\|_{\mathfrak B^r_{V,U}},
 &&d<r<\infty,                                      
 \label{sch:eq:direct-strong-upper}\\
 \|C_{B,R_j}\|_{\Sch^{d,\infty}(L^2(U),L^2(V))}
 &\leq C\|B\|_{\mathfrak W^d_{V,U}}.
 \label{sch:eq:direct-critical-upper}
\end{align}
The proof is deterministic.  It uses neither a dyadic representation
of $R_j$ nor an average over random grids.
\end{theorem}

\begin{proof}
If the right side of either estimate is infinite, there is nothing to
prove.  Otherwise Lemma~\ref{sch:lem:sequence-implies-VMO} gives
$B\in\VMO^2_{V,U}$.  Theorem~\ref{thm:main} therefore gives the
standard compact commutator $C_{B,R_j}:L^2(U)\to L^2(V)$.  We recover
its conjugate as a weak-operator limit of finite-rank operators.

We first take finite truncations in the Whitney cubes, Fourier modes,
and martingale generations.  The resulting operators have finite
rank.  On a fixed Whitney block, split
\begin{equation}
 B(x)-B(y)=(B(x)-B_{S_Q})+(B_{S_Q}-B(y)).
 \label{sch:eq:common-mean-split}
\end{equation}
For $x\in Q\subset S_Q$, martingale expansion in the adjacent
system containing $S_Q$ gives
\begin{equation}
 B(x)-B_{S_Q}
 =\sum_{k\geq0}\sum_{P\in\operatorname{ch}^k(S_Q)}
   \sum_{\varepsilon\in\Sigma_0}
   \widehat B(P,\varepsilon)h_P^\varepsilon(x).
 \label{sch:eq:local-full-martingale-expansion}
\end{equation}
%
Thus \eqref{sch:eq:Whitney-Fourier} and
\eqref{sch:eq:direct-exact-factorization} express the first term in
\eqref{sch:eq:common-mean-split} as follows.  Applied to
$f\in L^2(\mathbb R^d;\mathbb C^m)$, this part of the operator is
\begin{equation}
 \sum_{\ell\in\mathbb Z^{2d}}\sum_{k\geq0}\sum_Q
 \Upsilon^j_{\ell,Q,s}E_{Q,k}(x)c_{Q,k}
 \int_{\mathbb R^d}G_{Q,\ell,s}(y)^*f(y)\,dy.
 \label{sch:eq:x-half-grouped-series}
\end{equation}
Here the Whitney-neighbour index $s$, the adjacent dyadic system, and
the finite color are fixed.  The full $x$-part is obtained by summing
over these finite choices.  Every Haar cube in
\eqref{sch:eq:local-full-martingale-expansion} is a descendant of
$S_Q$.  There is no infinite ancestor term.

For the strong estimate take $q=r$ in
\eqref{sch:eq:one-generation-with-Fourier}.  Sum first in $k$, using
\eqref{sch:eq:growth-below-decay} and  then in $\ell$, with $N>2d$.
The indices $s$, the
adjacent systems and  the colors range over fixed finite sets.  We
obtain
\begin{equation}
 \|T_x\|_{\Sch^r}
 \leq C\sum_t
 \|\{\mathfrak h_P^+(B):P\in\D^t\}\|_{\ell^r}.
 \label{sch:eq:x-half-strong}
\end{equation}
Here $T_x$ denotes the part of the localized commutator containing
$B(x)-B_{S_Q}$.

The second term in \eqref{sch:eq:common-mean-split} is treated by taking
adjoints.  More precisely,
$$
 (\mathcal C_{B,R_j}^{V,U})^*
 =M_{U^{-1/2}}[M_{B^*},R_j]M_{V^{1/2}}.
$$
After interchanging $Q$ and $R_{Q,s}$, its $x$-half is the
construction leading to \eqref{sch:eq:x-half-strong}, with symbol
$B^*$, source weight $V^{-1}$ and  target weight $U^{-1}$.  It is controlled by
$\{\mathfrak h_P^-(B)\}$.  Hence
\begin{equation}
 \|T_y\|_{\Sch^r}
 \leq C\sum_t
 \|\{\mathfrak h_P^-(B):P\in\D^t\}\|_{\ell^r}.
 \label{sch:eq:y-half-strong}
\end{equation}
Theorem~\ref{sch:prop:matrix-besov-JN} now proves
\eqref{sch:eq:direct-strong-upper} for every finite truncation.

At the critical index use
\eqref{sch:eq:one-generation-with-Fourier} with
$(r,q)=(d,\infty)$.  Sum first over $k$, using
\eqref{sch:eq:growth-below-decay} and  then over $\ell$, using
$N>2d$.  Lemma~\ref{sch:lem:Marcinkiewicz-ideal} and the finite sums
over $s$, the adjacent systems and  the colors give, uniformly for
all finite truncations,
\begin{equation}
 |||T_x+T_y|||_{d,\infty}
 \leq C\|B\|_{\mathfrak W^d_{V,U}}.
 \label{sch:eq:finite-critical-bound}
\end{equation}

It remains to remove the truncations.  On $S_Q$, put
$$
 F_{S_Q}=(B-B_{S_Q})\mathcal U_{S_Q}^{-1}.
$$
The definition of $\alpha_{S_Q}^{\HS}(B)$ gives
$$
 \|F_{S_Q}\|_{L^2(V;S_Q)}^2
 =|S_Q|\alpha_{S_Q}^{\HS}(B)^2.
$$
The partial sum through generations $k=0,\ldots,n$ in
\eqref{sch:eq:local-full-martingale-expansion}, after right multiplication
by $\mathcal U_{S_Q}^{-1}$, is
$\mathbb E_{n+1}^{S_Q}F_{S_Q}$, because
$\langle F_{S_Q}\rangle_{S_Q}=0$.  Lemma
\ref{sch:lem:weighted-martingale-convergence} therefore gives convergence
to $F_{S_Q}$ in $L^2(V;S_Q)$.  For the adjoint half, apply the
same lemma with weight $U^{-1}$ to
$$
 (B^*-B_{S_Q}^*)(\mathcal V'_{S_Q})^{-1};
$$
its squared norm is $|S_Q|\beta_{S_Q}^{\HS}(B)^2$.

Introduce the dense subspaces
$$
 \mathcal X_U
 =\{U^{1/2}\varphi:\varphi\in L^\infty_c(\R^d;\C^m)\},
 \qquad
 \mathcal Y_V
 =\{V^{-1/2}\psi:\psi\in L^\infty_c(\R^d;\C^m)\}
$$
of unweighted $L^2(\R^d;\C^m)$.  Their density follows from the
density of bounded compactly supported functions in $L^2(U)$ and
$L^2(V^{-1})$, respectively.

The Fourier series converges absolutely and uniformly on every
Whitney block.  Only finitely many Whitney blocks meet the product of
the two supports: their side lengths are bounded below by the positive
distance between the supports and bounded above by their diameters.
Cauchy--Schwarz and the two weighted martingale
convergences now show that, if
$\varphi,\psi\in L^\infty_c(\R^d;\C^m)$ have compact supports at
positive distance,
then the truncated bilinear forms converge to
\begin{align}
 &\left\langle
 \mathcal C_{B,R_j}^{V,U}U^{1/2}\varphi,
 V^{-1/2}\psi\right\rangle =
 \iint_{\R^d\times\R^d}
 \psi(x)^*(B(x)-B(y))K_j(x,y)\varphi(y)\,dy\,dx.
 \label{sch:eq:separated-core-form}
\end{align}
The last integral is absolutely convergent.  Let $I$ be the
directed set of triples
$\iota=(\mathcal Q_0,\mathcal L_0,n_0)$, where $\mathcal Q_0$ is a
finite set of Whitney blocks, $\mathcal L_0\subset\mathbb Z^{2d}$ is
finite and  $n_0\geq0$.  Order the first two coordinates by inclusion
and the last by the usual order.  Let $T_\iota$ be the corresponding
truncation in which the martingale sum contains the initial
generations $0,\ldots,n_0$.  The uniform bounds place
$\{T_\iota\}_{\iota\in I}$ in a weak-operator compact ball.  Hence
the net has a weak-operator cluster point and  every cluster point has
the off-diagonal form \eqref{sch:eq:separated-core-form}.

For $r>d$, use the extended-valued formula
$$
 \|T\|_{\Sch^r}
 =\sup_{\substack{A\text{ finite rank}\\
                   \|A\|_{\Sch^{r'}}\leq1}}
   |\operatorname{Tr}(A^*T)|.
$$
Each expression inside the supremum is weak-operator continuous.
Hence the Schatten norm is weak-operator lower semicontinuous.  Apply
this fact along a subnet converging weakly to any cluster point.  This gives
the bound in \eqref{sch:eq:direct-strong-upper} for every cluster
point; in particular, every cluster point belongs to $\Sch^r$ and is
compact.
At the critical index, let $T$ be a weak-operator cluster point.
Choose a subnet $\{T_{\iota_\eta}\}_{\eta\in J}$ that converges
weakly to $T$.  Equation \eqref{sch:eq:finite-critical-bound} gives
$s_n(T_{\iota_\eta})\leq Cn^{-1/d}$.  Therefore, for every $N\geq1$,
Ky Fan lower semicontinuity gives
\begin{align*}
 \sum_{n=1}^Ns_n(T)
 &\leq\liminf_\eta\sum_{n=1}^Ns_n(T_{\iota_\eta})
 \leq C\sum_{n=1}^Nn^{-1/d}
 \leq C_dN^{1-1/d}.
\end{align*}
It follows that
$$
 s_N(T)\leq \frac1N\sum_{n=1}^Ns_n(T)
 \leq C_dN^{-1/d}.
$$
Thus $s_N(T)\to0$.  By the approximation-number characterization,
$T$ is compact and  the same estimate gives
$T\in\Sch^{d,\infty}$.  Hence every cluster point satisfies the
bound in \eqref{sch:eq:direct-critical-upper}.  We now identify each
cluster point with the standard commutator obtained from
Theorem~\ref{thm:main}.  Let $D$ be the difference between a cluster
point and the conjugate of that commutator.  Both operators have the
off-diagonal form \eqref{sch:eq:separated-core-form}.  Hence $D$
vanishes on the positive-distance pairs from
$\mathcal X_U\times\mathcal Y_V$.  This density is local.  To see this,
if $h$ is bounded and supported in a bounded set $E$, truncate
$U^{-1/2}h$ in norm while keeping its support in $E$.  Multiplication
by $U^{1/2}$ then gives approximants to $h$ from $\mathcal X_U$, with
the same support.  Truncating $V^{1/2}h$ gives the corresponding
statement for $\mathcal Y_V$.  Thus $D$ vanishes on every pair of
bounded compactly supported functions whose supports have positive
distance.  Lemma~\ref{sch:lem:compact-local} shows that $D$ is a
compact multiplication operator, because both terms in its definition
are compact.  The lemma gives $D=0$.  Thus every cluster point is the
conjugate of the standard commutator and agrees with the algebraic
commutator on its initial form domain.
Every weak-operator cluster point is therefore equal to
$\mathcal C_{B,R_j}^{V,U}$.  Hence the net of finite truncations
converges weakly to this operator.  Finally,
\eqref{sch:eq:unitary-conjugation} transfers the two estimates to
$C_{B,R_j}:L^2(U)\to L^2(V)$.
\end{proof}

\subsection{The median lower estimate}

\subsubsection{Companion cubes and a complex median}

Fix a sufficiently large integer $N_d$.  For an axis-parallel cube
$Q$, put
\begin{equation}
 \widetilde Q=Q+N_d\ell(Q)e_j.
 \label{sch:eq:companion-cube}
\end{equation}
For $x\in Q$ and $y\in\widetilde Q$, the kernel has a fixed sign and
\begin{equation}
 c_d|Q|^{-1}\leq |K_j(x,y)|\leq C_d|Q|^{-1}.
 \label{sch:eq:Riesz-companion-kernel}
\end{equation}
Indeed, $|x_j-y_j|\simeq N_d\ell(Q)$ and
$|x-y|\simeq N_d\ell(Q)$ on this product.

\begin{lemma}
\label{sch:lem:complex-median}
Let $\beta$ be a complex-valued integrable function on
$Q\cup\widetilde Q$.  There are $c\in\C$, a measurable
partition $Q=E_1\cup\cdots\cup E_4$, sets
$F_s\subset\widetilde Q$ with $|F_s|\geq|Q|/3$ and  unimodular
numbers $\zeta_s$ such that
\begin{equation}
 \operatorname{Re}\bigl(
 \zeta_s(\beta(x)-\beta(y))\bigr)
 \geq2^{-1/2}|\beta(x)-c|
 \label{sch:eq:complex-sector-inequality}
\end{equation}
for $x\in E_s$ and $y\in F_s$.
\end{lemma}

\begin{proof}
Apply the centerpoint theorem in $\R^2\simeq\C$ to the push-forward of
normalized Lebesgue measure on $\widetilde Q$ under $\beta$.  It gives
$c\in\C$ such that every closed half-plane containing $c$ has
push-forward measure at least $1/3$.

Choose four directions $\theta_s$ whose closed sectors of half-angle
$\pi/4$ cover the plane.  Put
\begin{align*}
 E_s&=\{x\in Q:\arg(\beta(x)-c)
              \text{ belongs to the $s$th sector}\},\\
 F_s&=\{y\in\widetilde Q:
 \operatorname{Re}(e^{-i\theta_s}(\beta(y)-c))\leq0\}.
\end{align*}
Assign points with $\beta(x)=c$ and sector boundaries in any
fixed way.  The centerpoint property gives $|F_s|\geq|Q|/3$.  For
$x\in E_s$,
$$
 \operatorname{Re}(e^{-i\theta_s}(\beta(x)-c))
 \geq2^{-1/2}|\beta(x)-c|.
$$
The defining inequality for $F_s$ now gives
\eqref{sch:eq:complex-sector-inequality} with
$\zeta_s=e^{-i\theta_s}$.
\end{proof}

\subsubsection{The companion NWO families}

\begin{lemma}
\label{sch:lem:companion-NWO}
Let $E_Q\subset Q$, $F_Q\subset\widetilde Q$ and  let $u_Q,v_Q$ be
unit vectors.  Define
\begin{align}
 e_Q(y)
 &=|Q|^{-1/2}U^{1/2}(y)\one_{F_Q}(y)
   \mathcal U_Q^{-1}v_Q,
 \label{sch:eq:lower-test-e}\\
 f_Q(x)
 &=|Q|^{-1/2}V^{-1/2}(x)\one_{E_Q}(x)
   \mathcal V_Qu_Q.
 \label{sch:eq:lower-test-f}
\end{align}
There is $\sigma>0$, independent of all choices, such that
\begin{equation}
 \|e_Q\|_{2+\sigma}+\|f_Q\|_{2+\sigma}
 \leq C|Q|^{1/(2+\sigma)-1/2}.
 \label{sch:eq:companion-Lp-size}
\end{equation}
For each $Q$, one may choose a cube $S_Q$ containing
$Q\cup\widetilde Q$ with $|S_Q|\leq C_d|Q|$.  With this same
controlling cube for both functions, the two families are NWO.  The
conclusion remains valid for a fixed finite number of choices for every
$Q$.
\end{lemma}

\begin{proof}
Put
$$
 w_Q(y)=|U^{1/2}(y)\mathcal U_Q^{-1}v_Q|^2.
$$
Its average on $Q$ is one.  Lemma~\ref{sch:lem:scalar-projections} gives a
uniform scalar $A_2$ constant.  Let $S_Q$ be an axis-parallel cube
containing $Q\cup\widetilde Q$ with $|S_Q|\leq C_d|Q|$.  Then
\begin{align*}
 \avg_{S_Q}w_Q
 &\leq[w_Q]_{A_2}\bigg(\avg_{S_Q}w_Q^{-1}\bigg)^{-1}
 \leq[w_Q]_{A_2}\frac{|S_Q|}{|Q|}
       \bigg(\avg_Qw_Q^{-1}\bigg)^{-1}\leq C.
\end{align*}
Here the last inequality follows from
$(\avg_Qw_Q)(\avg_Qw_Q^{-1})\geq1$.  Hence
$\avg_{\widetilde Q}w_Q\leq C$.  Scalar reverse H\"older gives
$\sigma_U>0$ such that
$$
 \int_{\widetilde Q}w_Q^{1+\sigma_U/2}\leq C|Q|.
$$
This proves the required estimate for $e_Q$.

For $f_Q$, use
$$
 z_Q(x)=|V^{-1/2}(x)\mathcal V_Qu_Q|^2.
$$
This is a scalar projection of $V^{-1}$ with a uniform $A_2$
constant and 
$$
 \avg_Qz_Q
 =|\mathcal V'_Q\mathcal V_Qu_Q|^2
 \leq[V]_{\mathcal A_2}.
$$
Reverse H\"older gives the required estimate with an exponent
$\sigma_V>0$.  Take $\sigma=\min\{\sigma_U,\sigma_V\}$.  Indicators can
only decrease the norms.  Since $|S_Q|\simeq_d|Q|$ and both supports
are contained in $S_Q$, Lemma~\ref{sch:lem:NWO-Lp-criterion} applies to
both families with the same controlling cube $S_Q$.  If there are at
most $L$ choices for each
$Q$, add the choice index to $Q$.  The associated maximal function
is bounded by the sum of the $L$ NWO maximal functions.  Its
$L^2$ norm is at most $L$ times the uniform NWO bound.
\end{proof}

For the adjoint estimate, replace $(U,V)$ by
$(V^{-1},U^{-1})$.  The corresponding test functions are
\begin{align*}
 e_Q^{\#}(y)
 &=|Q|^{-1/2}V^{-1/2}(y)\one_{F_Q}(y)
   (\mathcal V'_Q)^{-1}v_Q,\\
 f_Q^{\#}(x)
 &=|Q|^{-1/2}U^{1/2}(x)\one_{E_Q}(x)
   \mathcal U'_Qu_Q.
\end{align*}
Proposition~\ref{sch:prop:concrete-matrix-NWO}, applied to
$V^{-1}$ and $U^{-1}$, proves their
$L^{2+\sigma}$ estimates.  Lemma~\ref{sch:lem:NWO-Lp-criterion}
then proves the two NWO maximal estimates with the same cube $S_Q$
containing $Q\cup\widetilde Q$ and satisfying $|S_Q|\leq C_d|Q|$.

\subsubsection{Extraction of the matrix Haar coefficients}

\begin{theorem}
\label{sch:thm:Riesz-lower}
Let $d\geq2$, let $U,V\in\mathcal A_2$ be matrix weights, let $B$ be
locally integrable and  fix $1\leq j\leq d$.  If $1<r<\infty$ and
$0<q\leq\infty$, then
\begin{equation}
 \|B\|_{\mathfrak B^{r,q}_{V,U}}
 \leq C
 \|C_{B,R_j}\|_{\Sch^{r,q}(L^2(U),L^2(V))}.
 \label{sch:eq:Riesz-lower}
\end{equation}
In particular, this includes $(r,q)=(d,\infty)$.
\end{theorem}

\begin{proof}
Work first on one fixed adjacent dyadic system $\D$.  For
$Q\in\D$ and $\varepsilon\in\Sigma_0$, put
\begin{equation}
 A_{Q,\varepsilon}
 =\mathcal V_Q
 \frac{\widehat B(Q,\varepsilon)}{|Q|^{1/2}}
 \mathcal U_Q^{-1}.
 \label{sch:eq:normalized-coefficient-lower}
\end{equation}
Choose singular vectors $u_{Q,\varepsilon,\rho}$ and
$v_{Q,\varepsilon,\rho}$ so that
$$
 A_{Q,\varepsilon}v_{Q,\varepsilon,\rho}
 =s_{Q,\varepsilon,\rho}u_{Q,\varepsilon,\rho},
 \qquad 1\leq \rho\leq m.
$$
Define the scalar function
\begin{equation}
 \beta_{Q,\varepsilon,\rho}(x)
 =u_{Q,\varepsilon,\rho}^*\mathcal V_QB(x)
  \mathcal U_Q^{-1}v_{Q,\varepsilon,\rho}.
 \label{sch:eq:scalarized-symbol}
\end{equation}
Apply Lemma~\ref{sch:lem:complex-median} to this function.  Denote the resulting
quantities by $\mathfrak m_{Q,\varepsilon,\rho}$,
$E_{Q,\varepsilon,\rho,s}$, $F_{Q,\varepsilon,\rho,s}$ and  $\zeta_s$.
Use these sets and singular vectors in
\eqref{sch:eq:lower-test-e}--\eqref{sch:eq:lower-test-f}; call the resulting
functions $e_{Q,\varepsilon,\rho,s}$ and
$f_{Q,\varepsilon,\rho,s}$.

The separated-support kernel formula
\eqref{sch:eq:weighted-commutator-kernel} gives
\begin{align}
 &\langle\mathcal C_{B,R_j}^{V,U}e_{Q,\varepsilon,\rho,s},
 f_{Q,\varepsilon,\rho,s}\rangle=\frac1{|Q|}
 \int_{E_{Q,\varepsilon,\rho,s}}
 \int_{F_{Q,\varepsilon,\rho,s}}
 \bigl(\beta_{Q,\varepsilon,\rho}(x)
       -\beta_{Q,\varepsilon,\rho}(y)\bigr)
 K_j(x,y)\,dy\,dx.
 \label{sch:eq:lower-pairing-identity}
\end{align}
The identity is first obtained for bounded truncations of $B$.
Since the two supports are separated, $K_j$ is bounded on their
product and  local integrability of $B$ permits passage to the
limit.

Multiply the integral in \eqref{sch:eq:lower-pairing-identity} by the sign
of $K_j$ and by $\zeta_s$ and  take its real part.  Lemma
\ref{sch:lem:complex-median} and
\eqref{sch:eq:Riesz-companion-kernel} give
\begin{equation}
 \avg_Q|\beta_{Q,\varepsilon,\rho}-
       \mathfrak m_{Q,\varepsilon,\rho}|
 \leq C\sum_{s=1}^4
 \left|\langle\mathcal C_{B,R_j}^{V,U}
 e_{Q,\varepsilon,\rho,s},f_{Q,\varepsilon,\rho,s}\rangle\right|.
 \label{sch:eq:median-to-pairings}
\end{equation}
Indeed, $|F_{Q,\varepsilon,\rho,s}|\geq|Q|/3$ cancels one factor
$|Q|^{-1}$ from the two integrations.

The cancellation of $h_Q^\varepsilon$ gives
\begin{align}
 s_{Q,\varepsilon,\rho}
 &=|Q|^{-1/2}\left|
 \int_Qh_Q^\varepsilon(x)
 \beta_{Q,\varepsilon,\rho}(x)\,dx\right|\leq\avg_Q|\beta_{Q,\varepsilon,\rho}(x)
             -\mathfrak m_{Q,\varepsilon,\rho}|\,dx.
 \label{sch:eq:singular-value-to-median}
\end{align}
By Lemma~\ref{sch:lem:companion-NWO}, all test functions in
\eqref{sch:eq:median-to-pairings} form a fixed finite collection of NWO
families.  Lemma~\ref{sch:lem:NWO-diagonal} and the distribution-function
argument leading to \eqref{sch:eq:finite-block-Lorentz-equivalence}
give
\begin{align}
 \|\{\mathfrak h_Q^+(B):Q\in\D\}\|_{\ell^{r,q}}
 &\simeq
 \|\{s_{Q,\varepsilon,\rho}\}_{Q,\varepsilon,\rho}\|_{\ell^{r,q}}
 \leq C\|\mathcal C_{B,R_j}^{V,U}\|_{\Sch^{r,q}}.
 \label{sch:eq:primal-lower-final}
\end{align}

Since $R_j^*=-R_j$,
\begin{equation}
 (\mathcal C_{B,R_j}^{V,U})^*
 =M_{U^{-1/2}}[M_{B^*},R_j]M_{V^{1/2}}.
 \label{sch:eq:adjoint-conjugated-commutator}
\end{equation}
Apply \eqref{sch:eq:primal-lower-final} to the symbol $B^*$, source weight
$V^{-1}$ and  target weight $U^{-1}$.  This gives
\begin{equation}
 \|\{\mathfrak h_Q^-(B):Q\in\D\}\|_{\ell^{r,q}}
 \leq C\|\mathcal C_{B,R_j}^{V,U}\|_{\Sch^{r,q}}.
 \label{sch:eq:dual-lower-final}
\end{equation}
Sum \eqref{sch:eq:primal-lower-final} and
\eqref{sch:eq:dual-lower-final} over the fixed adjacent systems.  Theorem
\ref{sch:prop:matrix-besov-JN} and
\eqref{sch:eq:unitary-conjugation} prove \eqref{sch:eq:Riesz-lower}.
\end{proof}

\subsection{Endpoint estimates}

The next argument does not use a dyadic representation.

\begin{theorem}
\label{sch:thm:matrix-rigidity}
Let $d\geq2$, let $U,V\in\mathcal A_2$ be matrix weights, let $B$ be
locally integrable and  fix $1\leq j\leq d$.  If $0<r\leq d$ and
$$
 C_{B,R_j}\in\Sch^r(L^2(U;\C^m),L^2(V;\C^m)),
$$
then there is a fixed matrix $B_0\in\C^{m\times m}$ such that
$B(x)=B_0$ for almost every $x$.  Conversely, every such symbol
satisfies $C_{B,R_j}=0$.
\end{theorem}

\begin{proof}
Fix coordinate vectors $e_a,e_b$ and put
\begin{equation}
 \mu_b(x)=e_b^*U(x)e_b,
 \qquad
 \lambda_a(x)=\bigl(e_a^*V^{-1}(x)e_a\bigr)^{-1}.
 \label{sch:eq:scalar-compression-weights}
\end{equation}
Lemma~\ref{sch:lem:scalar-projections} gives $\mu_b\in A_2$ and
$e_a^*V^{-1}e_a\in A_2$.  The reciprocal of a scalar $A_2$ weight is
again in $A_2$, so $\lambda_a\in A_2$.

Define
$$
 I_b:L^2(\mu_b)\longrightarrow L^2(U;\C^m),
 \quad I_bf=fe_b,
$$
and
$$
 P_a:L^2(V;\C^m)\longrightarrow L^2(\lambda_a),
 \quad P_aF=e_a^*F.
$$
The first map is an isometry.  The matrix Cauchy--Schwarz inequality
gives
\begin{equation}
 \lambda_a(x)|e_a^*F(x)|^2
 \leq F(x)^*V(x)F(x),
 \label{sch:eq:coordinate-projection-contraction}
\end{equation}
so $P_a$ is a contraction.  Since $R_j$ acts componentwise,
\begin{equation}
 P_aC_{B,R_j}I_b=[M_{B_{ab}},R_j]
 :L^2(\mu_b)\longrightarrow L^2(\lambda_a).
 \label{sch:eq:coordinate-compression}
\end{equation}

If $r<d$, use $\Sch^r\subset\Sch^d$; if $r=d$, no inclusion is
needed.  The ideal property and \eqref{sch:eq:coordinate-compression} show
that every scalar commutator $[M_{B_{ab}},R_j]$ belongs to
$\Sch^d(L^2(\mu_b),L^2(\lambda_a))$.  Put
$$
 \nu_{ab}=(\mu_b/\lambda_a)^{1/2}.
$$
The compressed operator is bounded, so the scalar two-weight
boundedness theorem in \cite{HolmesLaceyWick} gives
$B_{ab}\in\BMO_{\nu_{ab}}$.  It is also compact.  Hence
\cite[Theorem~1.1]{LaceyLi} gives
$$
 B_{ab}\in\VMO_{\nu_{ab}}.
$$
Apply \cite[Theorem~1.3]{LLWW} with Schatten exponent $d$, source
weight $\mu_b$, target weight $\lambda_a$ and  Bloom weight
$\nu_{ab}=(\mu_b/\lambda_a)^{1/2}$.  Since
$B_{ab}\in\VMO_{\nu_{ab}}$, that theorem shows that $B_{ab}$ is
constant almost everywhere.  This holds for every pair $a,b$, so $B$ is
a constant matrix.  A constant matrix commutes with $R_j$, which proves
the converse.
\end{proof}

\subsection{Proof of the main theorem}

\begin{proof}[Proof of Theorem~\ref{sch:thm:target}]
Let $d\geq2$.  If $d<r<\infty$, Theorem
\ref{sch:thm:Riesz-upper} gives
$$
 \|C_{B,R_j}\|_{\Sch^r(L^2(U),L^2(V))}
 \leq C\|B\|_{\mathfrak B^r_{V,U}},
$$
and Theorem~\ref{sch:thm:Riesz-lower}, with $q=r$, gives the reverse
inequality.  This proves \eqref{sch:eq:strong-target} and
\eqref{sch:eq:strong-norm-target}.

Theorem~\ref{sch:thm:matrix-rigidity} proves
\eqref{sch:eq:subcritical-target}.  No upper estimate is needed in this
range because the commutator of a constant matrix is zero.

At the weak critical index, Theorem~\ref{sch:thm:Riesz-upper} gives
$$
 \|C_{B,R_j}\|_{\Sch^{d,\infty}(L^2(U),L^2(V))}
 \leq C\|B\|_{\mathfrak W^d_{V,U}}.
$$
Theorem~\ref{sch:thm:Riesz-lower}, with $(r,q)=(d,\infty)$, gives the
reverse inequality.  This proves \eqref{sch:eq:weak-target} and
\eqref{sch:eq:weak-norm-target}.
\end{proof}

\section{Sobolev estimates at the critical index}

Throughout this section, let $d\geq2$, let $U,V\in\mathcal A_2$
and fix $1\leq j\leq d$.  We prove a first-order Sobolev estimate
for $C^1$ symbols in the weak Schatten class
$\Sch^{d,\infty}(L^2(U),L^2(V))$.  For constant or uniformly
elliptic matrix weights, we obtain an equivalent Sobolev
characterization with distributional derivatives.

For $B\in W^{1,1}_{\mathrm{loc}}(\R^d;\C^{m\times m})$, define
\begin{equation}
 G_{V,U}B(x)
 =
 \bigg(
 \sum_{\kappa=1}^d
 \|V^{1/2}(x)(\partial_\kappa B)(x)U^{-1/2}(x)\|_{\HS}^2
 \bigg)^{1/2}.
 \label{sch:eq:matrix-Sobolev-density}
\end{equation}
The homogeneous Sobolev space $\dot W^{1,d}_{V,U}$ consists of
these symbols for which $G_{V,U}B\in L^d(\R^d)$, with seminorm
$$
 \|B\|_{\dot W^{1,d}_{V,U}}
 =\|G_{V,U}B\|_{L^d(\R^d)}.
$$
The derivatives act on $B$, and the weights occur as coefficients
in the seminorm.  We regard this space modulo constant matrices.

\begin{proposition}
\label{sch:prop:Sobolev-necessity}
If $B\in C^1(\R^d;\C^{m\times m})$, then
\begin{equation}
 \|G_{V,U}B\|_{L^d}
 \leq C_d\|B\|_{\mathfrak W^d_{V,U}}.
 \label{sch:eq:Sobolev-from-oscillation}
\end{equation}
Consequently,
\begin{equation}
 \|G_{V,U}B\|_{L^d}
 \leq C
 \|C_{B,R_j}\|_{\Sch^{d,\infty}(L^2(U),L^2(V))}.
 \label{sch:eq:Sobolev-from-commutator}
\end{equation}
\end{proposition}

\begin{theorem}
\label{sch:thm:elliptic-Sobolev}
Let $B\in L^1_{\mathrm{loc}}(\R^d;\C^{m\times m})$.
For fixed positive definite matrices $U_0,V_0$, one has
\begin{equation}
 B\in\mathfrak W^d_{V_0,U_0}
 \quad\Longleftrightarrow\quad
 B\in\dot W^{1,d}_{V_0,U_0},
 \qquad
 \|B\|_{\mathfrak W^d_{V_0,U_0}}
 \simeq
 \|G_{V_0,U_0}B\|_{L^d}.
 \label{sch:eq:constant-weight-Sobolev}
\end{equation}
If $U,V$ satisfy
$$
 cI_m\leq U(x),V(x)\leq CI_m
 \quad\text{for almost every }x,
$$
where $0<c\leq C<\infty$, then
$$
 \mathfrak W^d_{V,U}=\dot W^{1,d}_{V,U},
 \qquad
 \|B\|_{\mathfrak W^d_{V,U}}
 \simeq \|G_{V,U}B\|_{L^d}.
$$
In either case,
\begin{equation}
 C_{B,R_j}\in\Sch^{d,\infty}(L^2(U),L^2(V))
 \quad\Longleftrightarrow\quad
 B\in\dot W^{1,d}_{V,U},
 \label{sch:eq:elliptic-Sobolev-commutator}
\end{equation}
and
$$
 \|C_{B,R_j}\|_{\Sch^{d,\infty}(L^2(U),L^2(V))}
 \simeq \|B\|_{\mathfrak W^d_{V,U}}
 \simeq \|G_{V,U}B\|_{L^d}.
$$
The constants in the uniformly elliptic case depend only on
$d,m,c$ and $C$.
\end{theorem}

\subsection{Proof of the Sobolev estimate}

We use the reducing matrices
$\mathcal U_Q=\langle U\rangle_Q^{1/2}$ and
$\mathcal V'_Q=\langle V^{-1}\rangle_Q^{1/2}$ fixed in Part~III.
The following limit identifies the derivative measured by the two
cube oscillations.

\begin{lemma}
\label{sch:lem:infinitesimal-oscillation}
Suppose that $B$ is differentiable at $x$ and that $x$ is a
common Lebesgue point of $U,V,U^{-1},V^{-1}$.
Let $Q_n$ be cubes containing $x$, with sides parallel to the
coordinate axes and $\ell(Q_n)\to0$.  Then
\begin{equation}
 \lim_{n\to\infty}\frac{\alpha_{Q_n}^{\HS}(B)}{\ell(Q_n)}
 =
 \lim_{n\to\infty}\frac{\beta_{Q_n}^{\HS}(B)}{\ell(Q_n)}
 =\frac1{\sqrt{12}}G_{V,U}B(x).
 \label{sch:eq:infinitesimal-oscillation}
\end{equation}
\end{lemma}

\begin{proof}
Write $\ell_n=\ell(Q_n)$ and let $c_n$ be the center of $Q_n$.
Differentiability at $x$ gives
$$
 B(y)-B_{Q_n}
 =L_n(y)+E_n(y),
 \qquad
 L_n(y)=\sum_{\kappa=1}^d
 (y_\kappa-c_{n,\kappa})\partial_\kappa B(x),
$$
where
$$
 \sup_{y\in Q_n}\|E_n(y)\|_{\HS}
 \leq\varepsilon_n\ell_n,
 \qquad \varepsilon_n\to0.
$$
The Lebesgue point property gives
$$
 \langle U\rangle_{Q_n}^{-1/2}\longrightarrow U(x)^{-1/2},
 \qquad
 \avg_{Q_n}\|V(y)-V(x)\|_{\op}\,dy\longrightarrow0.
$$
In particular,
\begin{align*}
 &\avg_{Q_n}
 \|V^{1/2}(y)E_n(y)\langle U\rangle_{Q_n}^{-1/2}\|_{\HS}^2\,dy\leq
 \varepsilon_n^2\ell_n^2
 \|\langle U\rangle_{Q_n}^{-1/2}\|_{\op}^2
 \avg_{Q_n}\operatorname{tr}V(y)\,dy
 =o(\ell_n^2).
\end{align*}
The weighted square average of $L_n$ is $O(\ell_n^2)$, so
Cauchy--Schwarz bounds the cross term by $o(\ell_n^2)$.

Since $\|L_n(y)\|_{\HS}\leq C_x\ell_n$ on $Q_n$,
replacing $V(y)$ by $V(x)$ in the square average of $L_n$
changes that average by at most
$$
 C_x\ell_n^2
 \|\langle U\rangle_{Q_n}^{-1/2}\|_{\op}^2
 \avg_{Q_n}\|V(y)-V(x)\|_{\op}\,dy
 =o(\ell_n^2).
$$
We may also replace $\langle U\rangle_{Q_n}^{-1/2}$ by
$U(x)^{-1/2}$ in the limit.  Changing variables
$y=c_n+\ell_n z$ and using
$$
 \int_{[-1/2,1/2]^d}z_\kappa z_\tau\,dz
 =\frac{\delta_{\kappa\tau}}{12},
$$
we obtain
$$
 \lim_{n\to\infty}
 \frac{\alpha_{Q_n}^{\HS}(B)^2}{\ell_n^2}
 =\frac1{12}\sum_{\kappa=1}^d
 \|V^{1/2}(x)(\partial_\kappa B)(x)U^{-1/2}(x)\|_{\HS}^2.
$$

For $\beta_{Q_n}^{\HS}(B)$, apply the preceding calculation
to $B^*$ with the pair of weights $(V^{-1},U^{-1})$ in place
of $(U,V)$.  These weights have the required Lebesgue point
properties at $x$, and hence
$$
 \lim_{n\to\infty}
 \frac{\beta_{Q_n}^{\HS}(B)^2}{\ell_n^2}
 =\frac1{12}\sum_{\kappa=1}^d
 \|U^{-1/2}(x)(\partial_\kappa B(x))^*
 V^{1/2}(x)\|_{\HS}^2.
$$
The Hilbert--Schmidt norm is invariant under taking adjoints.
Thus both limits equal $G_{V,U}B(x)^2/12$.
\end{proof}

\begin{lemma}
\label{sch:lem:weak-cube-counting}
Let $\D$ be a dyadic system, let $c_Q\geq0$ for $Q\in\D$,
and let $g:\R^d\to[0,\infty)$ be measurable.
For $x$ outside the dyadic boundaries, let $Q_k(x)$ be the
cube of side length $2^{-k}$ containing $x$.  If
$$
 g(x)\leq\liminf_{k\to\infty}2^kc_{Q_k(x)}
 \quad\text{for almost every }x,
$$
then
\begin{equation}
 \|g\|_{L^d}
 \leq C_d\|\{c_Q\}_{Q\in\D}\|_{\ell^{d,\infty}}.
 \label{sch:eq:weak-cube-counting}
\end{equation}
\end{lemma}

\begin{proof}
It suffices to consider
$A=\|\{c_Q\}\|_{\ell^{d,\infty}}<\infty$.  Then
\begin{equation}
 \#\{Q:c_Q>t\}\leq(A/t)^d,\qquad t>0.
 \label{sch:eq:weak-sequence-distribution}
\end{equation}
Fix a bounded cube $H$ and an integer $N\geq1$.  For
$-N\leq n\leq N$, put
$$
 E_n=H\cap\{x:2^n<g(x)\leq2^{n+1}\}.
$$
For almost every $x\in E_n$, the assumed lower limit gives
$$
 c_{Q_{K+n}(x)}>2^{-K-1}
$$
for all sufficiently large $K$.  By dominated convergence,
we may choose one $K$ such that
$$
 E'_{n,K}
 =\{x\in E_n:c_{Q_{K+n}(x)}>2^{-K-1}\},
 \qquad |E'_{n,K}|\geq\tfrac12|E_n|,
$$
for every retained $n$.

At level $K+n$, the cubes meeting $E'_{n,K}$ have total
volume at least $|E'_{n,K}|$.  Their number is therefore at least
$$
 |E'_{n,K}|\,2^{d(K+n)}.
$$
All of these cubes satisfy $c_Q>2^{-K-1}$.
Cubes from different values of $n$ have different side lengths,
so \eqref{sch:eq:weak-sequence-distribution} gives
$$
 \frac12\sum_{n=-N}^N|E_n|2^{d(K+n)}
 \leq \#\{Q:c_Q>2^{-K-1}\}
 \leq A^d2^{d(K+1)}.
$$
After cancellation,
$$
 \sum_{n=-N}^N2^{dn}|E_n|\leq C_dA^d.
$$
Let $N\to\infty$ and then let $H$ increase to $\R^d$.
Since $g^d\leq2^{d(n+1)}$ on $\{2^n<g\leq2^{n+1}\}$,
monotone convergence proves \eqref{sch:eq:weak-cube-counting}.
\end{proof}

\begin{proof}[Proof of Proposition~\ref{sch:prop:Sobolev-necessity}]
Fix one of the adjacent dyadic systems and put
$c_Q=\omega_Q^{\HS}(B)$.
Lemma~\ref{sch:lem:infinitesimal-oscillation} gives
$$
 \lim_{k\to\infty}2^kc_{Q_k(x)}
 =\frac1{\sqrt{12}}G_{V,U}B(x)
 \quad\text{for almost every }x.
$$
Apply Lemma~\ref{sch:lem:weak-cube-counting} to
$g=G_{V,U}B/\sqrt{12}$.  We obtain
$$
 \|G_{V,U}B\|_{L^d}
 \leq C_d\|\{\omega_Q^{\HS}(B):Q\in\D\}\|_{\ell^{d,\infty}}
 \leq C_d\|B\|_{\mathfrak W^d_{V,U}}.
$$
This proves \eqref{sch:eq:Sobolev-from-oscillation}.
The lower estimate \eqref{sch:eq:Riesz-lower} gives
\eqref{sch:eq:Sobolev-from-commutator}.
\end{proof}

\subsection{Proof of the Sobolev characterization}

We first record the unweighted cube characterization.
Its converse follows from Frank's mean-oscillation theorem
\cite[Theorem~1]{Frank}.

\begin{lemma}
\label{sch:lem:unweighted-critical-cube}
Let $N<\infty$ and
$F\in L^1_{\mathrm{loc}}(\R^d;\C^N)$.  For each cube $Q$, put
$$
 \operatorname{osc}_2(F;Q)
 =\bigg(\avg_Q|F-F_Q|^2\bigg)^{1/2},
$$
where the integral takes values in $[0,\infty]$.
Then $F\in\dot W^{1,d}(\R^d;\C^N)$ if and only if the
oscillation sequences have finite weak $\ell^d$ quasi-norms.
In this case,
\begin{equation}
 \sum_{t=1}^{3^d}
 \big\|\{\operatorname{osc}_2(F;Q):Q\in\D^t\}\big\|_{\ell^{d,\infty}}
 \simeq_{d,N}
 \|\nabla F\|_{L^d(\R^d)}.
 \label{sch:eq:unweighted-critical-cube}
\end{equation}
The derivatives are distributional, and the seminorms are
understood modulo constant vectors.
\end{lemma}

\begin{proof}
Suppose first that $F\in\dot W^{1,d}(\R^d;\C^N)$.
Put $g=|\nabla F|$ and choose
$$
 s=\frac{2d}{d+2}\quad\text{if }d>2,
 \qquad s=1\quad\text{if }d=2.
$$
The Sobolev--Poincar\'e inequality gives
\begin{equation}
 \operatorname{osc}_2(F;Q)
 \leq C_d\ell(Q)
 \bigg(\avg_Qg^s\bigg)^{1/s}.
 \label{sch:eq:critical-Sobolev-Poincare}
\end{equation}
Here $|\nabla F|$ is the Euclidean norm of all first
derivatives of the components of $F$.

Fix a dyadic system $\D$, and let $\D_k$ consist of its
cubes of side length $2^{-k}$.  If $Q\in\D_k$ and
$\operatorname{osc}_2(F;Q)>\lambda$, then
$$
 Q\subset\{x:M(g^s)(x)>c_d\lambda^s2^{ks}\},
$$
where $M$ is the Hardy--Littlewood maximal operator.
Since the cubes in $\D_k$ are disjoint,
\begin{align*}
 \#\{Q\in\D:\operatorname{osc}_2(F;Q)>\lambda\}
 &\leq\sum_{k\in\mathbb Z}2^{kd}
 \big|\{M(g^s)>c_d\lambda^s2^{ks}\}\big|
 \leq C_d\lambda^{-d}
 \int_{\R^d}M(g^s)^{d/s}.
\end{align*}
The last inequality follows by integrating
$$
 \sum_{k\in\mathbb Z}2^{kd}
 \one_{\{h>c_d\lambda^s2^{ks}\}}
 \leq C_d\lambda^{-d}h^{d/s}.
$$
Since $d/s>1$, the maximal theorem yields
$$
 \#\{Q\in\D:\operatorname{osc}_2(F;Q)>\lambda\}
 \leq C_d\lambda^{-d}\|\nabla F\|_{L^d}^d.
$$
Taking the weak $\ell^d$ quasi-norm and summing over the
adjacent systems proves one inequality in
\eqref{sch:eq:unweighted-critical-cube}.

For the converse, denote the left side of
\eqref{sch:eq:unweighted-critical-cube} by $A$ and suppose that
$A<\infty$.  For a ball $D(a,r)$, set
$$
 m_F(a,r)=\avg_{D(a,r)}|F-F_{D(a,r)}|.
$$
The adjacent-grid lemma gives $t=t(a,r)$ and
$Q=Q(a,r)\in\D^t$ such that
$$
 D(a,r)\subset Q,
 \qquad r\leq\ell(Q)\leq C_dr.
$$
Changing the average and applying Cauchy--Schwarz, we obtain
\begin{equation}
 m_F(a,r)
 \leq C_d\avg_Q|F-F_Q|
 \leq C_d\operatorname{osc}_2(F;Q).
 \label{sch:eq:ball-by-dyadic-oscillation}
\end{equation}
Choose $Q(a,r)$ by a fixed ordering of the adjacent systems
and their admissible cubes.  All pairs $(a,r)$ assigned to a
fixed cube $Q$ satisfy
$$
 a\in Q,\qquad c_d\ell(Q)\leq r\leq C_d\ell(Q).
$$
Consequently,
\begin{equation}
 \iint_{\{(a,r):Q(a,r)=Q\}}
 \frac{da\,dr}{r^{d+1}}
 \leq |Q|\int_{c_d\ell(Q)}^{C_d\ell(Q)}
 \frac{dr}{r^{d+1}}
 \leq C_d.
 \label{sch:eq:parameter-cell-bound}
\end{equation}
Let $A_t$ be the weak $\ell^d$ quasi-norm of the
oscillation sequence in $\D^t$.  Equations
\eqref{sch:eq:ball-by-dyadic-oscillation} and
\eqref{sch:eq:parameter-cell-bound} give
\begin{align*}
 \iint_{\{m_F(a,r)>\lambda\}}\frac{da\,dr}{r^{d+1}}
 &\leq C_d\sum_{t=1}^{3^d}
 \#\{Q\in\D^t:\operatorname{osc}_2(F;Q)>c_d\lambda\}
 \leq C_d\lambda^{-d}\sum_{t=1}^{3^d}A_t^d
 \leq C_d(A/\lambda)^d.
\end{align*}
Thus
\begin{equation}
 \sup_{\lambda>0}\lambda^d
 \iint_{\{m_F(a,r)>\lambda\}}\frac{da\,dr}{r^{d+1}}
 \leq C_dA^d.
 \label{sch:eq:continuous-weak-oscillation}
\end{equation}

By \cite[Theorem~1]{Frank}, a scalar locally integrable
function $f$ with finite weak mean oscillation belongs to
$\dot W^{1,d}(\R^d)$ and satisfies
$$
 \|\nabla f\|_{L^d}^d
 \leq C_d\sup_{\lambda>0}\lambda^d
 \iint_{\{m_f(a,r)>\lambda\}}\frac{da\,dr}{r^{d+1}}.
$$
For each component $F_\ell$,
$$
 m_{\operatorname{Re}F_\ell}(a,r)
 \leq m_F(a,r),
 \qquad
 m_{\operatorname{Im}F_\ell}(a,r)
 \leq m_F(a,r).
$$
Apply Frank's estimate to these real-valued functions and use
the equivalence of norms in $\C^N$.  It follows that
$$
 F\in\dot W^{1,d}(\R^d;\C^N),
 \qquad
 \|\nabla F\|_{L^d}\leq C_{d,N}A.
$$
This proves the converse.
\end{proof}

\begin{proof}[Proof of Theorem~\ref{sch:thm:elliptic-Sobolev}]
For constant weights, put
$$
 D=V_0^{1/2}BU_0^{-1/2}.
$$
The definitions give
$$
 \alpha_Q^{\HS}(B)=\beta_Q^{\HS}(B)
 =\bigg(\avg_Q\|D-D_Q\|_{\HS}^2\bigg)^{1/2}.
$$
Moreover, the distributional derivatives satisfy
$$
 \partial_\kappa D
 =V_0^{1/2}(\partial_\kappa B)U_0^{-1/2},
 \qquad
 G_{V_0,U_0}B=|\nabla D|_{\HS}.
$$
Apply Lemma~\ref{sch:lem:unweighted-critical-cube} to $D$,
regarded as a vector in $\C^{m^2}$.  This proves
\eqref{sch:eq:constant-weight-Sobolev}.

Now suppose that $cI_m\leq U,V\leq CI_m$ almost everywhere.
The same bounds hold for their averages, and the inverse
matrices satisfy
$$
 C^{-1}I_m\leq U^{-1},V^{-1}\leq c^{-1}I_m.
$$
For every cube $Q$, it follows that
$$
 \sqrt{\frac cC}\operatorname{osc}_2(B;Q)
 \leq\alpha_Q^{\HS}(B),\,\beta_Q^{\HS}(B)
 \leq\sqrt{\frac Cc}\operatorname{osc}_2(B;Q),
$$
where the oscillation uses the Hilbert--Schmidt norm.
For symbols with weak first derivatives, we also have
$$
 \sqrt{\frac cC}|\nabla B|_{\HS}
 \leq G_{V,U}B
 \leq\sqrt{\frac Cc}|\nabla B|_{\HS}.
$$
Lemma~\ref{sch:lem:unweighted-critical-cube} therefore gives
$$
 \mathfrak W^d_{V,U}
 =\dot W^{1,d}(\R^d;\C^{m\times m})
 =\dot W^{1,d}_{V,U},
$$
with equivalent seminorms.  Finally, Theorem~\ref{sch:thm:target}
identifies $\mathfrak W^d_{V,U}$ with weak Schatten
membership and proves
\eqref{sch:eq:elliptic-Sobolev-commutator}.
\end{proof}

\noindent{\bf AI Disclosure:}
ChatGPT were used for language editing and proofreading. All mathematical results, arguments, proofs, and conclusions were developed, verified, and approved by the authors.
\medskip

\noindent {\bf Acknowledgement:}J. Li and B. D. Wick are supported by ARC DP 220100285. C.-W. Liang and C.-Y.
Shen were supported  by NSTC through grant 111-2115-M-002-010-MY5. B. D. Wick is partially supported by NSF-DMS \# 2000510, \# 2054863, and \# 1800057.

(J. Li) Department of Mathematics, Macquarie University, NSW, 2109, Australia.\\ 
{\it E-mail}: \texttt{ji.li@mq.edu.au}\\

(C.-W. Liang) Department of Mathematics, National Taiwan University, Taiwan.
\\ 
{\it E-mail}: \texttt{d10221001@ntu.edu.tw}\\

(C.-Y. Shen) Department of Mathematics, National Taiwan University, Taiwan. \\ {\it E-mail}: \texttt{cyshen@math.ntu.edu.tw}\\

(B.D. Wick) Department of Mathematics
Washington University - St. Louis
. \\ {\it E-mail}: \texttt{wick@math.wustl.edu}
\end{document}